\documentclass[11pt,a4paper,reqno]{amsart}

\usepackage[applemac]{inputenc}
\usepackage[T1]{fontenc}
\usepackage{amsmath}
\usepackage{amsthm}
\usepackage{amsfonts}
\usepackage{amssymb}
\usepackage{graphicx}
\usepackage{amsbsy}
\usepackage{mathrsfs}
\usepackage{bbm}
 \usepackage{relsize}
\usepackage{mathtools}
\usepackage{xcolor}
\newtheorem{theorem}{Theorem}[section]
\newtheorem{lemma}[theorem]{Lemma}

\newtheorem{proposition}[theorem]{Proposition}
\newtheorem{corollary}[theorem]{Corollary}
\newtheorem{definition}[theorem]{Definition}
\theoremstyle{remark}
\newtheorem{remark}[theorem]{\it \bf{Remark}\/}

\numberwithin{equation}{section}
\catcode`@=11
\def\section{\@startsection{section}{1}%
  \z@{1.5\linespacing\@plus\linespacing}{.5\linespacing}%
  {\normalfont\bfseries\large\centering}}
\catcode`@=12
\newcommand{\be}{\begin{equation}}
\newcommand{\ee}{\end{equation}}
\newcommand{\bea}{\begin{eqnarray}}
\newcommand{\eea}{\end{eqnarray}}
\newcommand{\bee}{\begin{eqnarray*}}
\newcommand{\eee}{\end{eqnarray*}}

\def\RR{\mathbb{R}}

\def\fref#1{{\rm (\ref{#1})}}

\catcode`@=11
\def\supess{\mathop{\operator@font Sup\,ess}}
\catcode`@=12

\def\RR{\mathbb{R}}

\def\bar#1{{\overline #1}}
\def\fref#1{{\rm (\ref{#1})}}

\def\R2+{\RR ^2_+}

\def\lim{\mathop{\rm lim}}

\def\log{{\rm log}}

\def\Id{\rm{Id}}

\def\ba{\begin{array}}
\def\ea{\end{array}}

\def\ep{\epsilon}

\author[C. Collot]{Charles Collot}
\address{CNRS, and Laboratoire Analyse G\'eometrie Mod\'elisation, CY Cergy Paris Universit\'e, 33 Boulevard du Port, 95000 Cergy, France}
\email{ccollot@cyu.fr}
\author[P. Germain]{Pierre Germain}
\address{Courant Institute of Mathematical Sciences, New York University, 251 Mercer Street, New York, NY 10003, United States of America.}
\email{pgermain@cims.nyu.edu}

\title{Derivation of the homogeneous kinetic wave equation: longer time scales}

\begin{document}

\maketitle

\begin{abstract}
We consider the nonlinear Schr\"odinger equation set on a flat torus, in the regime which is conjectured to lead to the kinetic wave equation; in particular, the data are random, and spread up to high frequency in a weakly nonlinear regime. We pursue the investigations of our previous paper, and show that, in the case where the torus is the standard one, only the scaling considered there allows convergence of the Dyson series up to the kinetic time scale. We also show that, for generic quadratic dispersion relations (non rectangular tori), the Dyson series converges on significantly longer time scales; we are able to reach the kinetic time up to an arbitrarily small polynomial error for a larger set of scalings. These results show the importance of the exact structure of the dispersion relation, more specifically of equidistribution properties of some bilinear quantities akin to pair correlations derived from it.
\end{abstract}

\tableofcontents

\section{Introduction}

\subsection{From the nonlinear Schr\"odinger equation to the kinetic wave equation} 

We want to study the kinetic limit in the weakly turbulent regime for the nonlinear Schr\"odinger equation\footnote{It would be indifferent for the remainder of this paper to flip the sign in front of the nonlinearity, turning the equation from focusing to defocusing}
\begin{equation} \tag{NLS} \label{id:NLSgeneric}
\left\{ 
\begin{array}{l}
i \partial_t u - \Delta_H u = \lambda^2 |u|^2 u, \\
u(t=0) = u_0,
\end{array}
\right. \qquad x\in \mathbb T^d = \mathbb{R}^d / (2\pi \mathbb{Z}^d).
\end{equation}
The Hamiltonian of the equation is, for $\lambda>0$:
$$
\frac 12 \sum_{k\in \mathbb Z^d} |k|_H^2 |\widehat u(k)|^2-\frac{\lambda^2}{4} \int_{x\in \mathbb T^d} |u(x)|^4 \, dx,
$$
and we shall consider both the canonical Laplacian $|k|_{\Id}^2=|k|^2$, as well as general linear Hamiltonians with quadratic dispersion relation:
\be \label{def:H}
\widehat{-\Delta_H u}(k)=|k|_H^2 \hat u(k),\qquad |k|_H^2=\sum_{i,j=1}^d h_{i,j}k_ik_j=Hk\cdot k,
\ee
where the matrix $H=(h_{ij})_{1\leq i,j\leq d}\in \mathbb R^{d\times d}$ is symmetric positive definite. Weak turbulence is a general framework to study linear Hamiltonians perturbed by a weak nonlinearity; we will see that "generic" dispersion relations are better behaved compared to the standard case $H = \operatorname{Id}$. 

These more general dispersion relations can also arise from natural geometric considerations: namely, if one considers the nonlinear Schr\"odinger equation with standard Laplacian
\be  \label{id:NLS}
\left\{ 
\begin{array}{l}
i \partial_t u - \Delta u = \lambda^2 |u|^2 u, \\
u(t=0) = u_0,
\end{array}
\right.
\ee
set on the quotient of $\mathbb{R}^d$ by a lattice. Mapping the equation to the standard torus, it is equivalent to consider \fref{id:NLSgeneric} set on the standard torus $\mathbb{T}^d$ where $H = \operatorname{Id}$ for the standard torus, $H$ is diagonal for a rectangular torus, and $H$ is a general symmetric matrix for a general torus.\\

\noindent The initial data is chosen to be a Gaussian field of the form
\begin{equation}
\label{data}
u_0(x,\omega) = \frac{\epsilon^{d/2}}{(2\pi)^{d/2}} \sum_{k \in \mathbb{Z}^d} A(\epsilon k) G_k(\omega)e^{ik\cdot x},
\end{equation}
where $A \in \mathcal{C}^\infty_0(\mathbb{R}^d,[0,\infty))$ and $(G_k)_{k \in \mathbb{Z}^d}$ are independent standard centred complex Gaussians defined on a probability space $\Omega$ - notice that this normalization implies that, on average $\| u_0 \|_{L^2(\mathbb T^d)} \sim 1$, and, by Khinchine's inequality, $\| u_0 \|_{L^p(\mathbb T^d)} \sim_p 1$, for any $p < \infty$.

Heuristic derivations show that, as $\epsilon \lambda \to 0$ (weakly nonlinear regime) and $\epsilon \to 0$ (high frequency limit), scaling properly the expectation (denoted $\mathbb{E}$) of the square modulus of the Fourier coefficients of $u$, it satisfies
\begin{equation}
\label{quetzalresplendissant}
\epsilon^{-d} \mathbb{E} \left| \widehat{u}( \lfloor \epsilon k \rfloor) \left( T_{kin}t \right) \right|^2 \longrightarrow \rho(t,k), \qquad \qquad \mbox{for }T_{kin} = \frac{1}{\epsilon^2 \lambda^4}
\end{equation}
where $\rho$ solves the kinetic wave equation
\begin{equation}
\tag{KWE} \label{KWE}
\left\{
\begin{array}{l}
\partial_t \rho(t,k) = c_0 \mathcal{C}[\rho](k) ,\\
\rho(0,k) = |F(k)|^2,
\end{array}
\right.
\end{equation}
with the collision operator given by
\begin{equation}
\begin{split}
\label{collisionoperator}
\mathcal{C}[\rho](k) =  & \int_{(\mathbb{R}^d)^3} \delta(k+\ell-m-n) \delta(|k|_H^2 + |\ell|_H^2 - |m|_H^2 - |n|_H^2)\\
& \qquad \qquad \qquad \qquad  \rho(k) \rho(\ell) \rho(m) \rho(n) \left[ \frac{1}{\rho(k)} + \frac{1}{\rho(\ell)} - \frac{1}{\rho(m)} - \frac{1}{\rho(n)} \right] \,d\ell \,dm \,dn.
\end{split}
\end{equation}
and
\be \label{def:c0}
c_0 = 2^{2-2d} \pi^{1-2d}.
\ee

\subsection{Background} We survey briefly relevant references here, and refer to~\cite{CG} for a more detailed discussion.

\subsubsection{Heuristic derivation of~\eqref{KWE}} The kinetic wave equation appeared first in work by Peierls~\cite{Peierls}, and then Hasselmann~\cite{Hasselmann1,Hasselmann2}, before becoming a focus of the school around Zakharov~\cite{ZLF}. More modern introductory texts are~\cite{Nazarenko,NR}. 

\subsubsection{Rigorous derivation of~\eqref{KWE}} As far as rigorous mathematics go, a fundamental work is due to Lukkarinen and Spohn~\cite{LS1} in the framewok of stationary statistical mechanics, see also~\cite{Faou,Spohn2,LS2}.  The first rigorous (partial) derivation for data out of statistical equilibrium is due to Buckmaster, Germain, Hani and Shatah~\cite{BGHS}. Derivation on a time scale arbitrarily close to the kinetic time $T_{kin}$ was established in the previous article by the authors~\cite{CG}, and independently, by Deng and Hani~\cite{DengHani}. Both these references could only reach the kinetic time up to an arbitrarily polynomial loss in the case $T_{kin} \sim 1$, and in the present paper we are able to reach a larger regime of time scales $T_{kin} \gg 1$.

Deng and Hani~\cite{DengHani} and Deng, Nahmod and Yue~\cite{DengNahmodYue2}, proved long time existence results of solutions to \fref{id:NLSgeneric} with data \fref{data} in the case $H$ diagonal in the regimes. In the regime $T_{kin}\ll 1$ we obtain a lower bound on the Dyson series that diverges at a time larger than theirs and smaller than $T_{kin}$, supporting ill-posedness below the probabilistic scaling (see \cite{DengNahmodYue1}). In the regime $T_{kin}\gg 1$, we prove an upper bound that diverges right after the time they reach.

Another line of research focuses on the derivation of~\eqref{KWE} from~\eqref{id:NLS} to which random forcing and dissipation are added, see \cite{ZL,DK1,DK2,DK3}.

\subsubsection{Derivation of related collisional kinetic models} The Boltzmann equation is perhaps the most famous collisional kinetic model; its derivation has been mathematically elucidated~\cite{Lanford,GST}. The derivation of its quantum counterpart remains mostly open, see however~\cite{BCEP1,BCEP2, BCEP3}; note that this question is closely related to the derivation of the kinetic wave equation.

Another strand of research addresses linear dispersive models with random potential, from which one can derive the linear Boltzmann equation on a short time scale~\cite{Spohn1}, and the heat equation on a longer time scale~\cite{ESY1,ESY2}. 

Finally, let us mention the possibility of deriving Hamiltonian models for the nonlinear Schr\"odinger equation with deterministic data in the infinite volume, or big box, limit~\cite{FGH,BGHS0}.

\subsubsection{Nonlinear dispersive equations with random data} Considering~\eqref{id:NLS} with data~\eqref{data} is an instance of nonlinear dispersive equations with random data, which is an active research field, see for instance \cite{BT1,BT2,CO,NS}. However, the regime of interest in the present paper is different from most articles on the subject, which typically deals with Sobolev data, and where the nonlinearity is not necessarily weak with respect to the linear part of the equation, see however~\cite{dST}.

An important research direction deals with Gibbs measures and how they can be used to obtain a global control of the equation, see~\cite{Bourgain2,Bourgain3,DengNahmodYue1}.

\subsubsection{KAM theory for PDEs} This is another point of view on the dynamics of nonlinear dispersive equations on compact domains, or, more generally, of Hamiltonian systems without a dissipation, or relaxation mechanism, see~\cite{Grebert,Berti} for reviews. This approach yields stability results in Sobolev spaces (for instance), but it does not seem to apply to the regime considered in the present paper.

\subsection{The Dyson series} 
To study a solution to \fref{id:NLS}, we apply first Wick renormalisation for the phase: $u=e^{-it \lambda^2 \frac{2}{(2\pi)^d}\| u_0\|_{L^2}^2}v$. This is a usual transformation that cancels the leading order nonlinear effect, as $|u|^2u\approx 2 (2\pi)^{-d}\| u_0\|_{L^2}^2u$ in some appropriate sense, and which does not alter relevant statistics. The equation for $v$ is
$$
i\partial_t v=-\Delta v +\lambda^2 \left(|v|^2-\frac{2}{(2\pi)^d}\| v\|_{L^2}^2\right)v.
$$
Since linear dynamics acts at a shorter time scale than nonlinear effects (by the weak nonlinearity assumption), it is customary to find an approximate solution by iterating Duhamel's formula, giving the so-called Dyson series expansion. We define for this purpose the truncation operator
\be \label{def:P}
\mathcal{F} [P(a,b,c)](k) = \frac{1}{(2\pi)^d} \sum_{k_1+k_2+k_3=k}  \widehat a(k_1)\widehat b(k_2)\widehat c(k_3)  (1-\delta(k_1+k_2)-\delta(k_2+k_3)).
\ee
This gives the decomposition of the product
$$
abc =P(a,b,c) + \frac{1}{(2\pi)^d} \langle \overline{a},b\rangle c + \frac{1}{(2\pi)^d} a \langle \overline{b},c\rangle.
$$
A quasi-solution is then a series expansion
$$
u^{app} = e^{-it \lambda^2 \frac{2}{(2\pi)^d}\| u_0\|_{L^2}^2} v^{app} \qquad \mbox{with} \qquad v^{app} = \sum_{n=0}^N u^n.
$$
The terms are defined through the following iterative resolution scheme
\begin{equation}
\label{defun}
u^0 = e^{it\Delta} u_0 \qquad \mbox{and if $n \geq 1$,} 
 \left\{ \begin{array}{l l}\displaystyle i \partial_t u^n + \Delta u^n = \lambda^2 \sum_{ i+j+k=n-1}P (u^i,\overline{u^j},u^k), \\ u^n(0)=0, \end{array} \right..
\end{equation}
It is useful to give a formula for $u^n$ with terms that are encoded by Feynman interaction diagrams:
\be \label{id:expansionfeynman}
u^n= \sum_{G\in \mathcal G(n)}u_G,
\ee
see Section \ref{sectiongraph}. Above, $\mathcal G(n)$ is the set of interaction diagrams of depth $n$. Formal derivations of the kinetic wave equation (see for example \cite{Nazarenko,ZLF}) usually stop at the second iteration $u\approx u^0+u^1+u^2$.

\bigskip

\noindent Our aim in the present article is to investigate the convergence properties of the Dyson series. This series is always truncated, but in order for the kinetic limit to be justified, it seems necessary to have polynomial bounds making the radius of convergence be of order $T_{kin}$. Rough formal estimates lead to the guess that
\begin{equation}
\label{papillon}
u^n=O\left(\frac{t}{T_{kin}}\right)^{\frac{n}{2}}
\end{equation}
in an appropriate topology. This was proved in~\cite{CG, DengHani} for $t \ll T_{kin} \sim 1$, and our aim here is to investigate the regimes $T_{kin} \ll 1$ and $T_{kin} \gg 1$.

To understand the fundamental difference between these two regimes, consider the so-called four waves interaction between frequencies $\xi_1, \xi_2, \xi_3$ in~\eqref{id:NLSgeneric}, which generates frequency $\xi_4 = \xi_1 - \xi_2 + \xi_3$. This interaction can have a significant impact on the dynamics, on a time scale $T$, if the modulus of resonance
$$
 \Omega_H(\xi_1, \xi_2, \xi_3,\xi_4) = | \xi_1|^2_H - |\xi_2|^2_H + | \xi_3|^2_H - |\xi_4|^2_H 
$$
 is such that
$$
\left| \Omega_H(\xi_1, \xi_2, \xi_3,\xi_4) \right| \lesssim \frac{1}{T}.
$$
In the kinetic limit that we consider, it is thus clear that the distribution properties of $\Omega_H$ on a scale $\sim \frac{1}{T}$ will be a decisive factor. As a simple heuristic argument shows, any choice of $H$ is expected to give equidistribution if $T \ll 1$, but not for $T \gg 1$ - indeed, it is obvious $\Omega_{\operatorname{Id}}$ takes values in the integers, hence it cannot be equidistributed on a scale less than 1.

\subsection{Convergence for $T_{kin} \gg 1$, $H$ generic} \label{subsec:convergence}
Our first result is the existence of a solution with suitable estimates up to times arbitrarily close to the kinetic time, in the large kinetic time regime. It is obtained by proving first that the estimate~\eqref{papillon} holds.

\begin{theorem} \label{th:main} Consider for $d\geq 3$ the equation \eqref{id:NLSgeneric} with data~\eqref{data}, in the regime where 
\be \label{id:largekinetictime}
T_{kin} \geq \ep^{-\nu}
\ee
for any fixed $0<\nu \ll 1$. For almost all (in the Lebesgue sense) symmetric matrices $H\in \mathbb R^{d\times d}$ close to the identity matrix, the following holds true. 

For any $\kappa,\beta >0$, $s\geq 0$ and $N\in \mathbb N$, there exist $\epsilon^*,\mu>0$, such that, for any $0<\ep<\ep^*$, there exists an exceptional set of size $\epsilon^\mu$ over the complement of which, for $T = \ep^\beta \min( T_{kin}, \epsilon^{2-d})$:
\begin{itemize}
\item The iterates $u^n$ enjoy the bound for $0\leq n\leq N$ and any $m\geq 0$:
\be \label{bd:iteratesHgeneric1}
\| u^n \|_{C_t([0,T],H^m(\mathbb T^d))} \lesssim_{H,N,\kappa} \epsilon^{-m-\kappa} \left( \frac{T}{T_{kin}} \right)^{n/2},
\ee
where $H^m(\mathbb T^d)$ stands for the usual Sobolev space.
\item There exists a unique smooth solution $u$ on the time interval $[0,T]$, with:
\be \label{def:omega}
u=e^{i\lambda^2\omega(t)}\sum_{n=0}^N u^n+\tilde u,  \qquad \qquad \omega(t)=\frac{2}{(2\pi)^d}\| u_0\|_{L^2(\mathbb T^d)}^2t,
\ee
where the remainder satisfies
\be \label{bd:iteratesHgeneric2}
\| \tilde u \|_{C_t([0,T],H^s(\mathbb T^d))} \lesssim_{H,N,\kappa}  \epsilon^{-s-\kappa}\left( \frac{T}{T_{kin}} \right)^{(N+1)/2}.
\ee
\end{itemize}
\end{theorem}
We refer to Section~\ref{renard} for a proof of this theorem. It follows mostly the scheme laid out in~\cite{CG}, but new number theoretical results are needed. They have to do with the distribution properties of the functions
$$
H \xi_0 \cdot \xi, \qquad |\xi|_H^2, \qquad \mbox{and} \qquad H\xi \cdot \eta
$$
(which are functions on $\mathbb{R}^d$, $\mathbb{R}^d$, and $\mathbb{R}^{2d}$ respectively) over integers. More precisely, we prove in Section~\ref{sec:numbergeneric} that the estimates, for $L^{2-d} < \delta < 1$, 
\begin{align*}
& \# \{ \xi, \; |\xi| \leq L, \; |H\xi_0 \cdot \xi| < \delta \} \lesssim_{H,\kappa} L^{d-1} \sqrt{\delta} \\
& \# \{ \eta, \; |\eta| < L , \; ||\eta|_H^2 - a| < \delta \} \lesssim_{H,\kappa} L^{d-1} \sqrt{\delta} \\
& \# \left\{\eta,\xi \in \mathbb Z^d \mbox{ with } |\eta|,|\xi|\leq L, \quad | H\xi \cdot \eta- a|\leq \delta \right\} \lesssim_{H,\kappa} L^{2d-2} \delta
\end{align*}
hold for $1\leq |\xi_0| \leq L$, $a \in \mathbb{R}$, generically in $H$. The parameters $L$ and $\delta$ should be thought of as $\epsilon^{-1}$ and $T^{-1}$ respectively. The above estimates allow to control resonant interactions in the kinetic limit under consideration. They are proved using tools of analytic number theory: geometry of numbers and the circle method, combined with averaging over the matrix $H$.

\medskip

\noindent The above theorem provides the desired control of~\eqref{id:NLSgeneric} almost up to the kinetic time scale. The natural next step would be to prove that the dynamics, on this time scale, are approximated by the kinetic wave equation (KWE). Following the proof in~\cite{CG}, the only missing item is a quantitative equidistribution result for the function $H \xi \cdot \eta$, extending the result of Sarnak~\cite{Sarnak} to include error terms as in the result of Bourgain~\cite{Bourgain4} for generic diagonal forms.

\begin{remark} \label{rk:dengnahmodyue}

Deng, Nahmod and Yue \cite{DengNahmodYue2} were able to show that for $H$ diagonal, solutions to:
\be \label{id:NLS2}
\left\{ 
\begin{array}{l l}
i\partial_t u-\Delta_Hu=|u|^{p-1}u, \\ u_0(x)=\ep^{\alpha}\sum_{k\in \mathbb Z^d} e^{ik.x}A(\ep k)g_k(\omega),
\end{array}
\right.
\qquad \alpha=s+\frac d2
\ee
exist with large probability up to time $T\sim \ep^{\kappa}\ep^{-(p-1)(s-s_{pr})}$ for $s>s_{pr}$, where $s_{pr}=-(p-1)^{-1}$ is the probabilistic scaling exponent. Renormalising the solution to map it to our problem \fref{id:NLSgeneric}, Theorem \ref{id:largekinetictime} gives existence up to the larger time $T\sim \ep^{\kappa}\ep^{-2(p-1)(s-s_{pr})}$ for $H$ generic and $p=3$, provided $-\frac 12 <s<\frac d4-1$. This answers positively the conjecture they made in Remark 1.8 of \cite{DengNahmodYue2}.

\end{remark}

\subsection{Failure of convergence up to the kinetic time scale} \label{subsec:failure}

\subsubsection{The case $T_{kin} \gg 1$}
The estimate~\eqref{papillon} fails if $T_{kin} \gg 1$ and $t\gg 1$ for the Laplacian (dispersion relation $H = \operatorname{Id}$). We will, for simplicity, assume that $A$ takes the form
\be \label{id:assumptionA}
A(k)=\chi(|k|) \ \mbox{is a smooth non-negative cut-off: } \chi(r)=1 \mbox{ for } r\leq 1 \mbox{ and } \chi(r)=0 \mbox{ for } r\geq 2.
\ee

\begin{proposition} \label{pr:tgeq1}
Assume $d\geq 2$, $H = \operatorname{Id}$, and~\eqref{id:assumptionA}. For all $n\geq 1$, three constants $0<c(n)<C(n)$ and $C'(n)>0$ exist such that for any interaction diagram $G$ of depth $n$, for all $t\geq \langle \log \epsilon \rangle^{C'(n)}$ there holds for $\epsilon$ small enough for $d\geq 3$:
$$
c(n)\left(\frac{t^2}{T_{kin}}\right)^n \leq \ \mathbb E \| u^n_{G} \|_{L^2}^2\  \leq \ C(n)\left(\frac{t^2}{T_{kin}}\right)^n.
$$
For $d=2$ the same result holds with the estimate:
$$
c(n)\left(\frac{t^2}{T_{kin}}\right)^n\langle \log \epsilon \rangle^{n} \ \leq \ \mathbb E \| u^n_{G} \|_{L^2}^2\  \leq \ C(n)\left(\frac{t^2}{T_{kin}}\right)^n\langle \log \epsilon \rangle^{n}.
$$

\end{proposition}

This proposition is proved in Section~\ref{sectionlarge}.
The basic mechanism behind the above statement a failure of equidistribution for the function $(\xi, \eta) \mapsto \xi \cdot \eta$. Indeed, this function obviously takes integer values, making it not equidistributed on scales $\ll 1$. Quasi-resonances contribute to lower order compared to exact resonances, which is a mechanism for the failure of kinetic approximation already noticed in the physics literature, see the discussion in Section 6.5 of \cite{Nazarenko}.

This failure of equidistribution allows us to isolate specific pairings of each graph $G$ for which the lower bound is achieved. Then, all other pairings are shown to have either a compatible sign, or to give a smaller contribution.

Such a divergence from the kinetic equation scaling $(\frac{t^2}{T_{kin}})^{\frac n2}$ is also expected for diagonal dispersion relations $A$ (i.e. rectangular tori, with different bounds though). This could easily be investigated with the same techniques. The above proposition is proved for $A$ a cut-off solely for the sake of clarity and would adapt to other functions. Deng and Hani \cite{DengHani}, and Deng, Nahmod and Yue \cite{DengNahmodYue2} proved existence up to time $T\sim \sqrt{T_{kin}}$ for $T_{kin}\gg 1$. Thus, Proposition \ref{pr:tgeq1} suggests that energy transfer between Fourier modes indeed starts at time $\sqrt{T_{kin}}$ and that their result is optimal. 

\subsubsection{The case $T_{kin} \ll 1$} In the regime $t\ll 1$, we find a bound $(\frac{1}{T_{kin}})^{n/2}$ (up to a fixed factor), showing again the violation of the $(\frac{t}{T_{kin}})^{n/2}$ bound. This happens \emph{for any dispersion relation}.
 
\begin{proposition} \label{pr:tleq1}
Consider any symmetric positive definite $H$ for the dispersion relation \fref{def:H}. Assume $d\geq 2$, \fref{id:assumptionA}, and pick any $1-\frac{1}{2d+1}<c_2<c_1<1$. Then for all $n$ with $2n\geq d+2$, there exist a graph $G_n\in \mathcal G(n)$, two constants $0<c(n)<C(n)$ such that for $\epsilon$ small enough for all $t \in [\ep^{c_1},\ep^{c_2}]$:
$$
c(n)\left(\frac{1}{T_{kin}}\right)^nt \ep^{d-1} \ \leq \ \mathbb E \| u^n_{G_n} \|_{L^2}^2\  \leq \ C(n)\left(\frac{1}{T_{kin}}\right)^nt \ep^{d-1}.
$$
\end{proposition}

This proposition is proved in Section~\ref{sectionsmall}. The bound above is of a different nature than that of Proposition \ref{pr:tgeq1}. It holds true for a specific graph $G_n$, and wether it holds true for all graphs $G\in \mathcal G(n)$ is an open question. For this specific $G_n$, we identify a specific pairing giving the lower bound; we call it the belt pairing due to the shape of the graph. This pathological paired graph was already noticed in \cite{DengHani} where it was asked wether cancellations coming from other pairings could erase this term. We prove here a sharp bound for this term, and show that \emph{all} other pairings give a lower order contribution, establishing that no cancellation from other pairings occur. This analysis goes deeper into the structure of graphs, and is a first step towards the full classification of graphs which is needed to reach the kinetic time scale. Note that we did not try to reach optimality for the polynomial bounds for the time interval $t\in [\ep,\ep^{1-\frac{1}{2d+1}}]$ as they would require a much more refined analysis due to number theoretic issues. Again, the result would adapt to other $A$'s.

Deng and Hani \cite{DengHani} proved existence up to the nonlinear time $T\sim \lambda^{-2}$ in this regime. The above Proposition supports a pathological behaviour for the solution around or prior to time $T\sim\ep \gg \lambda^{-2}$. Remark that the case $T_{kin}\ll 1$ corresponds to a supercritical regularity $s<s_{pr}$ for the normalisation \fref{id:NLS2}. Hence this Proposition supports ill-posedness for \fref{id:NLS2} for Sobolev data at supercritical regularity, that is, for initial data $\sum_{k\in \mathbb Z^d} e^{ik.x}\frac{1}{\langle k \rangle^\alpha}g_k(\omega)$ with $\alpha=s+\frac d2$ and $s<s_{pr}$. It is a first result toward almost sure ill-posedness as conjectured in Remark 1.10 of \cite{DengNahmodYue2}.

\subsection{General dispersion relations}
We believe that the above results shed light on the expected validity of the kinetic wave equation for nonlinear dispersive equations of the type
$$
i\partial_t u - \tau(D) u = \lambda^2 |u|^2 u,
$$
where $\tau(D)$ is the Fourier multiplier with symbol $\tau(\xi)$, and where this equation is set on the torus, with data~\eqref{data}. In order to simplify the discussion, let us assume that $\tau(D)$ behaves quadratically, so that the scaling of the above agrees with that of~\eqref{id:NLSgeneric}. The natural expectation is that it behaves similarly to that model, suggesting the following open question.

\medskip

\noindent
\underline{Open question} Prove that
\begin{itemize}
\item if $T_{kin} \ll 1$, the kinetic wave equation is never obtained in the limit under consideration;
\item if $T_{kin} \gg 1$, the kinetic wave equation appears in the limit, provided $\tau$ is generic.
\end{itemize}

\medskip

One should be able to obtain corresponding statements for more general nonlinear dispersive PDEs set on the torus. 

How about more general domains? This question is fascinating, but it also seems out of reach for the moment.

\subsection{Acknowledgements} C. Collot is supported by the ERC-2014-CoG 646650 SingWave. P. Germain is supported by the NSF grant DMS-1501019, by the Simons collaborative grant on weak turbulence, and by the Center for Stability, Instability and Turbulence (NYUAD). Part of this work was done when C. Collot was working at New York University, and he thanks the Courant Institute.

The authors are grateful to Simon Meyerson and Jacek Jendrej for very helpful conversations about Section~\ref{sec:numbergeneric}.

\subsection{Notations}

Throughout the paper, $\kappa>0$ will denote a positive constant that can be made arbitrarily small, whose value will possibly change from one line to another. We choose for the space and space time Fourier transform the normalizations
\begin{align*}
& \widehat{f}(k) = \frac{1}{(2\pi)^{d/2}} \int_{\mathbb{T}^d} f(x) e^{-i k \cdot x}\,dx \\
& \widetilde{f}(\tau,k) = \frac{1}{(2\pi)^{\frac{d+1}{2}}} \int_{\mathbb{T}^d} f(x) e^{-i (\tau t + k \cdot x)}\,dx\,dt.
\end{align*}

\section{Graph analysis}

\label{sectiongraph}

We recall in this section how to represent graphically the expansion \fref{defun} for the solution. This graph analysis is also used to compute key quantities such as norms of the quasi-solution in various function spaces and the operator of the linearised dynamics close to it. It will be used throughout the paper. This analysis is classical, so we repeat without proof its adaptation to the present problem that was done in \cite{CG} and refer to this paper for details. The analysis in \cite{CG} borrowed heavily from \cite{LS2} and we refer to this paper for references.

\subsection{Dyson series, Feynman interaction diagrams}

We first describe the expansion \fref{defun}, explaining the following notation (used in \fref{id:expansionfeynman}):
$$
u^n=\sum_{G\in \mathcal G(n)}u_G, \quad \quad \quad u_G=\sum_{o\in \mathcal O(G)}u_{\ell}.
$$
Above, the first formula encodes iterations of Duhamel formula \fref{defun} via diagrams, and the second formula associates to a diagram all possible time ordering which will provide us with an analytical formula. A Feynman interaction diagram of depth $n$ is an oriented planar tree $G=\{\mathcal V,\mathcal E,p)$ where:
\begin{itemize}
\item $\mathcal V=\mathcal V_R\cup \mathcal V_i \cup \mathcal V_0$ is the collection of vertices. $\mathcal V_R=\{v_R\}$ contains the root vertex (representing $\widehat{u_G}(k_R)$). $\mathcal V_0\neq \emptyset$ contains the initial vertices (representing the initial datum $\widehat{u_0}(k_{v_0}$)). $\mathcal V_i$ contains the $n$ interaction vertices (representing an iteration of the nonlinearity).
\item $ \mathcal E\subset \mathcal V^2$ is the set of edges (representing a free evolution $e^{is\Delta}$), and satisfies:
\begin{itemize}
\item[(i)] There exists a unique $v\in \mathcal V_0\cup \mathcal V_i$ such that $(v,v_R)\in \mathcal E$. There are no $v'\in \mathcal V$ with $(v_R,v')\in \mathcal E$.
\item[(ii)] For all $v\in \mathcal V_i$, there exist four unique $(v_i)_{-1\leq i \leq 1}\in (\mathcal V_i\cup \mathcal V_0)^3$ and $v'\in \mathcal V_R \cup \mathcal V_i$ (three vertices below $v$ and one above) such that $(v_i,v)\in \mathcal E$ for $i=-1,0,1$ and $(v,v')\in \mathcal E$.
\item[(iii)] For all $v_0\in \mathcal V_0$, there exists a unique $v'\in \mathcal V_R \cup \mathcal V_i$.
\end{itemize}
\item $p:\mathcal V_i\times \mathcal V_0\rightarrow \{-1,0,1\}$ is the position (represents the position with respect to the vertex above, $-1/0/1$ for left/center/right). For (i) above, $p(v)=0$; for (ii) $p(v_{-1})=-1$, $p(v_{0})=0$ and $p(v_{1})=1$.
\end{itemize}
With this definition, one can check that the tree is entirely connected and that $p$ is uniquely determined. The set of interaction diagram of depth $n$ is denoted by $\mathcal G(n)$.

One problem with these graph theoretic notations is that they are not convenient for enumerations. So next, we order the time slices; given $G\in \mathcal G(n)$, we say $o:\mathcal V\rightarrow \{1,...,n\}$ is an admissible ordering if it is a bijection such that, keeping the notation used above in $(ii)$: $o(v)>o(v_i)$ for $i=1,2,3$. The set of admissible ordering is denoted by $\mathcal O(G)$.

Given $G\in \mathcal G(n)$ and $o\in \mathcal O(G)$, we can now give an analytical expression. Define the index set $I_n=\{1,2,...,n\}$. Then there exists a unique so called interaction history $\ell=\ell(G,o)\in I_{2n-1}\times I_{2n-3}\times ...\times I_1$ encoding uniquely a diagram $G(\ell)$ which represents $G$ with ordered time slices.

More precisely, following \cite{CG,LS2}, a graph with $n$ interaction has the total time $t$ divided into $n+1$ time slices of length $s_i$, $i=0,1,...,n$ whose index label the time ordering: from bottom to top in the graph. Associated with a time slice $i$ there are $1+2(n-i)$ "waves", with three of them merging into a single one for the next time slice. Each wave in each time slice is represented by an edge $e_{i,j}$ in the graph, with index set $ \mathcal I_{n}=\{(i,j), \ 0 \leq i \leq n,1\leq j\leq 1+2(n-i)\}$. Edges of the $i$-th time slice are related to edges of the $i+1$-th as follows: an edge with index $(i,k)$ for $k< \ell_{i+1}$ is matched with the edge with index $(i+1,k)$ above it; the three edges with indexes $ (i,\ell_{i+1})$, $(i,\ell_{i+1}+1)$ and $(i,\ell_{i+1}+2)$ merge through the vertex $v_{i+1}$ to form the edge with index $(i+1,\ell_i)$; and an edge with index $(i,k)$ for $k>\ell_{i+1}+2$ is matched with the edge with index $(i+1,k-2)$. Complex conjugation is encoded by the \textit{parity} $\sigma_{i,j}\in \{ \pm 1\}$ associated to each edge. Parity is defined recursively from top to bottom: $\sigma_{n,1}=+1$ or $\sigma_{n,1}=-1$ if the graph corresponds to $u^n$ or to $\bar u^n$. Then, parity is kept unchanged from an edge to the one below in absence of merging, and in case of a merging we require that $\sigma_{i,\ell_{i+1}}=-1$, $\sigma_{i,\ell_{i+1}+1}=\sigma_{i+1,\ell_i}$ and $\sigma_{i,\ell_{i+1}+2}=+1$. 

At the initial time slice $s_0$, below each of the edges of index $(0,j)$ for $1\leq j \leq 2n+1$ an \textit{initial vertex} $v_{0,j}$ is placed. At the final time slice slice $s_n$, a \textit{root vertex} $v_{R}$ is placed. Vertices which are neither initial vertices nor root vertices are called \textit{interaction vertices}.

The graph obtained this way is a tree, and edges are oriented from bottom to top. 
The natural ordering between vertices is: $v\leq v'$ if there is an oriented path from $v$ to $v'$. We define for each interaction vertex $v_i$ the set of its initial vertices below as $In(v_i)=\{ j\in (0,2n+1), \ v_{0,j} \leq v_i\}$.
An example of a diagram $G$ and of a possible ordered version of $G$ is given below:\\

\begin{center}
\includegraphics[width=14cm]{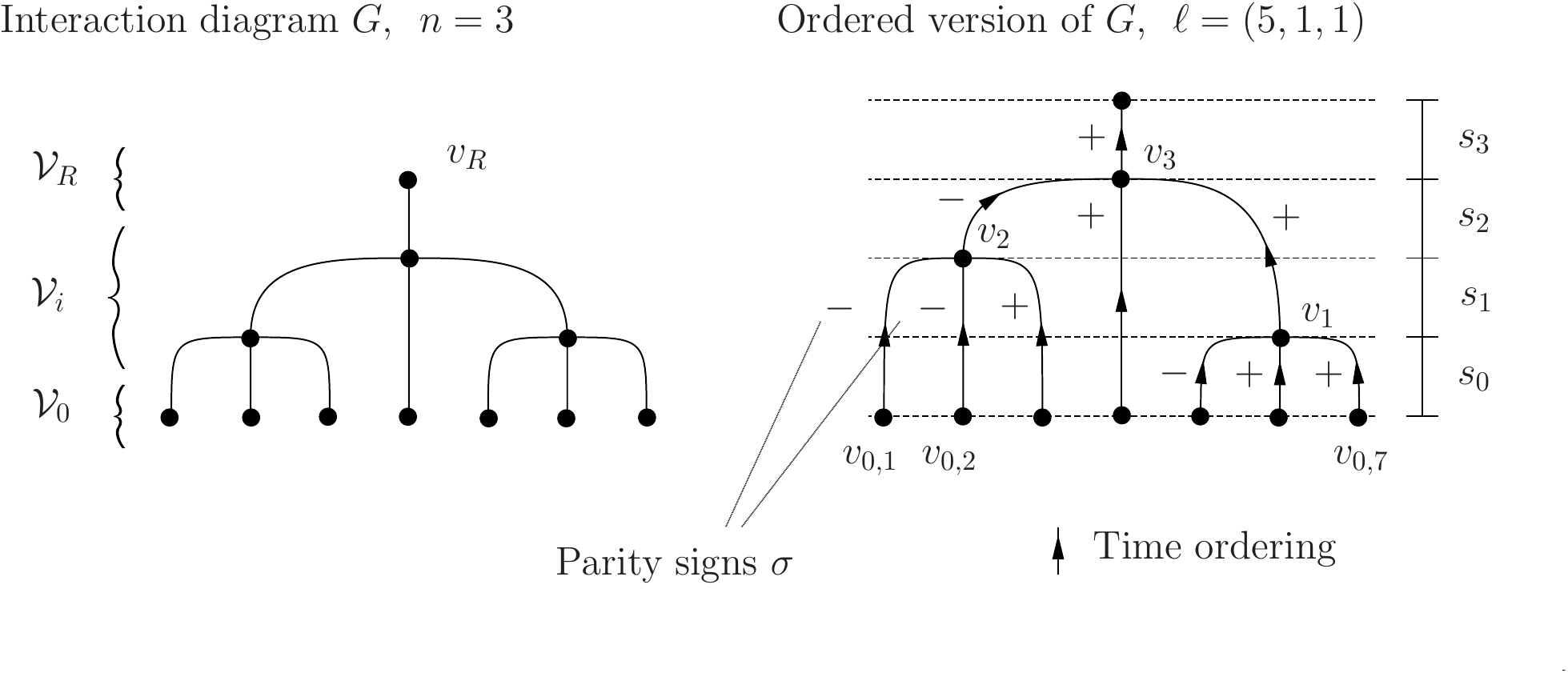}
\end{center}

We now associate to each edge $e_{i,j}$ in the extended graph a frequency $k_{i,j}$. At each vertex corresponds a $\delta$ function ensuring that the sum of the frequencies associated to the edges below is equal to that of the frequencies for edges above, and that frequencies associated with Wick ordering are removed. These are the Kirchhoff rules for the graph. At the final vertex $v_R$, we impose the Dirac $\delta(k_{n,1}-k_R)$ where $k_R$ denotes the total output frequency. This gives the following formula for $\widehat{u_{\ell}} (k)$, where $p$ is the number of vertices $v_i$ whose edge above them carries a $-1$ parity sign
\begin{equation} 
\label{colvert}
\begin{split}
& \widehat{u_{\ell}}(t,k_R) = e^{-it |k_R|^2} i^n (-1)^p \left(\frac{-i\lambda^2}{(2\pi)^d}\right)^n\sum_{\underline{k} \in \mathbb{Z}^{d \# \mathcal I_n}} \\
& \qquad \qquad \qquad  \int_{\mathbb R_+^{n+1}} \prod_{i\in I_{2n+1} }\widehat{u_0}(k_{0,i},\sigma_{0,i})  \prod_{k=1}^{n}e^{-i\Omega_k \sum_{j=0}^{k-1}s_j} \Delta_{\ell}( \underline{k}, k_R)\delta(t-\sum_{i=0}^{n}s_i) d \underline{s}
\end{split}
\end{equation}
where we used the shorthand notations 
\begin{itemize}
\item $\underline{k} =  (k_{i,j})_{(i,j) \in \mathcal{I}_n} \in \mathbb Z^{d \# \mathcal I_n}$, $\mathcal I_n$ being the total number of edges in $G(\ell)$. Each edge of index $(i,j)$ having an associated momentum $k_{i,j}$.
\item $\underline{s} = (s_0,\dots,s_n) \in \mathbb{R}_+^{n+1}$ representing the ordered time slices.
\item $\widehat{u_0}(k,+1)=\widehat{u_0}(k)$ and $\widehat u_0(k,-1)= \widehat{ \overline{ u_0}}(k) = \overline{\widehat{ u_0}}(-k)$, $(\sigma_{0,i})_{1\leq i \leq 2n+1}$ representing the parities of initial vertices.
\item $\Omega_k=|k_{k-1,\ell_k+2}|^2-|k_{k-1,\ell_k}|^2+\sigma_{k,\ell_k}\left( |k_{k-1,\ell_k+1} |^2- |k_{k,\ell_k}|^2\right)$ is the resonance modulus corresponding to the vertex $v_k$
\end{itemize}
and, finally, $ \Delta_{\ell}$ encapsulates the Kirchhoff law and frequency truncation due to Wick renormalisation at each vertex:
$$
\Delta_{\ell}(\underline{k},k_R) = \Delta_\ell^{K}(\underline{k}) \Delta_{\ell}^{W}(\underline{k}) \delta(k_{n,1} - k_R),
$$
with
\be
\label{id:defDeltaln} \Delta_\ell^{K}(\underline{k}) = \prod_{i=1}^n \left\{ \left(\prod_{j=1}^{\ell_i-1}\delta(k_{i,j}-k_{i-1,j})\right) \delta \left(k_{i,\ell_i}-\sum_{j=0}^2 k_{i-1,\ell_i+j} \right)\prod_{j=\ell_i+1}^{1+2(n-i)} \delta(k_{i,j}-k_{i-1,j+2}) \right\} ,
\ee
\be \label{id:defDeltaWick}
\Delta_{\ell}^W(\underline{k}) = \prod_{i=1}^n \left(1-\delta(k_{i-1,\ell_i}+k_{i-1,\ell_i+2})-\delta(k_{i-1,\ell_i+1}+k_{i-1,\ell_i+1-\sigma_{i,\ell_i}})\right).
\ee

\subsection{Computing averaged quantities, paired diagrams} \label{subsec:paireddiagrams}

We now describe how to represent diagrammatically the expectation of $\| u_G \|_{L^2}^2$. Consider two summands as in~\eqref{colvert} corresponding to two histories $\ell$ and $\ell'$. Each is represented by a tree, that we call the \textit{left subtree} and the \textit{right subtree}. We adopt the notation that the vertices, time variables, etc...  of the right subtree carry primes, and those of the left subtree do not. We will sometimes concatenate the initial frequencies (and parities, etc...) $(k_{0,i})$ and $(k'_{0,i})$ into a single vector as follows
$$
(\widetilde{k}_{0,1}, \dots , \widetilde{k}_{0,4n+2}) = (k_{0,1},\dots,k_{0,2n+1}, k'_{0,1},\dots,k'_{0,2n+1}).
$$
The expectation will force two by two equalities for initial frequencies. Indeed, by Wick's formula, 
\be \label{id:wickformula}
\mathbb E\left( \prod_{i=1}^{4n+2} \widehat {u_0}(k_{0,i},\sigma_{0,i}) \right) =\ep^{d(2n+1)}\sum_{P\in \mathcal P(n)} 
\prod_{\{i,j\}\in P} |A(\epsilon \widetilde k_{0,i})|^2,
\ee
where the so-called pairing $P$ is a partition of $I_{4n+2}$ into pairs satisfying $\sigma_{0,i}\sigma_{0,j}=-1$ if $\{i,j\}\in P$, and $\mathcal P(\ell,\ell')$ denotes the set of such pairings.

To each interactions histories $\ell$ and $\ell'$ and pairing $P$, we will associate a \emph{paired diagram} the following way. The top parities of the left and right subtrees are set to $\sigma_{n,1}=1$ and $\sigma_{n,1}' = -1$ respectively. The root vertices $v_R$ and $v_R'$ of the left and right subtrees are merged into a single root vertex $v_R$. An edge $e_R$ is attached to this root vertex, with frequency $k_R = 0$. We demand that Kirchhoff's rule applies at $v_R$, namely: $k_{n,1} + k'_{n,1} = 0$. For each $\{i,j\} \in P$, we create a vertex $v_{\{i,j\}}$; then we add two edges $e_{-1,i}$ and $e_{-1,j}$, called \textit{upper pairing edges}, between $v_{0,i}$ or $v_{0,j}$ respectively, and $v_{\{i,j\}}$. We attach an edge $e_{\{i,j\}}$, called \emph{root pairing edge}, to $v_{\{i,j\}}$, and set its frequency $k_{\{i,j\}} = 0$. Finally, we demand that Kirchhoff's rule applies at $v_{\{i,j\}}$ and $v_{0,i}$: namely $k_{0,i} + k_{0,j} = 0$ if $\{i,j\} \in P$, and $k_{-1,i} = k_{0,i}$ for all $i$.

This construction corresponds to the following picture for the paired diagram:

\begin{center}
\includegraphics[width=15cm]{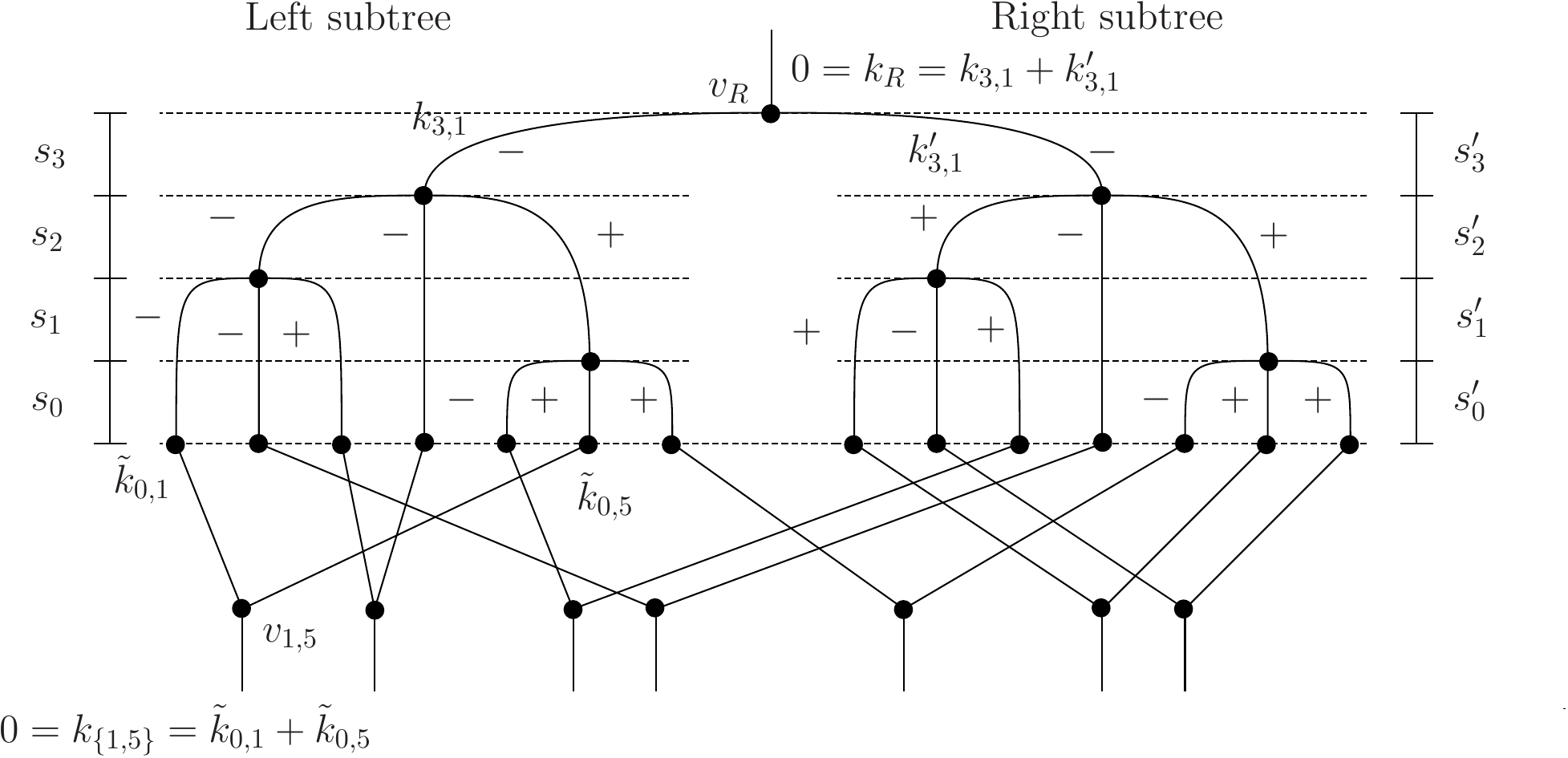}
\end{center}

We extend to paired diagrams the definition of the natural time ordering which was defined for ordered graphs. For two edges $e$, $e'$, we write $e\leq e'$ if:
\begin{itemize}
\item $e=e_{\{ i,j\}}$ is a root pairing edge and $e'$ is any other edge.
\item $e=e_{-1,i}$ is an upper pairing edge, and $e'\neq e_{\{i,j\}}$ is not a root pairing edge.
\item $e$ and $e'$ are not pairing edges, $e$ belongs to the interaction graph on the right and $e'$ to that on the left.
\item $e$ and $e'$ are not pairing edges and belong to the same interaction graph (left or right), and end at two time slices ($s_i$ and $s_j$ or $s_i'$ and $s_j'$) with $i\leq j$.
\end{itemize}
Notice that edges that lie below a same vertex are considered as equal for this order. This ordering corresponds to the following fact: to compute upper or lower bounds for the quantities represented by paired diagrams, we will use algorithms that will consider iteratively vertices over the graph according to the natural time ordering of the time slices, first over the graph on the right, then over the graph on the left.

From \fref{id:defDeltaln}, and from Wick formula \fref{id:wickformula}, there holds the identity:

\be \label{LpLpexpression}
\mathbb E \| u^n_{G} \|_{L^2}^2 =\sum_{P,\ell,\ell'} \mathcal F(\ell,\ell',P),
\ee
where $\ell$ and $\ell'$ give the interaction history for the left and right subgraphs (for $\hat{\overline{u_\ell}}$ and $\hat u_\ell$ respectively), that is, a special choice of the ordering of time slices for $G$; and $P \in \mathcal{P}$ is the pairing which is considered, and
\begin{align}
\label{id:mathcalF}\mathcal F(\ell,\ell',P)&= \frac{\lambda^{4n}\ep^{d(2n+1)}}{(2\pi)^{2dn}} \sum_{\substack{\underline{k},\underline{k'} \in \mathbb{Z}^{d(2 \# \mathcal{I}_n)}\\ \underline{k_{-1}} \in \mathbb{Z}^{d(4n+2)}}} \int_{\mathbb R_+^{n+1}\times \mathbb R_+^{n+1}} \prod_{k=1}^{n}e^{-i\Omega_k \sum_{j=0}^{k-1}s_j}\prod_{k=1}^{n}e^{-i\Omega_k' \sum_{j=0}^{k-1}s_j'} \\
\nonumber & \qquad \qquad  \qquad  \qquad  \qquad \prod_{\{i,j\}\in P} |A(\epsilon \widetilde k_{0,i})|^2   \Delta_{\ell,\ell',P}(\underline{k},\underline{k}',0)\delta \left(t-\sum_{i=0}^{n}s_i \right)\delta \left(t-\sum_{i=0}^{n}s_i' \right) \, d\underline{s} \, d \underline{s}' 
\end{align}
where 
\begin{itemize}
\item $\underline{k} =  (k_{i,j})_{(i,j) \in \mathcal{I}_n}  \in \mathbb{Z}^{d(\# \mathcal I_n + 2N+1)}$ and similarly for $\underline{k'}$. $\underline{k_{-1}} =  (k_{-1,i})_{i \in \{1,\dots,4n+2\}}$.
\item $\underline{s} = (s_0,\dots,s_n) \in \mathbb{R}_+^{n+1}$ and similarly for $\underline{s'}$,
\item $\Omega_k$ was defined previously below \fref{colvert}, and $\Omega_k'$ is defined in an analogous fashion,
\end{itemize}
and $\Delta_{\ell,\ell',P}$ records the Kirchhoff rules and frequency truncation of Wick renormalisation:
\begin{align*}
\Delta_{\ell,\ell',P}(\underline{k},\underline{k}',\underline{k_{-1}},k_R) = \Delta_{\ell,\ell',P}^K(\underline{k},\underline{k}',\underline{k_{-1}},k_R)  \Delta_{\ell,\ell'}^W(\underline{k},\underline{k}') 
\end{align*}
with
$$
 \Delta_{\ell,\ell'}^W(\underline{k},\underline{k}') =  \Delta_{\ell}^W(\underline{k})\Delta_{\ell'}^W(\underline{k'}),
$$
$$
\Delta_{\ell,\ell',P}^K(\underline{k},\underline{k}',\underline{k_{-1}},k_R) = \Delta_{\ell}^K(\underline{k}) \Delta_{\ell'}^K(\underline{k'}) \delta(k_R - k_{n,1} - k'_{n,1}) \Delta_P(\underline{k}),
$$
where $\Delta_P$ records the Kirchhoff laws due to the pairing:
$$
\Delta_P(\underline{k},\underline{k'}) = \prod_{\{i,j\}\in P}\delta(\widetilde k_{0,i}- k_{-1,i}) \delta(\widetilde k_{0,j}- k_{-1,j}) \delta(k_{-1,i}+k_{-1,j}).
$$

\subsection{Solving Kirchhoff's laws}

The difficulty in the study of $\mathcal F(\ell,\ell',P)$ defined by \fref{id:mathcalF} is that momenta in the graph are constrained via the Kirchhoff laws of the graph. This is analytically represented by the support of $\Delta_{\ell,\ell',P}^K$, and that the oscillatory phases take into account the natural time ordering in the graph. We aim at finding a minimal collection of edges from which, given their associated wave number $k_{i,j}$, one can retrieve the values of all other wave numbers. The first trivial simplification is to identify two edges when they are continued unchanged (no merging) from one slice to another: $(e_{i,k},k_{i,k})$ is identified with $(e_{i+1,k},k_{i+1,k})$ if $k<\ell_{i+1}$, and $(e_{i,k},k_{i,k})$ is identified with $(e_{i+1,k},k_{i+1,k-2})$ if $k>\ell_{i+1}+2$.
The vector space given by the Kirchhoff rules for the $k_{i,j}$ has dimension $2n+1$: indeed, there are $2(2n+1)$ initial frequencies, and $2n+1$ pairings (the condition that the frequencies add up to zero being induced by the pairing). Therefore, we will choose $2n+1$ \textit{free edges}, from whose frequencies all other frequencies can be reconstructed; these free edges will be determined by a spanning tree which will be constructed shortly.
Edges that are not free will be called \textit{integrated}. In addition, frequencies of integrated edges will only depend on the frequencies of free edges appearing after for the natural time ordering, allowing for an iterative algorithm to estimate $\mathcal F(\ell,\ell',P)$.\\

We sum up in the next Theorem the outcome of the standard strategy, explained in \cite{CG}, to determine such edges, and giving suitable properties. An example is given in the proof of Lemma \ref{lem:resonant}. The following Theorem collects the results obtained in Subsection 4.4 of \cite{CG}, as an adaptation of \cite{LS2}. Pictures of examples are given in Subsubsection \ref{subsubsec:example}.

\begin{theorem}[see \cite{CG}] \label{th:spanning}

Consider a frequency graph $G(\ell,\ell')$ with pairing $P$. There exists a complete integration of the frequency constraints \fref{id:defDeltaln} in the following sense. There exists a subset of free interaction edges:
$$
\mathcal E^f =\{e^f_1,...,e^f_{2n+1} \},
$$
with associated frequencies $(\xi_i)_{1\leq i \leq 2n+1}\in \mathbb Z^{d(2n+1)}$, satisfying the following properties.
\begin{itemize}
\item \emph{Basis property}: On the vectorial subspace of $\mathbb R^{2 \# \mathcal I_n+4n+2}$ determined by the Kirchhoff rules encoded in $\Delta_{\ell,\ell'}$, the family $(k^f_i)_{1\leq i \leq 2n + 1} $ is a basis in the following sense: the map $(\xi_i)_{1 \leq i \leq 2n+1} \to (\underline{k},\underline{k'},\underline{k_{-1}})$ is a linear bijection.
\item \emph{Time ordering for the spanning}: if $1\leq i\leq j\leq 2n+1$ then $e^f_i\leq e^f_j$ for the natural time ordering of the graph. Furthermore, if $e$ is not a free edge, then its associated wave number can be uniquely written as:
$$
k_e=\sum_{1\leq i \leq 2n+1} c_{i,e} \xi_i \quad \mbox{with} \quad c_{i,e}\in \{-1,0,1\},
$$
and $c_{i,e}=0$ whenever $e^f_i < e$ for the natural time ordering of the diagram.
\item \emph{Location of the last free edge}: The last free edge $e_{2n+1}^f$ is the one at the bottom left of the root vertex.
\item \emph{Properties of the vertices}: any interaction vertex $v$ is of one of the following forms:
\begin{itemize}
\item there is no free edge below $v$, and we say $v$ is of degree zero.
\item there is one free edge below $v$, and we say $v$ is of degree one.
\item there are two free edges below $v$, and we say $v$ is of degree two.
\end{itemize}
Denoting by $n_0$, $n_1$ and $n_2$ the number of degree zero, one and two vertices there holds:
\be \label{id:ni}
n_0+n_1+n_2=2n, \ n_1+2n_2=2n.
\ee
Moreover,
\be \label{id:nondegenn0}
\left|
\begin{array}{l l}
\mbox{either }n_0\geq 1,\\
 \mbox{or $v_1'$ is of degree one and there exist } (c_i)_{1<i\leq 2n+1}\mbox{ such that }|\Delta^W_{\ell,\ell'}(\underline{k},\underline{k}') |\lesssim \delta \left(\xi_1-\sum_{i=2}^{2n+1}c_i\xi_i\right)
\end{array}
\right.
\ee
\item \emph{Formula for degree one vertices}: for all $1\leq i \leq n$, assume that only one of the edges below $v_i$ is free, and denote by $\xi_k$ its wave number. Then there exist $\sigma,\sigma'\in \{\pm1\}$, $\zeta \in \mathbb Z^d$ that is a linear combination of free momenta $\xi_{k'}$ for $k'>k$, and $\omega \in \mathbb R$ a number depending only on $(\xi_{k'})_{k'>k}$ such that:
\be \label{id:formuladegre1omega}
\Omega_i=\sigma |\xi_k|^2_H+\sigma'|\xi_k+\zeta |^2_H+\omega.
\ee
In addition, if $\sigma\sigma'=-1$, then there exists $\eta \in \mathbb Z^d$ that is a linear combination of free momenta $\xi_{k'}$ for $k'>k$ such that:
\be  \label{id:wicknondegeneracy}
|\Delta^W_{\ell,\ell'}(\underline{k},\underline{k}') |\lesssim |1-\delta (\zeta)-\delta(\xi_k+\eta)|
\ee

The same property holds for resonance moduli $(\Omega_i')_{1\leq i \leq n}$ of the right subtree.
\item \emph{Formula for degree two vertices}: for all $1\leq i \leq n$, assume that two of the edges below $e_i$ are free, and denote by $\xi_k$ and $\xi_{k+1}$ their wave numbers. Then there exist $\tilde \xi \in \mathbb Z^d$ that is only a sum of free moment $\xi_{k'}$ for $k'>k$ such that:
\be \label{id:formuladegre2omega}
\Omega_i=2H\left(\xi_k+\tilde \xi \right)\cdot \left(\xi_{k+1} +\tilde \xi \right) \quad \mbox{or}\quad \Omega_i=2H\left(\xi_k+\tilde \xi \right)\cdot \left(\xi_k+\xi_{k+1} \right) .
\ee
The same formula applies for resonance moduli $(\Omega_i')_{1\leq i \leq n}$ of the right subtree.
\end{itemize}

\end{theorem}

In addition to the previous Theorem used in \cite{CG}, we need to refine the treatment of degree one vertices. Consider a degree one vertex $v$, with associated free variable $\xi$, and with the formula \fref{id:formuladegre1omega} expressing the resonance modulus at $v$ as a function of $\xi$. All possible combinations of $\sigma,\sigma'\in \{-1,+1\}$ in \fref{id:formuladegre1omega} are possible (examples are easy to construct). Then, the modulus given by \fref{id:formuladegre1omega} can be of two forms. It can be of linear type:
\be \label{id:degreeonelinear}
\Omega_v=2H\zeta \cdot\xi+\tilde \Omega,
\ee
in the case $\sigma \sigma'=-1$, where $\zeta$ and $\tilde \Omega$ depend only on free variables appearing after $\xi$. Or it can be of quadratic type:
\be \label{id:degreeonequadratic}
\Omega_v=2\sigma \left|\xi+\frac{\zeta}{2}\right|^2_H+\tilde \Omega
\ee
in the case $\sigma \sigma'=1$, where $\zeta$ and $\tilde \Omega$ depend only on free variables appearing after $\xi$.

\begin{definition} \label{def:degreeonecases}

We say a degree one vertex $v$ is of \emph{linear type} if the formula \fref{id:degreeonelinear} holds true. We then call $\zeta$ in \fref{id:degreeonelinear} the \emph{dual variable} at the vertex $v$. We say a degree one vertex $v$ is of \emph{quadratic type} if the formula \fref{id:degreeonequadratic} holds true. The total number of degree one vertices of linear type is denoted by $n_1^l$, and that of quadratic type by $n_1^q$, so that $n_1=n_1^l+n_1^q$.

\end{definition}

The identity \fref{id:wicknondegeneracy} directly implies the following.

\begin{lemma} \label{lem:wicknondegeneracy}

Consider a degree one vertex of linear type with associated free and dual variables $\xi_k$ for some $k\in \{1,...,2n+1\}$ and $\zeta$ respectively. Then there exists $\eta$ a linear combination of free variables $(\xi_{k'})_{k'>k}$ appearing after $\xi_k$, such that:
$$
\mbox{if} \quad \zeta=0\quad  \mbox{then}\quad  |\Delta^W_{\ell,\ell'}(\underline k,\underline{k'})|\lesssim \delta (\xi_k+\eta).
$$

\end{lemma}

\section{Convergence for large kinetic time and generic dispersion relation} \label{sec:diagrams}

\label{renard}
We give in this Section the proof of Theorem \ref{th:main}. Throughout the Section, $H$ is taken such that the "generic" bounds of Theorem \ref{theoremNT} hold true. The proof uses an approximation scheme, relying on certain estimates for its implementation. These estimates are proved using graph analysis and number theoretical results. The analysis follows closely that of \cite{CG}, except for the number theoretical results. These new number theoretical results are proved precisely in Section \ref{sec:numbergeneric}. The present Section uses them to prove Theorem \ref{th:main}; it is very similar to the analogue part in \cite{CG}. Thus, we start by detailing the approximation scheme for the sake of clarity, then we detail the proof of a key bound using the new number theoretical results, and next refer safely to \cite{CG} for the obtention of the other bounds adapting the very same strategy.

\subsection{Approximate solution, resolution scheme}

We construct a solution to \fref{id:NLSgeneric} as the sum of an approximate solution and a remainder:
$$
u = u^{app} + u^{err}=e^{i\lambda^2 \omega}\left(v^{app}+ v^{err} \right), \qquad \mbox{with} \qquad v^{app} = \sum_{n=0}^N u^n ,
$$
where the phase $\omega$ is given by \fref{def:omega} and where $u^n$ is defined via the iteration of Duhamel formula and Wick renormalisation \fref{defun}. We define further:
$$
V^{i,j} =  \langle  u^i , u^{j}\rangle , \qquad V = \sum_{i,j \leq N} V^{i,j}=\left\| \sum_{n=0}^N u^n\right\|_{L^2(\mathbb T^d) }^2.
$$
As equation \fref{id:NLSgeneric} preserves the mass $\| u(t) \|_{L^2(\mathbb T^d)}=\| u_0\|_{L^2(\mathbb T^d)}$, from its very definition \fref{defun} the approximate solution satisfies:
\begin{align*}
 i\partial_t u^{app} + \Delta_H u^{app} - \lambda^2 |u^{app}|^2 u^{app}  &= \frac{2\lambda^2}{(2\pi)^d} \left(2\mathfrak{Re}\langle u^{err},u^{app}\rangle +\left\| u^{err}\right\|_{L^2(\mathbb T^d)}^2\right)u^{app}+e^{i\lambda^2 \omega} E^N ,
\end{align*}
where on the right hand side the first term is part of the linearised dynamics for the remainder $u^{err}$ and where error that is generated is
\begin{align*}
&E^N = \lambda^2 \left[- \sum_{\substack{i,j,k \leq N \\ i+j+k \geq N}} u^i \overline{u^j} u^k + \frac{2}{(2\pi)^d} \sum_{\substack{i,j,k \leq N \\ i+j+k \geq N}} V^{i,j} u^k \right].
\end{align*}
The equation for the renormalised remainder $v^{err}$ is then:
\begin{align*}
i\partial_t v^{err} + \Delta_H v^{err} + \frac{2 \lambda^2}{(2\pi)^d} \| u(t)\|_{L^2}^2 v^{err} &= \lambda^2 \left( |v^{err} + v^{app}|^2 (v^{err} + v^{app}) - |v^{app}|^2v^{app} \right)  - E^N\\
&- \frac{2\lambda^2}{(2\pi)^d} \left(2\mathfrak{Re}\langle v^{err},v^{app}\rangle +\left\| v^{err}\right\|_{L^2}^2\right)v^{app}.
\end{align*}
In order to use Bourgain spaces, we shall only solve locally in time the above equation, and introduce for this aim a smooth cut-off function $\chi \in C^{\infty}(\mathbb R,\mathbb R)$ with $\chi(t)=1$ for $|t|\leq 1$ and $\chi (t)=0$ for $|t|\geq 2$. We will then construct a solution to the above equation but smoothly localised in time:
\begin{equation} \label{id:eqwerr}
i\partial_t w^{err} + \Delta_H w^{err} = \mathcal{L}(w^{err}) + \mathcal{B}(w^{err}) + \mathcal{T}(w^{err}) + \mathcal{E},
\end{equation}
where the linear, bilinear, trilinear, and error terms are given by
\begin{align*}
& \mathcal{L}(w) = \chi \left( \frac{t}{T}\right) \lambda^2 \left[ 2 |v^{app}|^2 w - 2(2\pi)^{-d} V w-2(2\pi)^{-d}\langle v^{app},w\rangle v^{app} + (v^{app})^2 \overline{w}-2(2\pi)^{-d}\langle w,v^{app}\rangle v^{app} \right] \\
& \mathcal{B}(w) = \chi \left( \frac{t}{T}\right) \lambda^2 \left[ 2 |w|^2 v^{app}-2(2\pi)^{-d}\| w\|_{L^2}^2v^{app}-2(2\pi)^{-d}\langle w,v^{app}\rangle w + w^2 \overline{v^{app}}-2(2\pi)^{-d}\langle v^{app},w\rangle w \right] \\
& \mathcal{T}(w) =\chi \left( \frac{t}{T}\right) \lambda^2 \left(|w|^2 w- 2(2\pi)^{-d} \| w \|_{L^2}^2w\right) \\
& \mathcal{E} = - \chi \left( \frac{t}{T}\right)  E^N.
\end{align*}
We will ensure that $v^{err}\equiv w^{err}$ on $[0,\ep^\nu T_{kin}]$ by taking for $0<\tilde \nu<\nu$:
\be \label{def:T}
T=\ep^{\tilde \nu} T_{kin}.
\ee

\subsection{Control of the remainder, proof of Theorem \ref{th:main}}

We now turn to the construction and control of $w^{err}$ solving \fref{id:eqwerr}.

\subsubsection{$X^{s,b}_{\epsilon,T}$ spaces}

We define the adapted $H^s_\epsilon$ Sobolev space by its norm
$$
\| f \|_{H^s_\epsilon} = \| \langle \epsilon D \rangle^s f \|_{L^2},
$$
and accordingly the adapted $X^{s,b}_{\epsilon,T}$ Bourgain space (depending on the relation dispersion $|\cdot |_H^2$):
$$
\| u \|_{X^{s,b}_\epsilon} = \| \langle \epsilon k \rangle^s \left( |\tau + |k|_H^2| + T^{-1} \right)^b \widetilde{u}(\tau,k) \|_{L^2_\tau L^2_k}.
$$
It satisfies the following properties. Note that since $b$ will be taken very close to $1/2$ the $T^{b-1/2}$ factors will not matter for the analysis.
\begin{itemize}
\item (Free solution) 
$$
\left\| \chi \left( \frac{t}{T} \right) e^{it\Delta_H} u_0 \right\|_{X^{s,b}_{T,\epsilon}} \lesssim T^{\frac{1}{2}-b} \| u_0 \|_{H^s_\epsilon}.
$$

\item (Time restriction)
$$
\left\| \chi \left( \frac{t}{T} \right) u \right\|_{X^{s,b}_{T,\epsilon}} \lesssim  \| u \|_{X_{T,\epsilon}^{s,b}}.
$$

\item (Time continuity)
\be \label{bd:bourgaintosobo}
\displaystyle \| u \|_{\mathcal{C}_t H^s_\ep} \lesssim T^{b-\frac{1}{2}} \| u \|_{X^{s,b}_{T,\epsilon}}
\ee
 
\item (Strichartz estimates). We recall the following result from Theorem 2.2 in \cite{BoDe} for $s>\frac d4-\frac 12$:
\be \label{bd:strichartzsobolevspaces}
\| e^{it \Delta_H } f \|_{L^4_{T} L^4} \lesssim \left( 1 + T^{\frac 14} \right) \left\| \langle D \rangle^{s} f \right\|_{L^2}
\ee
which then implies after a scaling argument and applying Lemma 2.9 from \cite{Tao}:
\be \label{bd:strichartzbourgainspaces}
\| u \|_{L^4_T L^4} \lesssim T^{b - \frac{1}{2}} \epsilon^{\frac{1}{2} - \frac{d}{4}-\kappa} \left( 1 + T^{\frac 14} \right) \left\| u \right\|_{X^{s,b}_{T,\epsilon}}.
\ee

\item (Hyperbolic regularity) If $u$ solves
$$
\left\{
\begin{array}{l}
i \partial_t u + \Delta_H u = F \\
u(t=0) = 0,
\end{array}
\right.
$$
then
\be
\label{hyperbolicregularity}
\left\| \chi \left( \frac{t}{T} \right) u \right\|_{X^{s,b}_{\epsilon,T}} \lesssim \| F \|_{X^{s,b-1}_{\epsilon,T}}.
\ee
\end{itemize}

\subsubsection{Estimates for the resolution scheme, contraction argument}

The following proposition states that all necessary bounds to solve \fref{id:eqwerr} via a Banach-Picard expansion hold with large probability.

\begin{proposition}  \label{aigleroyal}

Assume \fref{id:largekinetictime}, and take $N\in \mathbb N$, $p\geq 2$ and $\mu>0$, $s>\frac d2-1$ and $b^*>\frac 12$. Then there exists $b^*\geq b > \frac{1}{2}$ and a set $E=E_{t,N,\mu,s,p}$ with probability $\mathbb{P}( E )\geq 1 - \epsilon^{\mu}$ such that on $E$, if $n \leq N$,  $1\leq T\leq \min(T_{kin},\ep^{2-d+\kappa})$ and $0<t\leq T$:
\begin{itemize}
\item Bounds on the expansion:
\begin{align}
\label{bd:unL2} & \| u^n(t) \|_{L^2} \lesssim_{N,\mu} \epsilon^{-\mu} \left( \frac{1+t}{T_{kin}} \right)^{n/2} \\
\label{bd:unXsb} & \| \chi\left(\frac tT \right)u^n \|_{X^{0,b}_{\ep,T}} \lesssim_{N,\mu,b'} \epsilon^{-\mu} \left( \frac{T}{T_{kin}} \right)^{n/2}  \\
\label{bd:unLp} & \| u^n(t) \|_{L^p_T L^p} \lesssim_{N,\mu,p} \epsilon^{-\mu} T^{\frac 1p} \left(  \frac{T}{T_{kin}}\right)^{\frac{n}{2}} (T^{1/2} \epsilon^{-1})^{\frac{1}{2} - \frac{1}{p}} \\
\label{bd:mathcalE}& \left\| \chi\left( \frac{t}{2T} \right) \int_0^t e^{i(t-s)\Delta_H} \mathcal E \,ds \right\|_{X^{s,b}_{\epsilon,T}} \lesssim_{N,\mu} \epsilon^{-2\mu} T^{\frac{11}{8}} \ep^{-\frac 14 -\frac d4}\left(\frac{T}{T_{kin}}\right)^{\frac N2}
\end{align}
\item Bound on the linear term:
\be \label{bd:mathfrakL}
\left\| \chi\left(\frac{t}{2T}\right)\int_0^t e^{i(t-s)\Delta_H} \mathfrak{L} \,ds \right\|_{X^{s,b}_{\epsilon,T} \to X^{s,b}_{\epsilon,T}} \lesssim_{N,\mu}
\epsilon^{-\mu} \sqrt{\frac{T}{T_{kin}}}
\ee
\item Bound on the bilinear term:
\be \label{bd:mathcalB}
\left\| \chi \left( \frac{t}{2T}\right)\int_0^t e^{i(t-s)\Delta_H} \mathcal{B}(u) \,ds  \right\|_{X^{s,b}_{\epsilon,T}} \lesssim_{N,\mu} \lambda^2 T^{\frac 54} \epsilon^{\frac{1}{2} - \frac{d}{2}-2\mu} \| u \|_{X^{s,b}_\epsilon}^2
\ee
\item Bound on the trilinear term:
\be \label{bd:mathcalT}
\left\| \chi \left( \frac{t}{2T}\right) \int_0^t e^{i(t-s)\Delta_H} \mathcal{T}(u) \,ds  \right\|_{X^{s,b}_\epsilon}  \lesssim \lambda^2 T \epsilon^{2-d-\kappa} \| u \|_{X^{s,b}_{\epsilon}}^3
\ee
\end{itemize}
\end{proposition}

\begin{remark}

Apart from the $T/T_{kin}$ factors, there are also $T^C$ and $\ep^{-C}$ extra factors for various constant $C$'s above, except for the bounds \fref{bd:unL2}, \fref{bd:unXsb}, \fref{bd:mathfrakL} which are sharp. These extra factors may not be sharp: they come from the $T^{1/4}$ factor in the long time Strichartz estimate \fref{bd:strichartzsobolevspaces} and from the $L^p$ bound \fref{bd:unLp} for $p>2$ which are probably not optimal. However, since they are only used in multilinear estimates, and that the error can be made arbitrarily small in \fref{bd:mathcalE} by taking $N$ large enough, these potentially not optimal factors are harmless for the fixed point argument.

\end{remark}

Assuming the above Proposition, we can end the proof of Theorem \ref{th:main}.

\begin{proof}[Proof of Theorem \ref{th:main}]
We apply the Banach fixed point theorem in $B_{X^{s,b}_{\epsilon,T}}(0,\rho)$, where $s> \frac{d}{2}-1$, and $\rho>0$ will be fixed shortly, to the mapping
$$
\Phi: w \mapsto \chi\left(\frac{t}{2T}\right) \int_0^t e^{i(t-s) \Delta_H} [\mathcal{L}(w) + \mathcal{B}(w) + \mathcal{T}(w) + \mathcal{E} ] \,ds 
$$
The assumption \fref{id:largekinetictime} and \fref{def:T} ensure $\frac{1+T}{T_{kin}} \leq 2\ep^{\mu}$ for $\mu>0$ small enough. Hence, applying \fref{bd:mathcalE}, the error term can be made $< \epsilon^M$, for any arbitrarily large fixed $M$, in $X^{s,b}_{\ep,T}$ by choosing $N$ sufficiently big. This leads to the choice $\rho = 2 \epsilon^{M}$. The bound \fref{bd:mathfrakL} for $\mu$ small enough depending on $\nu$ then shows that the linear term has an operator norm $\ll 1$. Similarly, applying the bounds \eqref{bd:mathcalB} and~\ref{bd:mathcalT}, one checks easily that the bilinear and cubic term act as contractions on $B(0,\rho)$. Therefore, the Banach fixed point theorem gives the existence of a fixed point $w^{err}$, with norm $\| w^{err}\|_{X^{s,b}_{\ep,T}}\lesssim \epsilon^{M}$. By construction, this fixed point $w^{err}$ then solves \fref{id:eqwerr} on the time interval $[0,T]$. The bounds \fref{bd:iteratesHgeneric1} and \fref{bd:iteratesHgeneric2} then follow easily: once first applies the fixed point argument for some $\tilde N\gg N$ so that the size of $w^{err}$ is as small as we want, and then use the bounds \fref{bd:unXsb} and \fref{bd:bourgaintosobo} and the fact that $u^n$ has Fourier support in $|k|\lesssim_n \ep^{-1}$.
\end{proof}

\subsection{Estimates for the resolution scheme, proof of Proposition \ref{aigleroyal}}

\subsubsection{Paired Feynman diagrams}

We give first a detailed proof of the bound \fref{bd:unL2}. It is a direct consequence of the following proposition and of the Bienaym\'e-Tchebychef inequality.

\begin{proposition}  \label{pr:L2expansion}
There holds the following estimate for $0\leq T\leq \ep^{2-d+\kappa}$:
$$
\mathbb E \| u^n(T) \|_{L^2}^2 \lesssim_{n} \epsilon^{-\kappa} \left( \frac{1+T}{T_{kin}} \right)^{n} 
$$
\end{proposition}

We recall that from \fref{LpLpexpression}, $\mathbb E \| u^n_{G} \|_{L^2}^2 =\sum_{P,\ell,\ell'} \mathcal F(\ell,\ell',P)$ where the quantity $\mathcal F(\ell,\ell',P)$ is defined by \fref{id:mathcalF}, and are represented by the corresponding paired diagram defined in Section \ref{subsec:paireddiagrams}. To bound these quantities, we first recall the resolvent formula from Lemma 4.2 in \cite{CG}:
\be \label{bd:resolventidentity}
\int_{\mathbb R_+^{m}} \prod_{k=1}^{m} e^{-i s_k e_k} \delta \left( \sum_{k=1}^{m} s_k - t \right) ds_1 \dots ds_m =\frac{e}{2\pi}\int_{\mathbb R}e^{-i\alpha t}\prod_{k=1}^{m} \frac{i}{\alpha-e_k+\frac{i}{t}} d\alpha.
\ee
In particular we infer that (all terms below being defined at the same location as \fref{id:mathcalF}):
\bea
\label{id:Fresolvant}\left| \mathcal F\ell,\ell',P)\right|&\leq &\left(\frac{\lambda^2}{(2\pi)^d}\right)^{2n}\ep^{d(2n+1)} \frac{e^2}{(2\pi)^2}\int_{\mathbb R^2}\sum_{\substack{\underline{k},\underline{k'} \in \mathbb{Z}^{d(2 \# \mathcal{I}_n)}\\ \underline{k_{-1}} \in \mathbb{Z}^{d(4n+2)}}}  d\alpha d \alpha ' |\Delta_{\ell,\ell',P}(\underline{k},\underline{k}',0)| \\
\nonumber &&\qquad  \prod_{k=1}^{n} \frac{1}{|\alpha-\sum_{i=k}^{n}\Omega_i+\frac{i}{t}|}\frac{1}{|\alpha'-\sum_{i=k}^{n}\Omega_i'+\frac{i}{t}|}\frac{1}{|\alpha+\frac{i}{t}|}\frac{1}{|\alpha'+\frac{i}{t}|}  \prod_{\{i,j\}\in P} |A(\epsilon \widetilde k_{0,i})|^2  .
\eea

\subsubsection{Resolvent sums}

In view of \fref{id:FresolvantmathcalRc}, we now give technical lemmas to estimate terms of the form $\sum \frac{1}{|\alpha-\Omega+\frac{i}{t}|}$. The sum will be performed over the free variables of Theorem \ref{th:spanning}, for which the resonance moduli will take the forms \fref{id:formuladegre1omega} and \fref{id:formuladegre2omega}.

\begin{lemma} \label{lemmaRS}
Assume $d\geq 3$, $1\leq T\leq \ep^{2-d+\kappa}$ and that $H$ is generic. Then:
\begin{itemize}
\item[(i)] (Degree one vertices) If $\zeta\in \mathbb Z^d$ with $|\zeta|\lesssim \ep^{-1}$:
$$
\sum_{|\xi| < \ep^{-1}} \frac{1}{\left| |\xi|^2_H + |\zeta - \xi|^2_H - \alpha + \frac{i}{T}\right|} 
\lesssim_\kappa 
\epsilon^{1-d-\kappa} \sqrt T 
$$
and if $\zeta \neq 0$:
$$
\sum_{|\xi| < \ep^{-1}} \frac{1}{\left| |\xi|^2_H - |\zeta - \xi|^2_H - \alpha + \frac{i}{T}\right|} 
\lesssim_\kappa
\epsilon^{1-d-\kappa} \sqrt T
$$
\item[(ii)] (Degree two vertices) If $(\sigma,\sigma') = (-,-)$, $(-,+)$, or $(+,-)$:
$$
\sum_{|\xi|, |\xi'| < \ep^{-1}} \frac{1}{\left| |\xi|^2_H + \sigma |\xi'|_H^2 + \sigma' |\xi_0 - \xi - \xi'|^2_H - \alpha + \frac{i}{T}\right|} \lesssim_\kappa \epsilon^{2-2d-\kappa}.
$$
\end{itemize}
\end{lemma}

\begin{proof} These bounds follow directly from Theorem~\ref{theoremNT}. Indeed, consider the first one. We estimate from (ii) in Theorem~\ref{theoremNT}:
\bee
\sum_{|k| < \ep^{-1}} \frac{1}{\left| |k|^2_H + |k_0 - k|^2_H - \alpha + \frac{i}{T}\right|} &=&\sum_{n\approx -\ep^{-2}}^{\approx \ep^{-2}} \frac{T}{\langle n\rangle} \# \left\{ nT \leq |k|^2_H + |k_0 - k|^2_H-\alpha<(n+1)T  \right\}\\
&\lesssim & \sqrt{T}\ep^{d-1-\kappa}
\eee
and the other bounds of the Lemma are obtained analogously.

\end{proof}

\begin{remark} Lemma \ref{lemmaRS} fails for $H$ diagonal, in particular for the Laplacian $H=\Id$. Indeed, in the case $T>1$, $H$ diagonal, one can derive the bound
$$
\sum_{|\xi| < \ep^{-1}} \frac{1}{\left| |\xi|^2_H + \sigma |\zeta - \xi|^2_H - \alpha + \frac{i}{T}\right|} 
\lesssim_\kappa \epsilon^{1-d-\kappa} T,
$$
which is furthermore optimal. To see why it is optimal, simply choose $\sigma = -1$, $\alpha = -1$, $H = \operatorname{Id}$, and $\zeta = e_1$ (unitary vector along the first coordinate axis). The sum becomes then $\sum_{|\xi| < \ep^{-1}} \frac{1}{\left|2 \xi_1 + \frac{i}{T}\right|}$, which is clearly of order $\epsilon^{1-d} T+\epsilon^{1-d-\kappa}$.
\end{remark}

\subsubsection{Upper bounds for paired diagrams}

We are now ready to estimate \fref{id:Fresolvant} using Lemma \ref{lemmaRS} and the graph analysis of Section \ref{sec:diagrams}. This will end the proof of Proposition \ref{pr:L2expansion}.

\begin{proof}[Proof of Proposition \ref{pr:L2expansion}]

Recall that $\mathbb E \| u^n \|_{L^2}^2$ is given by \fref{id:FresolvantmathcalRc}, so that we need to bound $\mathcal F(\ell,\ell',P)$ given by \fref{id:Fresolvant}. In particular, denote by $(\xi_i)_{1\leq i \leq 2n+1}$ the free variables in the paired graph given by Theorem \ref{th:spanning}. We first split the sum for a fixed large constant $K>0$:

\bee
\left| \mathcal F(\ell,\ell',P)\right|&\lesssim & \underbrace{\lambda^{4n} \ep^{d(2n+1)} \int_{\max(|\alpha|,|\alpha'|)\geq \epsilon^{-K}}\sum_{\substack{\underline{k},\underline{k'} \in \mathbb{Z}^{d(2 \# \mathcal{I}_n)}\\ \underline{k_{-1}} \in \mathbb{Z}^{d(4n+2)}}}  [...]}_{I} +\underbrace{\lambda^{4n} \ep^{d(2n+1)} \int_{|\alpha|,|\alpha'|\leq \epsilon^{-K}}\sum_{\substack{\underline{k},\underline{k'} \in \mathbb{Z}^{d(2 \# \mathcal{I}_n)}\\ \underline{k_{-1}} \in \mathbb{Z}^{d(4n+2)}}}  [...]}_{II}.
\eee
By noticing that $A(\ep k)=0$ for $|k|\gg \ep^{-1}$, one has that in the integrand $\prod_{\{i,j\}\in P} |A(\epsilon \widetilde k_{0,i})|^2=0$ for $|(\xi_1,...,\xi_{2n+1}|\gg \ep^{-1}$. Therefore, we can rewrite the first term as:
\bee
I&\lesssim& \lambda^{4n} \ep^{d(2n+1)} \int_{\max(|\alpha|,|\alpha'|)  \geq  \epsilon^{-K}}\sum_{|\xi_1|,...,|\xi_{2n+1}|\lesssim \ep^{-1}} \frac{1}{|\alpha+\frac{i}{T}|}\frac{1}{|\alpha'+\frac{i}{T}|} \\
&& \qquad \qquad \qquad \qquad \qquad \qquad \qquad \qquad \prod_{k=1}^n \frac{1}{|\alpha+O(\ep^{-2})+\frac{i}{T}|}\frac{1}{|\alpha'+O(\ep^{-2})+\frac{i}{T}|} \lesssim \ep^{CK} 
\eee
for some universal $C>0$, hence the term $I$ is negligible for $K$ large enough.\\

\noindent We now turn to estimating the term $II$. We will employ an algorithm that involves all properties and definitions of Theorem \ref{th:spanning}.

In the first part of the algorithm, the variables $\alpha$, $\alpha'$ and $\xi_{2n+1}$ are fixed. The right graph is considered first. The algorithm estimates iteratively for $k=1,...,n$ the sums $\sum \frac{1}{|\alpha'-\sum_{i=k}^n \Omega_i'+\frac iT|}$ as follows, while the term $\prod_{k=1}^n \frac{1}{|\alpha-\sum_{i=k}^n \Omega_i+\frac iT|}$ associated to the left graph is kept fixed. The interaction vertices are considered one after another, starting from the bottom of the right graph, and going up toward its top. At the $k$-th step, the algorithm considers the vertex $e_k'$ and estimates the term $\sum \frac{1}{|\alpha'-\sum_{i=k}^n \Omega_i'+\frac iT|}$ where the sum is performed over the free variables associated to this vertex, while the term $\prod_{k'=k+1} \frac{1}{|\alpha'-\sum_{i=k'}^n \Omega_i'+\frac iT|}$ is kept fixed. This algorithm is permitted by the following observation: from Theorem \ref{th:spanning}, one has that $\Omega_k'$ depends on the free variables associated to $e_k'$ and on free variables appearing after $e_k'$ for the natural time ordering of the paired diagram, while $\Omega_{k'}'$ for $k'>k$ or $\Omega_{k''}$ for $1\leq k'' \leq n$ only depends on free variables appearing strictly after $e_k'$ (and are thus fixed quantities when summing over the free variables associated to $e_k'$).

Let us run the first part of the algorithm on the right graph. Let us assume that we are at step $k$, considering the interaction vertex $e_k'$ with resonance modulus $\Omega_k'$.
\begin{itemize}
\item If $e_k'$ is a degree zero vertex, we then estimate $ \frac{1}{|\alpha'-\sum_{i=k}^n \Omega_i'+\frac iT|}\leq T $ and do not perform summation as there are no free variables to sum over.
\item If $e_k'$ is a degree one vertex, let $\xi_k$ be the free variable attached below it. Let all free variables $(\xi_{j})_{j>k}$ appearing after $\xi_k$ being fixed. We use the formula \fref{id:formuladegre1omega} to compute $\Omega_k'$, and the bounds (i) of Lemma \ref{lemmaRS} to bound:
$$
\sum_{|\xi_k|\lesssim \ep^{-1}} \frac{1}{|\alpha'-\Omega_k'-\sum_{i=k+1}^n \Omega_i'+\frac iT|}\lesssim \sqrt T \ep^{1-d-\kappa}.
$$
There is only one degenerate case for which the above bound does not hold: when $\sigma\sigma'=-1$ and $\zeta=0$ in \fref{id:formuladegre1omega}. But in that case, Lemma \ref{lem:wicknondegeneracy} implies that on the support of $\Delta_{\ell,\ell',P}$, $\xi_k$ is a linear combination of the free variables $(\xi_{j})_{j>k}$. Hence the above sum restricted to the support of $\Delta_{\ell,\ell',P}$ contains only one element and is $\lesssim T\lesssim \sqrt T \ep^{1-d-\kappa}$.

\item If $e$ is a degree two vertex, let $(\xi_k,\xi_{k+1})$ be the free variables attached below it. Let all free variables $(\xi_{j})_{j>k+1}$ appearing after $\xi_k,\xi_{k+1}$ be fixed. We use the formula \fref{id:formuladegre2omega} to compute $\Omega$, and the bounds (ii) of Lemma \ref{lemmaRS} to bound:
$$
\sum_{|\xi_k,\xi_{k+1}|\lesssim \ep^{-1}} \frac{|\Delta_{\ell,\ell',P}(\underline{k},\underline{k}',0)|}{|\alpha'-\Omega_k'-\sum_{i=k+1}^n \Omega_i'+\frac iT|}\lesssim \ep^{2-2d-\kappa}.
$$
\end{itemize}
Once the algorithm has considered all vertices of the right graph, the left graph is considered and the very same estimates are performed. At the end of this first step, the summation has been performed over all free variables $(\xi_{i})_{1\leq i \leq 2n}$ except the last one $\xi_{2n+1}$. In \fref{id:Fresolvant}, all terms $\frac{1}{|\alpha-\sum_k^n \Omega_i +\frac iT|}$ and $\frac{1}{|\alpha-\sum_k^n \Omega_i' +\frac iT|}$ have been estimated and only the two terms $\frac{1}{|\alpha+\frac iT|}$ and $\frac{1}{|\alpha'+\frac iT|}$ remain.

In the second part of the algorithm, we estimate the number of possible values for $\xi_{2n+1}$ by the rough estimate $\epsilon^{-d}$ since $|\xi_{2n+1}|\leq \epsilon^{-1}$. Then we integrate over $d\alpha d\alpha'$ the term $\frac{1}{|\alpha+\frac iT|}\frac{1}{|\alpha'+\frac iT|}$ which produces a log correction.

We finally arrive at the estimate:
$$
II\lesssim \lambda^{4n}\ep^{d(2n+1)}T^{n_0}T^{\frac{n_1}{2}}\ep^{(1-d)n_1}\ep^{(2-2d)n_2}\ep^{-d}\ep^{-\kappa}=\left(\frac{T}{T_{kin}}\right)^{n}\ep^{-\kappa},
$$
where we used \fref{id:ni}. This bound and the first one we derived for the term $I$ prove Proposition \ref{pr:L2expansion}.

\end{proof}

We just exemplified how, admitting the number theoretical results of Theorem \ref{theoremNT}, the analysis of \cite{CG} adapts directly to prove the first bound of Proposition \ref{aigleroyal}, through the proof of Proposition \ref{pr:L2expansion}. This is also the case for all other bounds of Proposition \ref{aigleroyal}, so that we now solely sketch the adaptation of \cite{CG}.

\begin{proof}[Proof of Proposition \ref{aigleroyal}]

All bounds can be proved by employing the graph analysis tools and estimation algorithms as in \cite{CG}. The only major difference that the resolvent bounds there have to be replaced by those of Lemma \fref{lemmaRS}.

The $X^{s,b}_{\ep,T}$ bound \fref{bd:unXsb} can be obtained by this exact way following Subsection 5.4 of \cite{CG}, with the sole minor modifications of the weight involved in the $X^{s,b}_{\ep,T}$ space and of the time localising cut-off $\chi$ that needs to be replaced by $\chi(\frac{t}{T})$.

The $L^p$ bound \fref{bd:unLp} can be obtained by this exact way following Subsection 5.3 of \cite{CG} without other modifications.

The bound on the error \fref{bd:mathcalE} is estimated the same way as in Subsection 5.5 of \cite{CG}, using the Strichartz estimates \fref{bd:strichartzbourgainspaces} and \fref{hyperbolicregularity}, and the $L^p$ bound \fref{bd:unLp}.

The linear bound \fref{bd:mathfrakL} can be obtained following Section 6 of \cite{CG}. There, again the $X^{s,b}$ weights need to be replaced by $X^{s,b}_{\ep,T}$ weights, the cut-offs $\chi(t)$ by $\chi(\frac tT)$, so that the resolvent factor $\frac{1}{|[...]+i|}$ are transformed into $\frac{1}{|[...]+\frac{i}{T}|}$. All the computations follow through the exact same way.

The multilinear bounds \fref{bd:mathcalB} and \fref{bd:mathcalT} can be proved as in Section 7 of \cite{CG}. The only replacements are: $\chi$ by $\chi(t/T)$, $X^{s,b}$ by $X^{s,b}_{\ep,T}$, standard Strichartz estimates by \fref{bd:strichartzbourgainspaces} and \fref{hyperbolicregularity}, and $L^p$ bounds by \fref{bd:unLp}.

\end{proof}

\section{Failure of convergence for large kinetic time and standard Laplacian}

\label{sectionlarge}

This section is devoted to the proof of Proposition \ref{pr:tgeq1}. We also detail the analysis of two explicit examples of paired diagrams. For $t\gg 1$, the accumulation of nonlinear effects is due to leading order to completely resonant configurations. To show it we decompose the sums in \fref{id:mathcalF} as:
\begin{align}
\nonumber \mathcal F(\ell,\ell',P)&= \frac{\lambda^{4n}\ep^{d(2n+1)}}{(2\pi)^{2dn}} \sum_{\substack{\underline{k},\underline{k'} \in \mathbb{Z}^{d(2 \# \mathcal{I}_n)}\\ \underline{k_{-1}} \in \mathbb{Z}^{d(4n+2)}}} \int_{\mathbb R_+^{n+1}\times \mathbb R_+^{n+1}} \prod_{k=1}^{n}e^{-i\Omega_k \sum_{j=0}^{k-1}s_j}\prod_{k=1}^{n}e^{-i\Omega_k' \sum_{j=0}^{k-1}s_j'} \\
\nonumber & \qquad \qquad  \qquad  \qquad  \qquad \prod_{\{i,j\}\in P} |A(\epsilon \widetilde k_{0,i})|^2   \Delta_{\ell,\ell',P}(\underline{k},\underline{k}',0)\delta \left(t-\sum_{i=0}^{n}s_i \right)\delta \left(t-\sum_{i=0}^{n}s_i' \right) \, d\underline{s} \, d \underline{s}'  \\
\label{id:decompositionmathcalF} &= \underbrace{\frac{\lambda^{4n}\ep^{d(2n+1)}}{(2\pi)^{2dn}}  \sum_{\mathcal R}...}_{\mathcal F_{\mathcal R}(\ell,\ell',P)}+\underbrace{\frac{\lambda^{4n}\ep^{d(2n+1)}}{(2\pi)^{2dn}}  \sum_{\mathcal R^c}...}_{\mathcal F_{\mathcal R^c}(\ell,\ell',P)}
\end{align}
where  $\mathcal R$ is the set of totally resonant configurations, and $\mathcal R^c$ its complement set:
$$
(\underline k,\underline{k'},\underline{k_{-1}})\in \mathcal R \mbox{ if } \Omega_k=0 \mbox{ and }\Omega_{k}'=0 \mbox{ for all }1\leq k \leq n.
$$
As all frequencies $(\underline k,\underline{k'},\underline{k_{-1}})$ are retrieved from the free frequencies $(\xi_{i})_{1\leq i \leq 2n+1}$ given by Theorem \ref{th:spanning}, we shall abuse notations and write $(\xi_{i})_{1\leq i \leq 2n+1} \in \mathcal R$ and $(\xi_{i})_{1\leq i \leq 2n+1} \in \mathcal R^c$. Proposition \ref{pr:tgeq1} is a direct consequence of the two Lemmas below.

\begin{lemma}[Estimate for totally resonant configurations] \label{lem:resonant}

There exists $0<c<C$ depending uniquely on $n$ and $d$, such that for all $t\geq 0$, $0<\epsilon \leq 1$ and $0<\lambda$ in dimension $d=2$:
\be \label{bd:totallyresnonant}
c \left(\frac{t^2}{T_{kin}} \right)^n \langle \log \epsilon \rangle^{n} \leq \sum_{P,\ell,\ell'} \mathcal F_{\mathcal R}(\ell,\ell',P)\leq C \left(\frac{t^2}{T_{kin}} \right)^n \langle \log \epsilon \rangle^{n},
\ee
and in dimension $d\geq 3$:
\be \label{bd:totallyresnonant2}
c \left(\frac{t^2}{T_{kin}} \right)^n \leq \sum_{P,\ell,\ell'} \mathcal F_{\mathcal R}(\ell,\ell',P)\leq C \left(\frac{t^2}{T_{kin}} \right)^n .
\ee

\end{lemma}

\begin{lemma}[Estimate for non-resonant configurations] \label{lem:nonresonant}

There exists $0<C$ depending uniquely on $n$, such that for all $t\geq 1$, $0<\epsilon \leq 1$ and $0<\lambda$, for all $d\geq 2$:
\be \label{bd:nonresonant}
\left|\sum_{P,\ell,\ell'} \mathcal F_{\mathcal R^c}(\ell,\ell',P)\right| \leq \frac{C}{t} \left(\frac{t^2}{T_{kin}} \right)^n \langle \log \epsilon \rangle^{2n+2}.
\ee

\end{lemma}

\subsection{Key contributions: completely resonant configurations}

We prove in this subsubsection Lemma \ref{lem:resonant}. A first Lemma establishes an upper bound on the number of completely resonant configurations.

\begin{lemma} \label{lem:tgeq1upper}

For all interaction histories $\ell$, $\ell'$ of depth $n$ for the left and right subtrees, for all admissible pairing $P$:
\be \label{id:fact1}
N_{\ell,\ell',P,\mathcal R}=\# \{ (\xi_i)_{1\leq i \leq 2n+1}\in \mathcal R \mbox{ with } |\xi_i|\leq \epsilon^{-1}\} \lesssim \left\{ \begin{array}{l l} \epsilon^{-2n-2}\langle \log \epsilon \rangle^n \qquad \mbox{for }d=2, \\ \epsilon^{-2n(d-1)-d}  \qquad \mbox{for }d\geq 3. \end{array}\right.
\ee

\end{lemma}

\begin{proof}

Fix $(\ell,\ell',P)$, and denote by $(\xi_i)_{1\leq i \leq 2n+1}$ the free variables in the graph given by Theorem \ref{th:spanning}. We estimate $N_{\ell,\ell',P,\mathcal R}$ iteratively using an algorithm that is very similar to that of the proof of Proposition \ref{pr:L2expansion}. We look at the vertices one after another, starting from the bottom of the right graph, going up to the top of the right graph, then starting from the bottom of the left graph and going up to its top. This algorithm is permitted by the following: observe from Theorem \ref{th:spanning} that given a vertex $e$ with associated resonance modulus $\Omega$, the condition $\Omega=0$ only involve the free variables that may be attached below $e$, and the free variables appearing after.\\
Let us run our algorithm, and assume that we are at step $i$, and let $e$ be the vertex under consideration, with resonance modulus $\Omega$.
\begin{itemize}
\item If $e$ is a degree zero vertex, then we drop the constraint $\Omega=0$ out of the definition of the set $\mathcal R$. That is to say, we do not use this constraint for the free variables appearing in the next steps of the algorithm. Hence our algorithm will count a number of points which is greater our equal to $N_{\ell,\ell',P}$.
\item If $e$ is a degree one vertex, let $\xi_k$ be the free variable attached below it. Let all free variables $(\xi_{j})_{j>k}$ appearing after $\xi_k$ being fixed. We use the formula \fref{id:formuladegre1omega} to compute $\Omega$, and the bounds \fref{bd:degree1quadraticresonant} and \fref{bd:degree1linearresonant} to estimate:
$$
\forall (\xi_j)_{j>k} \mbox{ with }|\xi_j|\leq \epsilon^{-1}, \quad \# \left\{ |\xi_k|\leq \epsilon^{-1}, \quad \Omega=0\right\}\lesssim \epsilon^{1-d}.
$$
The above bound does not hold for only one degenerate case: if $\sigma\sigma'=-1$ and $\zeta=0$ in \fref{id:formuladegre1omega}. In that case however, Lemma \ref{lem:wicknondegeneracy} implies that on the support of $\Delta_{\ell,\ell',P}$, $\xi_k$ is a linear combination of the free variables $(\xi_{j})_{j>k}$. Hence the above set, intersected with the support of $\Delta_{\ell,\ell',P}$, contains only $1\lesssim \ep^{1-d}$ element.

\item If $e$ is a degree two vertex, let $(\xi_k,\xi_{k+1})$ be the free variables attached below it. Let all free variables $(\xi_{j})_{j>k+1}$ appearing after $\xi_k,\xi_{k+1}$ be fixed. We use the formula \fref{id:formuladegre2omega} to compute $\Omega$, and the bound \fref{bd:degree2resonant} to bound:
$$
\forall (\xi_j)_{j>k+1} \mbox{ with }|\xi_j|\leq \epsilon^{-1}, \quad \# \left\{ |\xi_k|,|\xi_{k+1}|\leq \epsilon^{-1}, \quad \Omega=0\right\}\lesssim \left\{ \begin{array}{l l} \epsilon^{-2}\langle \log \epsilon\rangle \qquad \mbox{for }d=2, \\ \epsilon^{2-2d} \qquad \mbox{for }d\geq 3. \end{array} \right.
$$
\end{itemize}
At the end of the algorithm, we estimate the number of possible values for $\xi_{2n+1}$ by the rough estimate $\epsilon^{-d}$ since $|\xi_{2n+1}|\leq \epsilon^{-1}$, so that one has using \fref{id:ni} that for $d=2$:
$$
N_{\ell,\ell',P}\lesssim \epsilon^{-2}\epsilon^{-n_1}\epsilon^{-2n_2}\langle \log \epsilon \rangle^{n_2}\lesssim \epsilon^{-2n-2}\langle \log \epsilon \rangle^{n},
$$
and for $d\geq 3$:
$$
N_{\ell,\ell',P}\lesssim \epsilon^{-d}\epsilon^{-n_1 (d-1)}\epsilon^{-2n_2 (d-1)} \lesssim \epsilon^{-2n(d-1)-d} 
$$

\end{proof}

A second Lemma finds pairings that maximise the number of completely resonant configurations, so that the upper bound \fref{id:fact1} is saturated. These are the so-called ladder configurations

\begin{lemma}  \label{lem:tgeq1lower}
We keep the notations of Lemma \ref{lem:tgeq1upper}. Then there holds the following property:
\be \label{id:fact2}
\forall \ell'\mbox{ there exists }P, \ \ell\mbox{ and }c>0 \quad  \mbox{such that } \quad N_{\ell,\ell',P,\mathcal R}\geq \left\{ \begin{array}{l l}  c \epsilon^{-2n-2}\langle \log \epsilon \rangle^n \qquad \mbox{for }d=2, \\ c \epsilon^{-2n(d-1)-d} \qquad \mbox{for }d\geq 3. \end{array} \right.
\ee

\end{lemma}

\begin{proof}

Fix an interaction history $\ell'$ for the right graph. Take the interaction history $\ell=\ell'_{mirror}$ for the left graph such that the left graph is obtained by the application to the right graph of the symmetry with respect to the central axis as in the picture below. We take the ladder pairing $P_{ladder}$, that is to say we take the pairing that pairs the $i$-th initial vertex of the left graph with its symmetric image being the $2n+2-i$-th initial vertex of the right graph. Then, the application of the algorithm to construct the minimal spanning tree of Theorem \ref{th:spanning} gives that all vertices of the right graph are of degree zero, and that all that of the left graph are of degree two. The integration of Kirchhoff's laws produce the following. For all $i=1,...,n$, the free variable associated to the edge on the bottom left of the vertex $e_i$ is $\xi_{2i+2}$, that on for the bottom center is $\xi_{2i+1}$, and let us denote by $k_{top}$ that of the edge on top which from Theorem \ref{th:spanning} only depends on $(\xi_k)_{k>2i+2}$. Then:
$$
\Omega_i=-\Omega_i', \quad \Omega_i=|k_{top}-\xi_{2i+1}-\xi_{2i+2}|^2-|\xi_{2i+2}|^2-\sigma_i\left(|k_{top}|^2-|\xi_{2i+1}|^2 \right)
$$
This is illustrated on an example below.
\begin{center}
\includegraphics[width=17cm]{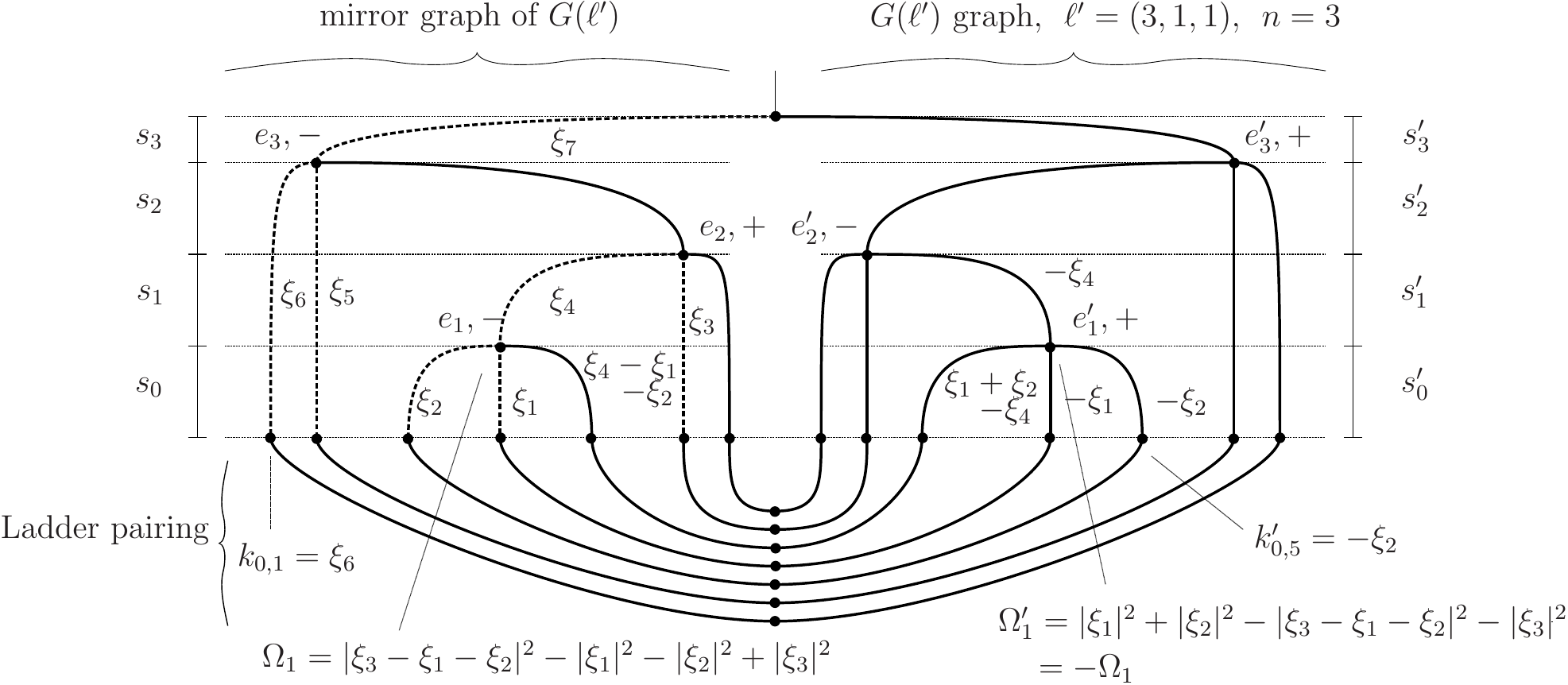}
\end{center}

To control that the addition of momenta we will perform do not take us out of the ball $B(0,\epsilon^{-1})$, we will find a lower bound for:
$$
\tilde N_{\ell'_{mirror},\ell',P_{ladder},\mathcal R}=\# \{ (\xi_i)_{1\leq i \leq 2n+1}\in \mathcal R \mbox{ with } |\xi_{i}|\leq \frac{1}{100^{\lceil \frac i2 \rceil}}\epsilon^{-1}\}
$$

We consider each vertex of the left graph one by one, from bottom to top. At the $i$-th step, let $\xi_k$ for $k>2i+2$ be all fixed satisfying $|\xi_{i}|\leq \frac{1}{100^{\lceil k \rceil}}\epsilon^{-1}$. In particular, this forces $|k_{top}|\leq \frac{1}{100}\frac{1}{100^{\lceil i \rceil}}$. We treat the case $\sigma_i=+1$ without loss of generality. Then
$$
\Omega_i=2\left(\xi_{2i+1}+\xi_{2i+2}\right).\left(\xi_{2i+1}-k_{top}\right)=2\eta.\eta'
$$
where we changed variables $\eta=\xi_{2i+1}+\xi_{2i+2}$ and $\eta'=\xi_{2i+1}-k_{top}$. Hence:
\bee
\# \left\{ |\xi_{2i+1}|,|\xi_{2i+2}|\leq \frac{\ep^{-1}}{100^{i}}, \ \Omega_i=0\right\}& \geq & \# \left\{ |\eta_{2i+1}|,|\eta'_{2i+2}|\leq \frac 13 \frac{\epsilon^{-1}}{100^{i}}, \ \eta.\eta'=0\right\}\\
&\gtrsim & \left\{ \begin{array}{l l} \epsilon^{-2}\langle \log \epsilon \rangle \qquad \mbox{for }d=2, \\ \epsilon^{2-2d}  \qquad \mbox{for }d\geq 3, \end{array} \right.
\eee
where we used the lower bound \fref{bd:degree2resonantlower}. Performing $n$ times this computation for $i=1,...,2n$, then estimating the contribution of the last variable $\xi_{2n+1}$ by $\#\{|\xi_{2n+1}|\leq 100^{-n}\epsilon^{-1} \}\gtrsim \epsilon^{-d}$ produces precisely \fref{id:fact2}.

\end{proof}

We now easily end the proof of Lemma \ref{lem:resonant}.

\begin{proof}[Proof of Lemma \ref{lem:resonant}]

Recall \fref{id:decompositionmathcalF}. On $\mathcal R$, all resonance moduli $\Omega_k$ and $\Omega_k'$ are $0$ so that for any $\ell$, $\ell'$ and $P$:
$$
 \int_{\mathbb R_+^{n+1}\times \mathbb R_+^{n+1}} \prod_{k=1}^{n}e^{-i\Omega_k \sum_{j=0}^{k-1}s_j}\prod_{k=1}^{n}e^{-i\Omega_k' \sum_{j=0}^{k-1}s_j'}\delta \left(t-\sum_{i=0}^{n}s_i \right)\delta \left(t-\sum_{i=0}^{n}s_i' \right) \, d\underline{s} \, d \underline{s}'= C t^{2n}
$$
for some $C>0$. In particular, all terms are nonnegative: $\mathcal F_{\mathcal R}(\ell,\ell',P)\geq 0$ for all $\ell,\ell',P$. This implies that to prove the desired bound of the Lemma, it is sufficient to prove the upper bound in \fref{bd:totallyresnonant} and \fref{bd:totallyresnonant2} for all terms $\mathcal F_{\mathcal R}(\ell,\ell',P)$, and that there exists at least one choice $(\ell,\ell',P)$ such that the lower bound in \fref{bd:totallyresnonant} and \fref{bd:totallyresnonant2} is satisfied by $\mathcal F_{\mathcal R}(\ell,\ell',P)$. Recalling that $T_{kin}=\lambda^{-4}\ep^{-2}$, these two results are exactly those provided by Lemma \ref{lem:tgeq1upper} and Lemma  \ref{lem:tgeq1lower} respectively.

\end{proof}

\subsection{Lower order terms}

In this subsection we examine the non fully resonant terms corresponding to $\mathcal R^c$ in \fref{id:decompositionmathcalF}, and prove Lemma \ref{lem:nonresonant}. We refine the analysis performed in \cite{CG}, extracting an additional gain for the temporal term from the definition of $\mathcal R^c$. We rewrite $\mathcal F_{\mathcal R^c}(\ell,\ell',P)$ in \fref{id:decompositionmathcalF}, once the resolvent identity \fref{bd:resolventidentity} and Theorem \ref{th:spanning} has been applied to determine the free frequencies $(\xi_1,...,\xi_{2n+1})$ as:
\bea
\nonumber \mathcal F_{\mathcal R^c}(\ell,\ell',P)&=&\left(\frac{\lambda^2}{(2\pi)^d}\right)^{2N}\ep^{d(2n+1)} \frac{e^2}{(2\pi)^2}\int_{\mathbb R^2}\sum_{(\xi_1,...,\xi_{2n+1})\in \mathcal R^c} d\alpha d \alpha 'e^{-it(\alpha+\alpha')} \Delta_{\ell,\ell',P}(\underline{k},\underline{k}',0) \\
\label{id:FresolvantmathcalRc} &&\prod_{k=1}^{n} \frac{1}{\alpha-\sum_{i=k}^{n}\Omega_i+\frac{i}{t}}\frac{1}{\alpha'-\sum_{i=k}^{n}\Omega_i'+\frac{i}{t}}\frac{1}{\alpha+\frac{i}{t}}\frac{1}{\alpha'+\frac{i}{t}}  \prod_{\{i,j\}\in P} |A(\epsilon \widetilde k_{0,i})|^2 .
\eea

\begin{proof}[Proof of Lemma \ref{lem:nonresonant}]

Our first step is to localise for $|\alpha|,|\alpha'|\leq  \epsilon^{-K}$ for $K\gg 1$ and to localise the $d\alpha$ and $d\alpha'$ integrals near integers:
\bee
 \left| \mathcal F_{\mathcal R^c}(\ell,\ell',P) \right| \lesssim  \underbrace{\lambda^{4n} \ep^{2d(2n+1)} \int_{\max (|\alpha|,|\alpha'|)\geq \epsilon^{-K}} |...|}_{I}+\underbrace{\lambda^{4n}\ep^{2d(2n+1)}\sum_{a,a'=-\lceil \epsilon^{-K}\rceil}^{\lceil \epsilon^{-K}\rceil} \int_{a-\frac 12}^{a+\frac 12} d\alpha \int_{a'-\frac 12}^{a'+\frac 12} d\alpha' |...|}_{II}.
\eee
Treating $I$ the exact same way as in the proof of Proposition \ref{pr:L2expansion}, we find that it has irrelevant size
\be \label{bd:nonfullyresonantinter1}
I\lesssim \epsilon^{cK}\ll \frac{C}{t} \left(\frac{t^2}{T_{kin}} \right)^n \langle \log \epsilon \rangle^{n+1}
\ee
for some universal $c>0$ and for $K$ large enough. To evaluate $II$ we fix $a,a'$. In order to distinguish wether a term is resonating with the temporal frequencies $a$ and $a'$ or not we do the following decomposition. For all subsets $\mathcal S,\mathcal S'$ of $\{1,...,n\}$ we define:
\bee
\mathcal R^c_{\mathcal S,\mathcal S',a,a'}&:=&\Bigl\{ \sum_{i=k}^n \Omega_i=a\mbox{ for all }k\in \mathcal S, \ \sum_{i=k}^n \Omega_i\neq a\mbox{ for all }k\notin \mathcal S,\\
&&\quad \quad \sum_{i=k}^n \Omega_i'=a'\mbox{ for all }k\in \mathcal S', \mbox{ and } \sum_{i=k}^n \Omega_i'\neq a'\mbox{ for all }k\notin \mathcal S'  \Bigr\},
\eee
with the following convention since $\mathcal R^c$ excludes the totally resonant configuration:
\be \label{id:convention}
\mathcal R^c_{\{1,...,n\},\{1,...,n\},0,0}=\emptyset
\ee
and write:
\bea
\nonumber \int_{a-\frac 12}^{a+\frac 12} d\alpha \int_{a'-\frac 12}^{a'+\frac 12} d\alpha' |...|&\lesssim& \sum_{\mathcal S,\mathcal S'}  \int_{a-\frac 12}^{a+\frac 12} d\alpha \int_{a'-\frac 12}^{a'+\frac 12} d\alpha' \sum_{(\xi_i)_{1\leq i \leq 2n+1}\in \mathcal R^c_{\mathcal S,\mathcal S',a,a'}}  |\Delta_{\ell,\ell',P}|(\underline{k},\underline{k}',0) \\
 \label{id:localisedresolvent}  && \prod_{k=1}^{n} \frac{1}{\left|\alpha-\sum_{i=k}^{n}\Omega_i+\frac{i}{t}\right|}\frac{1}{\left|\alpha'-\sum_{i=k}^{n}\Omega_i'+\frac{i}{t}\right|}\frac{1}{\left|\alpha+\frac{i}{t}\right|}\frac{1}{\left|\alpha'+\frac{i}{t}\right|}
\eea
We distinguish now depending on the values of $a,a'$.

\textbf{Case $a=a'=0$}. We keep the $a$ and $a'$ notation since this will adapt naturally for $(a,a')\neq (0,0)$. In this case, we estimate the sums and integral the following way. We first fix $a-1/2<\alpha<a+1/2$ and $a'-1/2<\alpha'<a'+1/2$. We use an algorithm whose iteration procedure is similar to that of the proof of Lemma \ref{lem:tgeq1upper}. We consider each vertex one by one, starting from the bottom of the right subtree then going to its top, then going from bottom to top of the left subtree. At the $k$-th step on the right graph, for $1\leq k \leq n$, we consider the vertex $v_k'$. Since $\sum_{i=k}^n \Omega_i'$ does only depend on the free variables either below $v_k'$ or appearing after $v_k'$, we can assume that all free variables appearing after $v_k'$ are fixed, and that the one appearing before have already been treated. We do the following:

\emph{Subcase (i) If $k\in \mathcal S'$:} We write 
$$
\frac{1}{\left|\alpha'-\sum_{i=k}^{n}\Omega_i'+\frac{i}{t}\right|}=\frac{1}{|\alpha-a'+\frac{i}{t}|}
$$
We then estimate that:
\begin{itemize}
\item If $v_k'$ is a degree one vertex, let us denote by $\xi_{i_k}$ the free variable attached to it and by $(\xi_{j})_{j>i_k}$ the free variables appearing after $v_k$. Then we write the condition $\sum_{i=k}^{n}\Omega_i'=a'$ as $\Omega_{k}'(\xi_{i_k})=a'-\sum_{i=k+1}^n\Omega_i'$. Note that the right hand side does not depend on $\xi_{i_k}$. We estimate using the formulas \fref{id:formuladegre1omega} and \fref{id:wicknondegeneracy}, and the upper bounds \fref{bd:degree1quadraticresonant} and  \fref{bd:degree1linearresonant}:
$$
\# \left\{\xi_{i_k}\in \mathbb Z^d, \quad  |\xi_{i_k}|\leq \epsilon^{-1}, \mbox{ and } \Omega_{k}'(\xi_{i_k})=a'-\sum_{i=k+1}^n\Omega_i'\right\} \lesssim \epsilon^{1-d}.
$$
The above bound does not hold for only one degenerate case: if $\sigma\sigma'=-1$ and $\zeta=0$ in \fref{id:formuladegre1omega}. In this case, Lemma \ref{lem:wicknondegeneracy} implies that the above set, intersected with the support of $\Delta_{\ell,\ell',P}$, contains only $1\ll \ep^{1-d}$ element so the bound is actually improved.

\item If $v_k'$ is a degree two vertex, let us denote by $\xi_{i_k},\xi_{i_k+1}$ the free variables attached to it and by $(\xi_{j})_{j>i_k+1}$ the free variables appearing after $v_k$. Then we write the condition $\sum_{i=k}^{n}\Omega_i'=a'$ as $\Omega_{k}'(\xi_{i_k},\xi_{i_{k+1}})=a'-\sum_{i=k+1}^n\Omega_i'$. Note that the right hand side does not depend on $\xi_{i_k},\xi_{i_{k}+1}$. We estimate using the formula \fref{id:formuladegre2omega} and the upper bound \fref{bd:degree2resonant}:
$$
\# \left\{\xi_{i_k},\xi_{i_k+1}\in \mathbb Z^d, \quad  |\xi_{i_k}|,|\xi_{i_k+1}|\leq \epsilon^{-1}, \mbox{ and } \Omega_{k}'(\xi_{i_k},\xi_{i_k+1})=a'-\sum_{i=k+1}^n\Omega_i'\right\} \lesssim  \epsilon^{2-2d}\langle \log \epsilon \rangle
$$
\item If $v_k'$ is a degree zero vertex, we do nothing.
\end{itemize}

\emph{Subcase (ii) If $k\notin \mathcal S'$:} Then since $a'-1/2<\alpha'<a'+1/2$, there holds $|\sum_{i=k}^{n}\Omega_i'-\alpha|>1/2$. We then estimate that:
\begin{itemize}
\item If $v_k'$ is a degree one vertex, let us denote by $\xi_{i_k}$ the free variable attached to it and by $(\xi_{j})_{j>i_k}$ the free variables appearing after $v_k$. Note that $\sum_{i=k+1}^n\Omega_i'$ only depends on $(\xi_{j})_{j>i_k}$. We sum over the $\xi_{i_k}$ variable, and estimate using the formulas \fref{id:formuladegre1omega} and \fref{id:wicknondegeneracy} and the two bounds \fref{bd:resolvantnonresonantdeg11} and \fref{bd:resolvantnonresonantdeg12}:
$$
\sum_{|\xi_{i_k}|\leq \epsilon^{-1}, \ |\sum_{i=k}^{n}\Omega_i'-a'|>1/2} \frac{1}{\left|a'- \sum_{i=k}^{n}\Omega_i'+\frac it \right|}\lesssim \epsilon^{1-d}\langle \log \epsilon \rangle.
$$
The above bound does not hold for only one degenerate case: if $\sigma\sigma'=-1$ and $\zeta=0$ in \fref{id:formuladegre1omega}. In this case, the above sum, restricted to the support of $\Delta_{\ell,\ell',P}$, contains only one element from Lemma \ref{lem:wicknondegeneracy}. Hence it is $\leq 2 \ll \epsilon^{1-d}\langle \log \epsilon \rangle$ and the bound is improved.

\item If $v_k'$ is a degree two vertex, let us denote by $\xi_{i_k},\xi_{i_k+1}$ the free variables attached to it and by $(\xi_{j})_{j>i_k+1}$ the free variables appearing after $v_k$. Note that $\sum_{i=k+1}^n\Omega_i'$ does not depend on $\xi_{i_k},\xi_{i_{k}+1}$. We sum over the $\xi_{i_k}$ and $\xi_{i_k+1}$ variables, and estimate using the formula \fref{id:formuladegre2omega} and the bound \fref{bd:resolvantnonresonantdeg2}:
$$
\sum_{|\xi_{i_k}|,|\xi_{i_k+1}|\leq \epsilon^{-1}, \ |\sum_{i=k}^{n}\Omega_i'-a'|>1/2} \frac{1}{\left|a'- \sum_{i=k}^{n}\Omega_i'+\frac it \right|}\lesssim  \epsilon^{2-2d}\langle \log \epsilon \rangle^2 
$$
\item If $v_k$ is a degree zero vertex, we simply upper bound:
$$
\frac{1}{\left|\alpha'-\sum_{i=k}^{n}\Omega_i'+\frac{i}{t}\right|}\leq 2.
$$

\end{itemize}

\emph{End of the algorithm} We treat similarly the left subtree. At the end of the procedure, we estimate the contribution of the last free variable $\xi_{2n+1}$ by the support estimate $\# \{|\xi_{2n+1}|\leq \epsilon^{-1} \}\lesssim \epsilon^{-d}$. At this point we have summed the integrand over all free variables, and found, recalling $a=a'=0$:
\bee
&& \sum_{(\xi_i)_{1\leq i \leq 2n+1}\in \mathcal R^c_{\mathcal S,\mathcal S',a,a'}}  \Delta_{\ell,\ell',P}(\underline{k},\underline{k}',0)\prod_{k=1}^{n} \frac{1}{\left|\alpha-\sum_{i=k}^{n}\Omega_i+\frac{i}{t}\right|}\frac{1}{\left|\alpha'-\sum_{i=k}^{n}\Omega_i'+\frac{i}{t}\right|}\frac{1}{\left|\alpha+\frac{i}{t}\right|}\frac{1}{\left|\alpha'+\frac{i}{t}\right|} \\
&\lesssim &  \frac{1}{|\alpha-a|^{\# \mathcal S}}\frac{1}{|\alpha-a'|^{\# \mathcal S'}}\frac{1}{|\alpha+\frac it|} \frac{1}{|\alpha'+\frac it|}  \epsilon^{-d} \epsilon^{-n_1(d-1)} \epsilon^{-2n_2(d-1)}\langle \log \epsilon\rangle^{n_1+2n_2} \\
&=& \frac{\epsilon^{-2n(d-1)-d}\langle \log \epsilon\rangle^{2n}}{|\alpha-a+\frac it|^{\# \mathcal S+1}|\alpha'-a'+\frac it|^{\# \mathcal S'+1}} ,
\eee
where we used \fref{id:ni}. We then use Lemma \ref{lem:tech3} to evaluate the $\alpha$ and $\alpha'$ integrals:
$$
\int_{a-\frac 12}^{a+\frac 12}\int_{a'-\frac 12}^{a'+\frac 12}d\alpha d\alpha' \frac{\epsilon^{-2n(d-1)-d}\langle \log \epsilon\rangle^{2n}}{|\alpha-a+\frac it|^{\# \mathcal S+1}|\alpha'-a'+\frac it|^{\# \mathcal S'+1}} \lesssim t^{2n-1}\epsilon^{-2n(d-1)-d}\langle \log \epsilon\rangle^{2n+2}
$$
since $\# S+\# S'<2n$ from \fref{id:convention} and since $t\geq 1$.\\

\textbf{Case $a=0$ and $a'\neq 0$}. In this case, we perform the same algorithm as in the case $a=a'=0$. Except that at the moment of estimating the $d\alpha'$ integral, we bound the last fraction in \fref{id:localisedresolvent} that does not involve a cumulated sum of resonance moduli by $\frac{1}{|\alpha'+\frac it|}\leq \frac{1}{a'}$. We thus find the bound in this case:
\bee
\int_{a-\frac 12}^{a+\frac 12} d\alpha \int_{a'-\frac 12}^{a'+\frac 12} d\alpha' \sum_{(\xi_i)_{1\leq i \leq 2n+1}\in \mathcal R^c_{\mathcal S,\mathcal S',a,a'}}  \Delta_{\ell,\ell',P}(\underline{k},\underline{k}',0) &&\\
  \prod_{k=1}^{n} \frac{1}{\left|\alpha-\sum_{i=k}^{n}\Omega_i+\frac{i}{t}\right|}\frac{1}{\left|\alpha'-\sum_{i=k}^{n}\Omega_i'+\frac{i}{t}\right|}\frac{1}{\left|\alpha+\frac{i}{t}\right|}\frac{1}{\left|\alpha'+\frac{i}{t}\right|}& \lesssim & \frac{1}{|a'|} t^{2n-1}\epsilon^{-2n(d-1)-d}\langle \log \epsilon\rangle^{2n}.
 \eee

\textbf{Case $a,a'\neq 0$}. The term is bounded similarly by $\frac{1}{|a|} \frac{1}{|a'|} t^{2n-1}\epsilon^{-2n(d-1)-d}\langle \log \epsilon\rangle^{2n}$. 

\textbf{End of the proof} We sum all the contributions for \fref{id:localisedresolvent} found in the three cases above, using that  $\sum_{a=-\lceil \epsilon^{-K}\rceil}^{\lceil \epsilon^{-K}\rceil}\frac{1}{\langle a \rangle}\lesssim \langle \log \ep \rangle$, and find the bound for $II$:
$$
II\lesssim \lambda^{4n}\epsilon^{d(2n+1)}  t^{2n-1}\epsilon^{-2n(d-1)-d}\langle \log \epsilon\rangle^{2n+1} \lesssim \frac{C}{t} \left(\frac{t^2}{T_{kin}} \right)^n \langle \log \epsilon \rangle^{2n+1}.
$$
The previous bound \fref{bd:nonfullyresonantinter1} for I and the above bound for II prove the desired upper bound.

\end{proof}

\subsection{Two explicit examples} \label{subsubsec:example}

For the sake of clarity, let us detail an example of an interaction diagram and of pairings saturating the bound \fref{lem:tgeq1upper}. We consider the simplest interaction diagram $G=G^*$ of depth $n\geq 1$. A formula for $G^*$ is easy to find from the picture below, so we do not provide it explicitly. There is only one way to order the interaction vertices for this example. Hence, we identify $u_{G*}$ with $u_{\ell^*}$ where $\ell^*$ is the corresponding interaction history:
$$
\ell^*:=\left(2n-1,2n-3,...,3,1 \right).
$$
\begin{center}
\includegraphics[width=13cm]{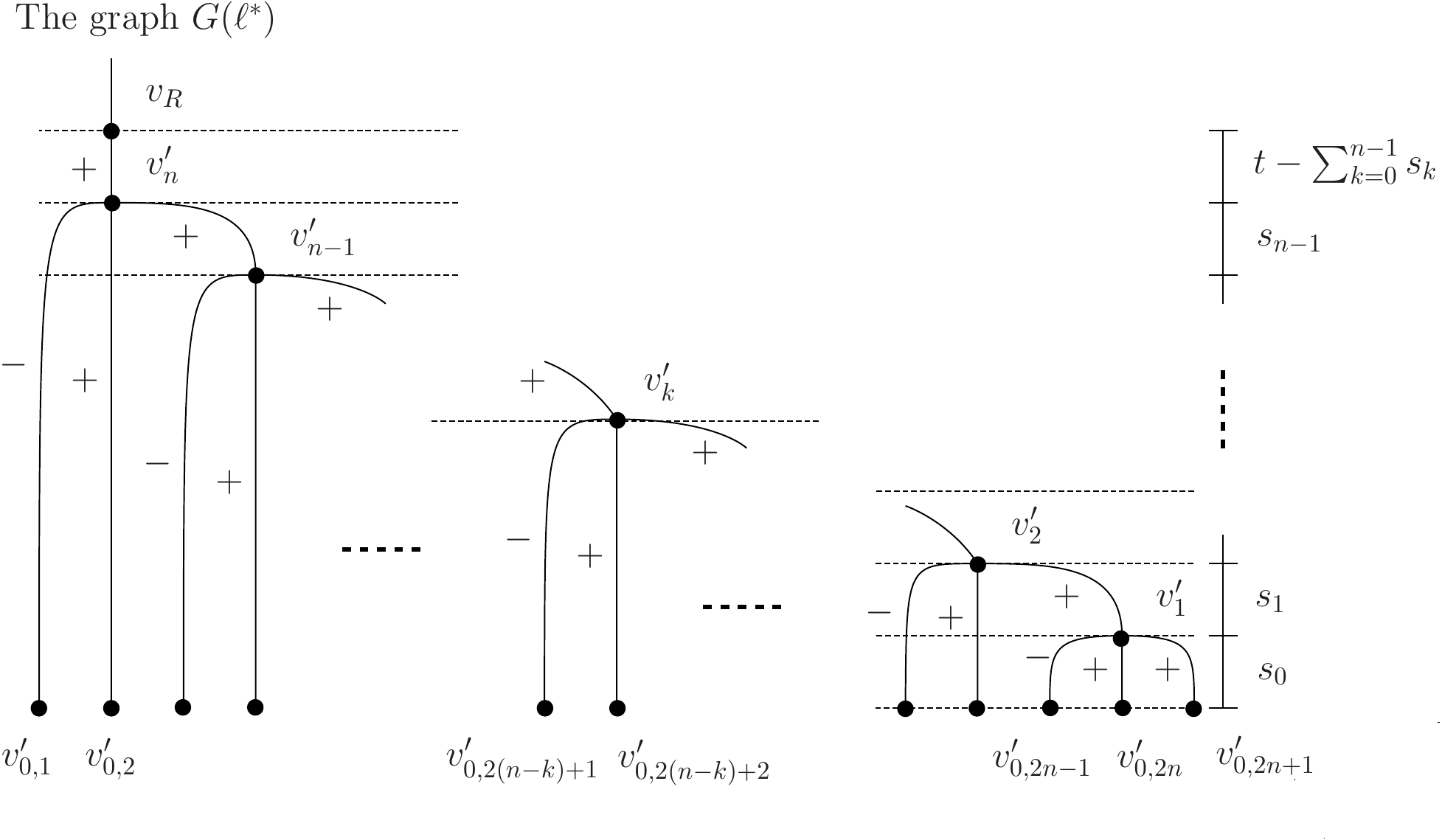}
\end{center}
There is only one interaction history for $\overline{u_{G^*}}$, and the corresponding interaction diagram is:
\begin{center}
\includegraphics[width=9cm]{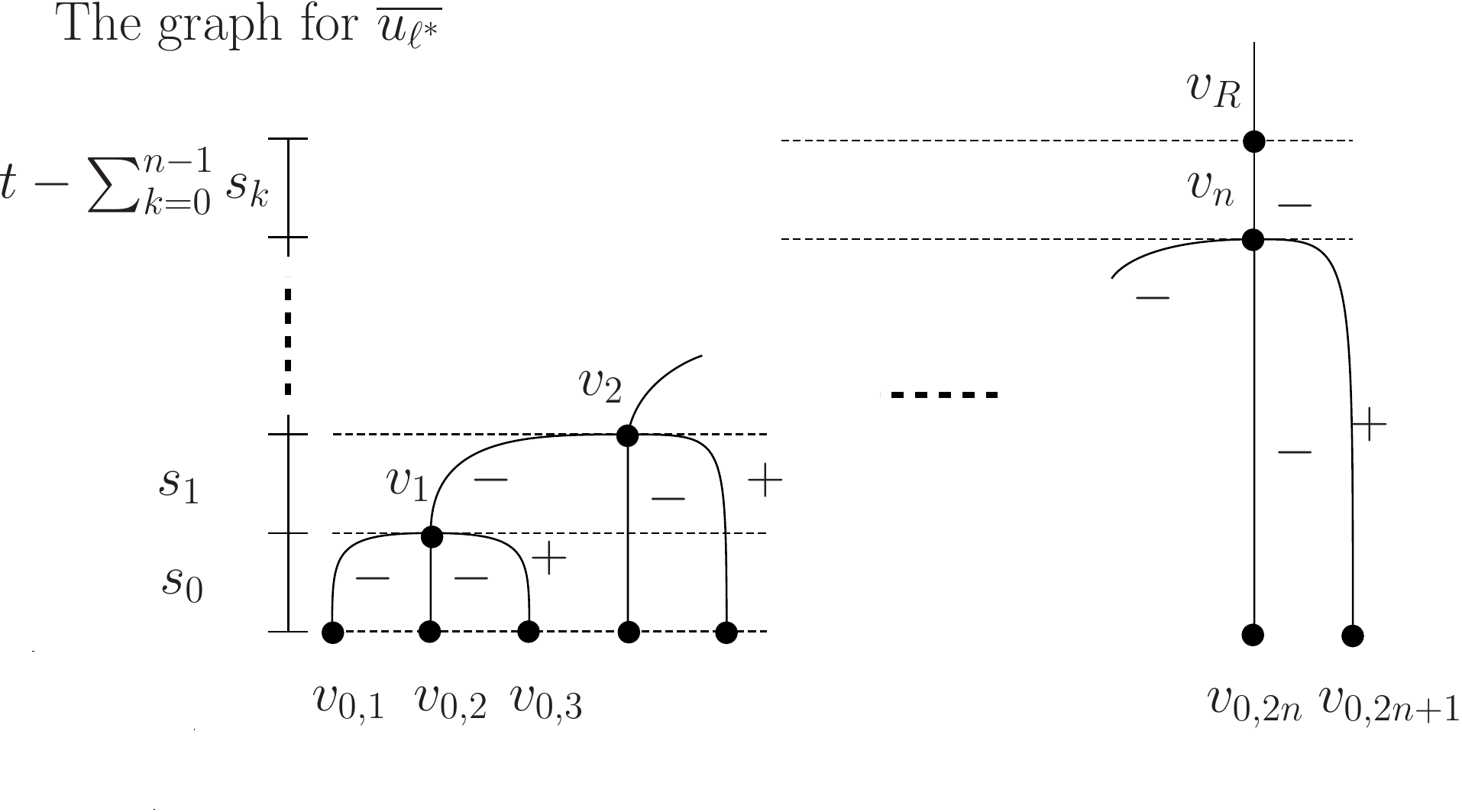}
\end{center}
We now give two very different pairings for the formula \fref{id:wickformula}. The first pairing is the so-called ladder pairing:
$$
P_{ladder}=\cup_{i=1,...,2n+1} \{i,2n+2-i'\}
$$
where the prime notation stands for the initial vertices of the right subtree, and the second one is the "belt" (due to the form of the graph, a long chain, see below):
\begin{eqnarray*}
P_{belt}&:=&\cup_{k=1,...,n-1}\{2k,2k+3\} \cup_{k=1,...,n-1} \{2k-1',2k+2' \} \cup \{1,2n+1'\}\cup \{ 3,2k-1'\} \cup \{2n,2'\}
\end{eqnarray*}
The corresponding paired diagrams, as described in Subsection \ref{subsec:paireddiagrams}, and once the algorithm of Theorem \ref{th:spanning} has been applied to identify the free vertices/momenta, are the following. For the ladder pairing:
\begin{center}
\includegraphics[width=16cm]{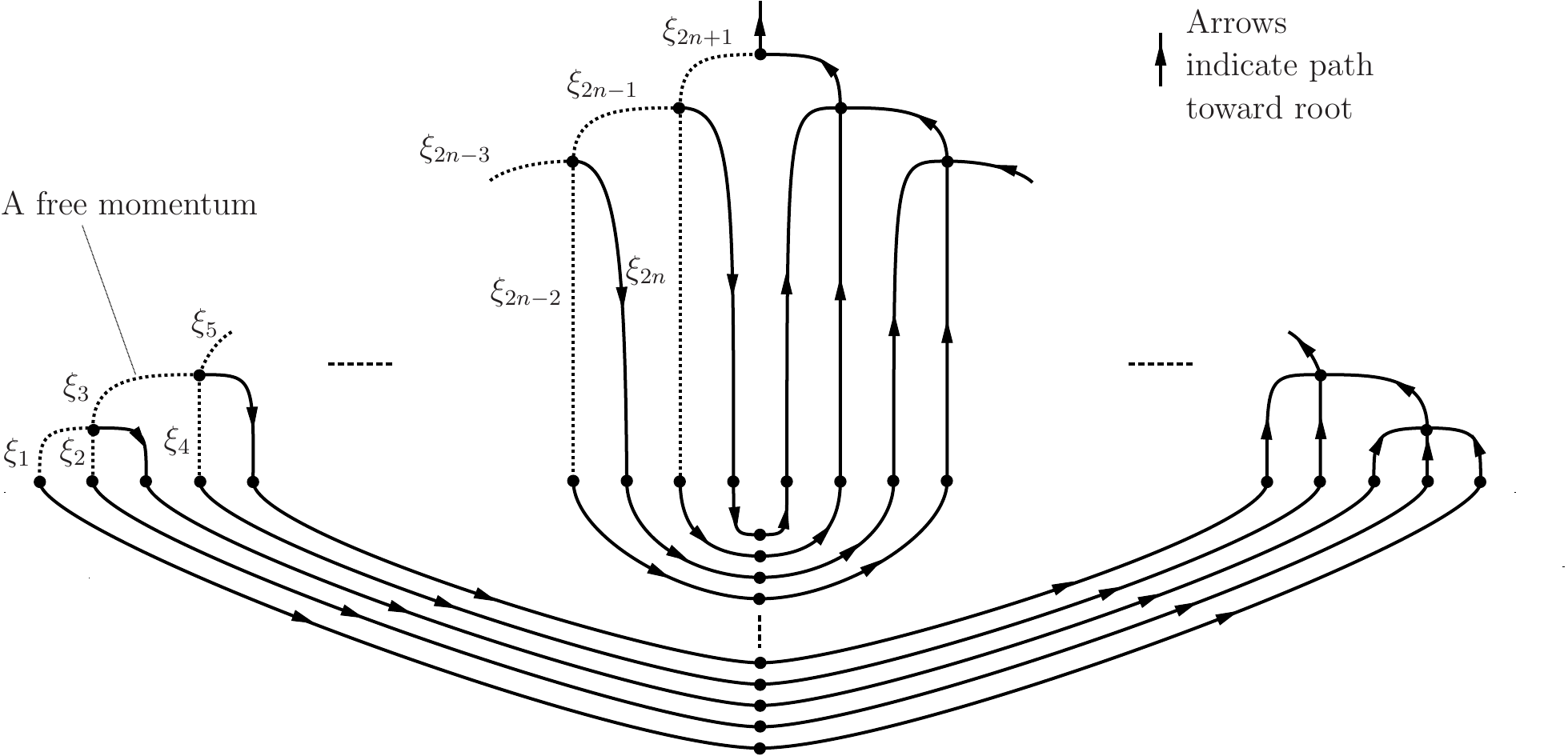}
\end{center}
where we applied the usual notation $( \xi_i )_{1\leq i \leq 2n+1}$ for the free momenta, and for the belt pairing:
\begin{center}
\includegraphics[width=16cm]{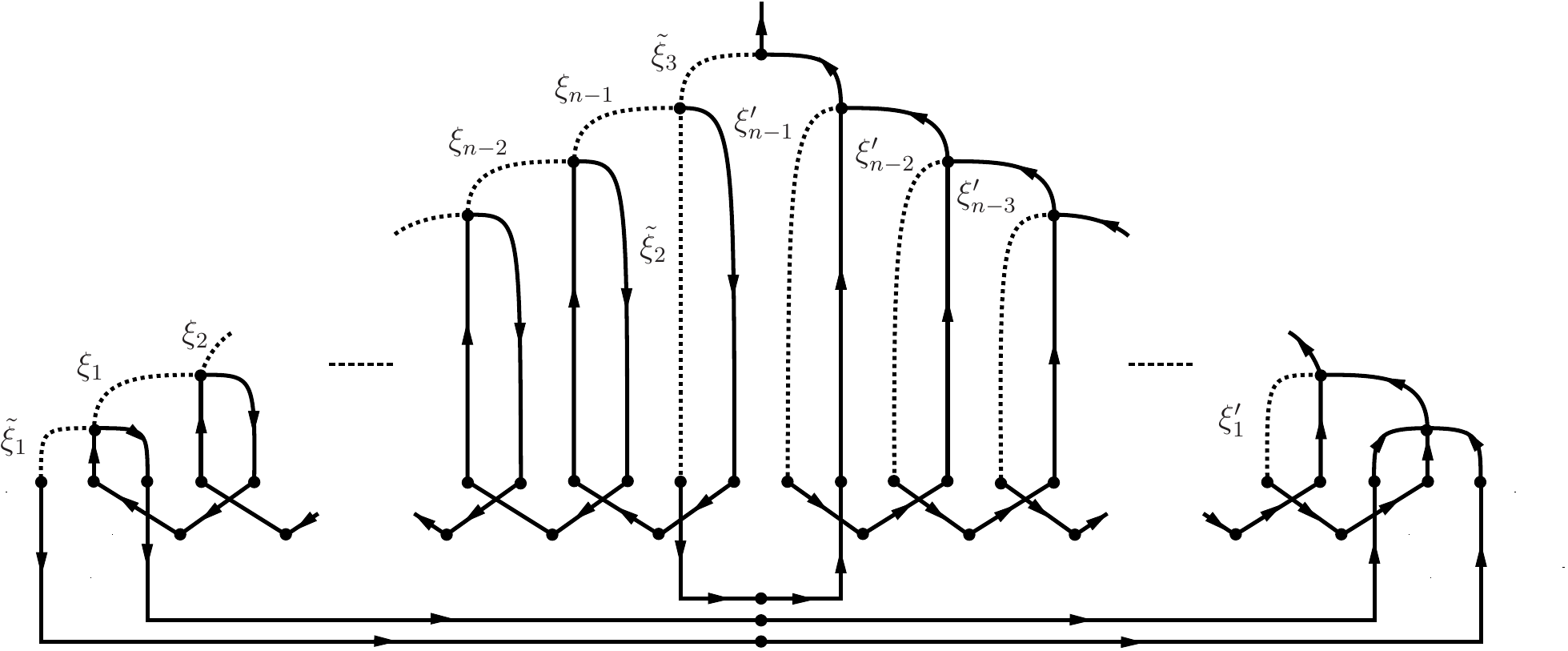}
\end{center}
where we set a specific notation for the free momenta $(\tilde \xi_1,\tilde \xi_2,\tilde \xi_3,\xi_1,...,\xi_{n-1},\xi_{1}',...,\xi_{n-1}')$. Both pairings degenerate for $t\geq 1$, saturating the main part of the bound \fref{bd:totallyresnonant} and \fref{bd:totallyresnonant2}, and one has:

\begin{proposition}[Equivalent bounds for specific examples] \label{pr:tgeq1examples}

For $A$ a cut-off function localising near the origin and for all $n\geq 1$, three constants $0<c<C$ and $C'>0$ exist such that: for all $t \geq C' \langle \log \epsilon \rangle^{n+2}$ for $d=2$:
\be \label{bd:exampleladder}
c \left(\frac{t^2}{T_{kin}}\right)^n \langle \log \epsilon \rangle^n\leq \mathcal F(\overline{\ell^*},\ell^*,P_{ladder}) \leq C \left(\frac{t^2}{T_{kin}}\right)^n \langle \log \epsilon \rangle^n
\ee
and for $d\geq 3$:
\be \label{bd:exampleladder3}
c \left(\frac{t^2}{T_{kin}}\right)^n \leq \mathcal F(\overline{\ell^*},\ell^*,P_{ladder}) \leq C \left(\frac{t^2}{T_{kin}}\right)^n
\ee
and for $t \geq C' \langle \log \epsilon \rangle^{2n+1}$ for all $d\geq 2$ (where $A\approx B$ means $cB\leq A \leq CB$):
\be \label{bd:exampleladder2}
F(\overline{\ell^*},\ell^*,P_{belt}) \approx  \left\{ \begin{array}{l l}  \left(\frac{t^2}{T_{kin}}\right)^n \epsilon^{2n-2}  \qquad \mbox{for }2n<d, \\  \left(\frac{t^2}{T_{kin}}\right)^n \epsilon^{d-2}   \langle \log \epsilon \rangle \qquad \mbox{for }2n=d,\\ \left(\frac{t^2}{T_{kin}}\right)^n \epsilon^{d-2}  \qquad \mbox{for }2n>d . \end{array} \right.
\ee

\end{proposition}

\begin{remark}

Ladder configurations are those expected to give the main contribution. Indeed, as $\Omega_k=\Omega_k'$ for all $1\leq k \leq n$, it is easier to find configurations for the frequencies such that all oscillatory phases are simultaneously resonant. For the belt configuration however, simultaneous resonances should appear less frequently, since all but one vertices have at least a free frequency attached to it. This argument however fails for $t\geq 1$ due to the very structure of the Laplacian. The Proposition suggests that the kinetic wave equation fails to predict the dynamics of $\mathbb E |\hat u(k)|^2$ for $t\geq 1$. In dimension $2$, the corrective factor in \fref{bd:exampleladder2} is not present $\epsilon^{d-2}=1$ and this term has the same size as the ladder configuration.  It is then unclear if an other effective equation holds since since very different pairings contribute with similar polynomial sizes.

\end{remark}

\begin{proof}[Proof of Proposition \ref{pr:tgeq1examples}]

The first bounds \fref{bd:exampleladder} and \fref{bd:exampleladder} are proved in a more general setting in the proof of Lemmas \ref{lem:resonant} and \ref{lem:nonresonant}. The possibility of taking $t\geq C'\langle \log \epsilon \rangle^{n+2} $ comes from a direct check of the proof of Lemma \ref{lem:resonant}, as the power for the log correction in \fref{bd:nonresonant} is $2+n_1+2n_2$ in dimension $2$, $2+n_1+n_2$ in dimension $d\geq 3$, and that for the ladder pairing $n_1=0$ and $n_2=n$.\\

\noindent We now turn to the proof of \fref{bd:exampleladder2}. We give additional details as the belt pairing will be examined later on in this paper.

\textbf{Step 1} \emph{Explicit formula} The resolution of the momenta constraints is the following. There are two specific zones to be considered separately, the middle zone and the extremes on the left and right, but in the middle of each graph a recursive property appears. First, for the middle part:
\begin{center}
\includegraphics[width=14cm]{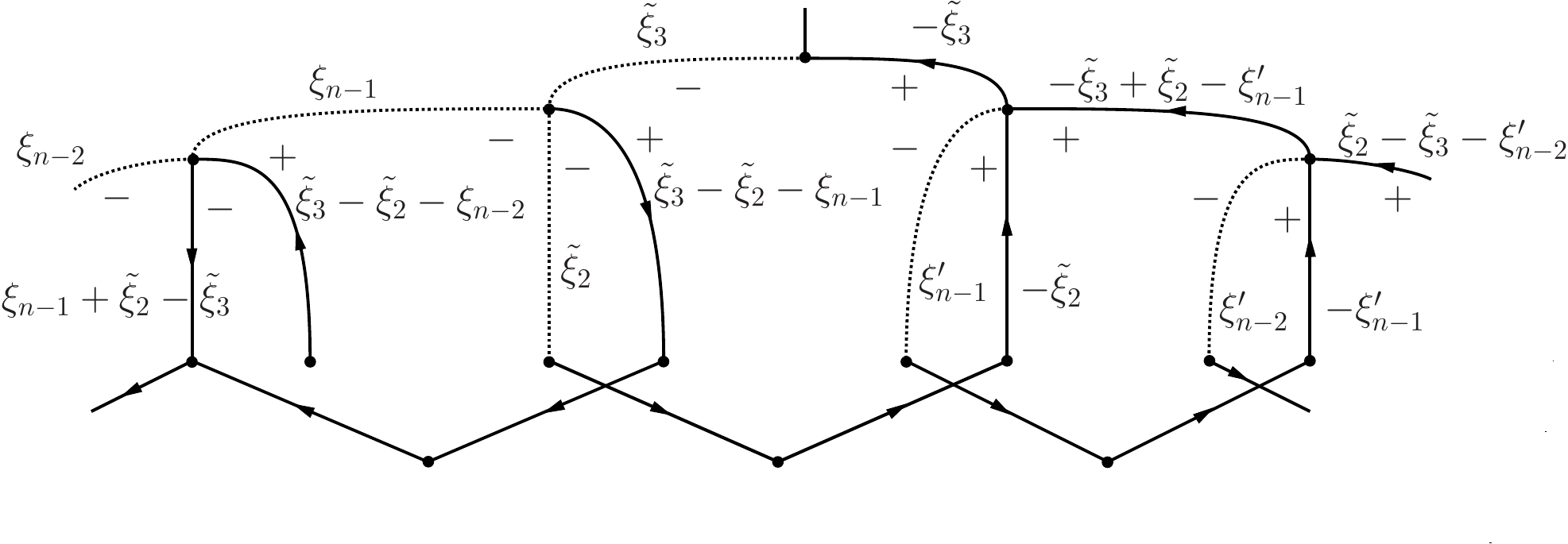}
\end{center}
For the center parts of the left and right subtrees:
\begin{center}
\begin{tabular}{l l}
\includegraphics[width=10cm]{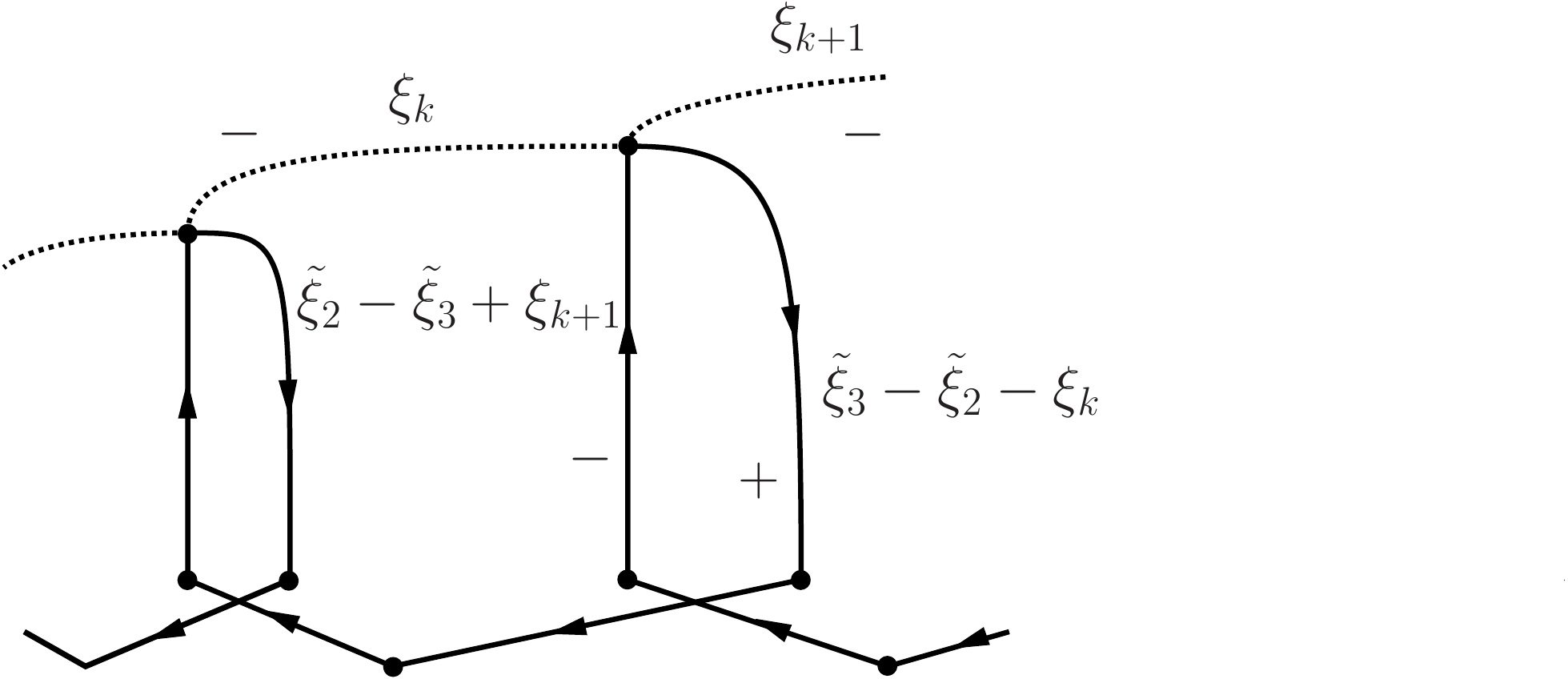} &\includegraphics[width=7.5cm]{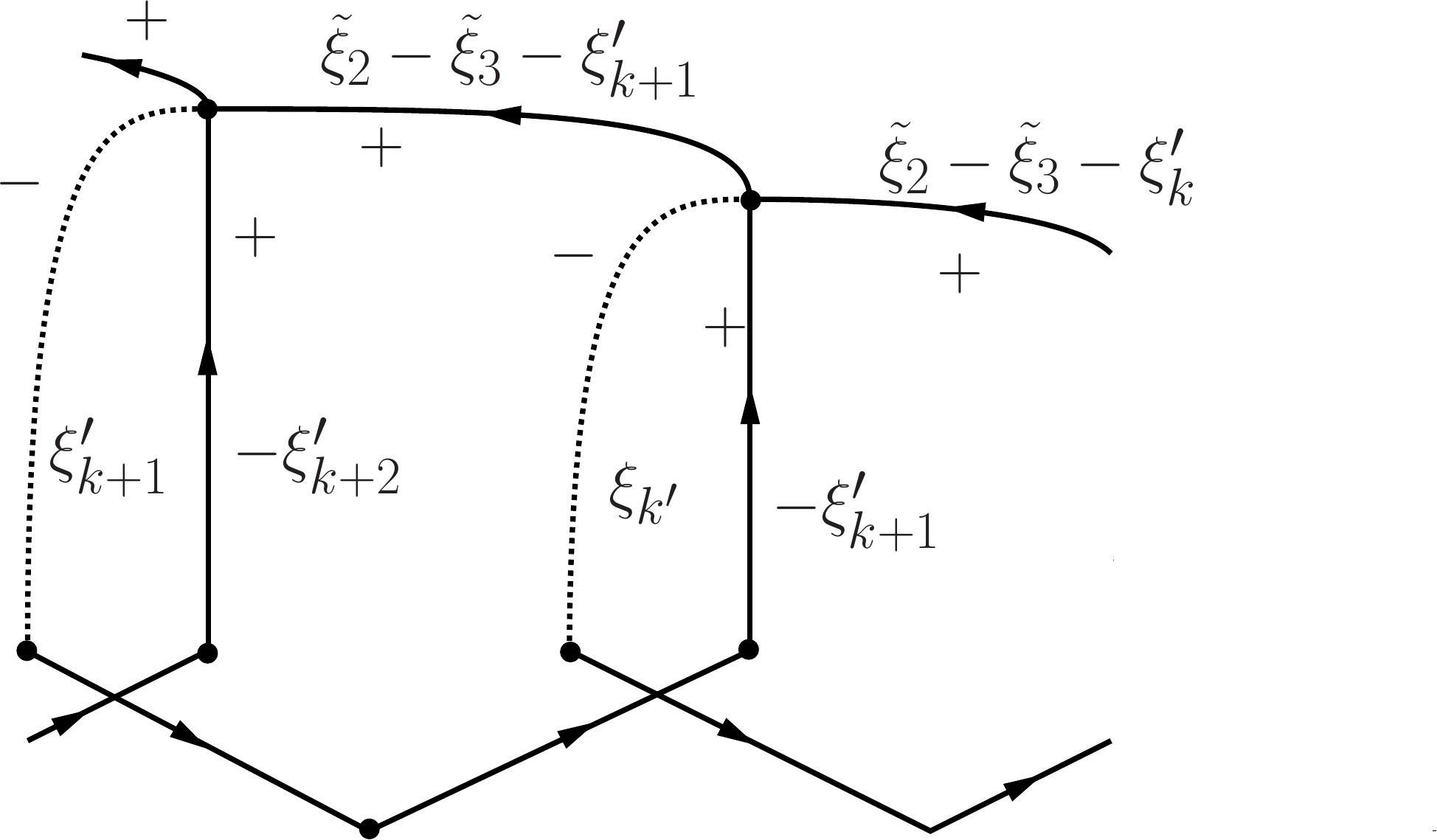}
\end{tabular}
\end{center}
For the extreme left and right parts:
\begin{center}
\includegraphics[width=13cm]{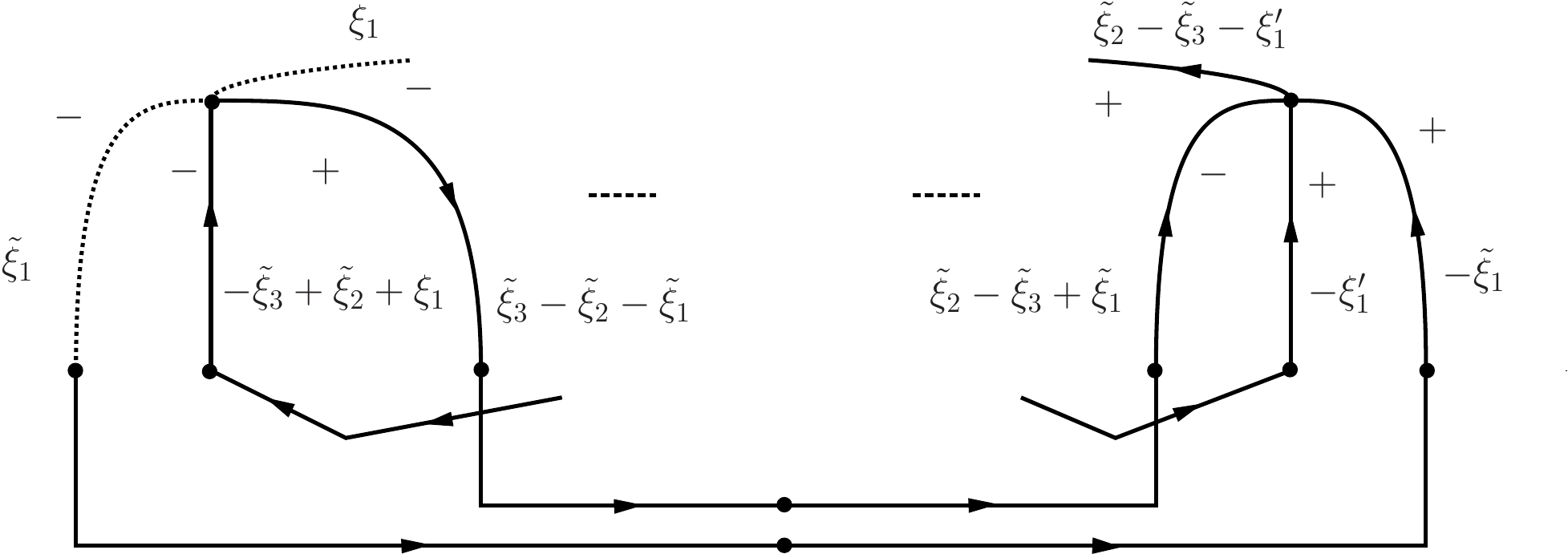}
\end{center}
In particular, expressing resonance moduli as a function of the free momenta one has:
\be \label{id:formulaOmegabelt}
\begin{split}
\Omega_n=2\left(\tilde \xi_3-\tilde \xi_2\right).\left(\tilde \xi_3- \xi_{n-1}\right), \qquad \Omega_n'=2\left(\tilde \xi_3-\tilde \xi_2\right).\left( \xi'_{n-1}-\tilde \xi_2\right),\\
\Omega_{k}=2\left(\tilde \xi_3-\tilde \xi_2\right).\left(\xi_{k}-\xi_{k-1} \right), \qquad \Omega_{k}'=2\left(\tilde \xi_3-\tilde \xi_2\right).\left(\xi_{k-1}'-\xi_{k}' \right),\\
\Omega_{1}=2\left(\tilde \xi_3-\tilde \xi_2\right).\left(\xi_{1}-\tilde \xi_1 \right), \qquad \Omega_1'=2\left(\tilde \xi_3-\tilde \xi_2\right).\left(\tilde \xi_2-\tilde \xi_3+\tilde \xi_1-\xi_1' \right).
\end{split}
\ee
Therefore, one has the following formulas for cumulative resonance moduli for all $1\leq n\leq N$:
$$
\sum_{i=k}^{n}\Omega_i=2\left(\tilde \xi_3-\tilde \xi_2\right).\left\{ \begin{array}{l l} (\tilde \xi_3-\tilde \xi_1) \quad \mbox{for }k=1,\\ (\tilde \xi_3-\xi_{k-1}) \quad \mbox{for } 2\leq k\leq n,  \end{array} \right.
$$
and
$$
\sum_{i=k}^{n}\Omega_i'= 2\left(\tilde \xi_3-\tilde \xi_2\right).\left\{ \begin{array}{l l} (\tilde \xi_1-\tilde \xi_{3}) \quad \mbox{for }k=1, \\ (\xi_{k-1}'-\tilde \xi_2) \quad \mbox{for } 2\leq k\leq n, \end{array} \right.
$$
We write $\zeta=2(\tilde \xi_3-\tilde \xi_2)$ in view of Definition \ref{def:degreeonecases}. Note that in the degenerate case $\zeta=0$, all resonance moduli are zero: $\Omega_k=0=\Omega'_k$ for $k=1,...,n$. In this case however, from Lemma \ref{lem:wicknondegeneracy}, on the support of $\Delta_{\overline{\ell^*},\ell^*}^W$ all free variables $\tilde \xi_1,\xi_1,...,\xi_{n-1},\xi'_{1},...,\xi'_{n-1}$ are fixed linear combinations of $\tilde \xi_2$ and $\tilde \xi_3$. Thus, one finally obtains that, using that $A$ is radially symmetric:
\bea
\nonumber \mathcal F(\overline{\ell^*},\ell^*,P_{belt} )&=& \left(\frac{\lambda^2}{(2\pi)^d}\right)^{2n}\ep^{d(2n+1)}\sum_{(\underline{\xi},\underline{\xi'},\tilde \xi_1,\tilde \xi_2,\tilde \xi_3)\in \mathbb{Z}^{2n+1}} \int_{\mathbb R_+^{2n+2}}e^{-i  2(\tilde \xi_3-\tilde \xi_2).(\tilde \xi_3-\tilde \xi_{1})(s_0-s_0')}\Delta_{\ell,\ell'}^W(\underline{k},\underline{k}')  \\
\nonumber&& \qquad  \prod_{k=1}^{n-1}e^{-i 2 (\tilde \xi_3-\tilde \xi_2).(\tilde \xi_3-\xi_{k})s_k} \prod_{k=1}^{n-1}e^{-i2(\tilde \xi_3-\tilde \xi_2).(\xi_{k}'-\tilde \xi_2) s_k' }A(\ep \tilde \xi_1)A(\ep \tilde \xi_2) A(\ep (\tilde \xi_3-\tilde \xi_2-\tilde \xi_1))  \\
\nonumber&& \qquad \prod_{k=1}^{n-1} A(\ep (\tilde \xi_3-\tilde \xi_2-\xi_k)) \prod_{k=1}^{n-1} A(\ep \xi'_k) \delta(t-\sum_{i=0}^{n}s_i) \delta(t-\sum_{i=0}^{n}s_i')d \underline{s}d \underline{s}'  \\
\label{id:formulaFPbelt} &=& \left(\frac{\lambda^2}{(2\pi)^d}\right)^{2n}\ep^{d(2n+1)}\sum_{(\tilde \xi_1,\tilde \xi_2,\tilde \xi_3)\in \mathbb{Z}^{3d}, \ \zeta\neq 0} A(\ep \tilde \xi_1)A(\ep \tilde \xi_2) A(\ep (\tilde \xi_3-\tilde \xi_2-\tilde \xi_1)) \\
\nonumber && \qquad \qquad \qquad \left| \sum_{\underline \xi \in \mathbb Z^{d(n-1)}} \int_{\mathbb R_+^{n+1}} e^{-i  \zeta \cdot(\tilde \xi_3-\tilde \xi_{1})s_0} \prod_{k=1}^{n-1}e^{-i  \zeta \cdot( \xi_{n}-\tilde \xi_2)s_k} A(\xi_k) \delta(t-\sum_{i=0}^{n}s_i)  \right|^2 \\
\nonumber && +O(\lambda^{4n}\ep^{d2n}t^{2n})
\eea

\textbf{Step 2} \emph{Proof of \fref{bd:exampleladder2}}. We first decompose between fully resonant and non fully resonant configurations as in \fref{id:decompositionmathcalF}. We write $ \tilde \xi_3-\tilde \xi_2=\zeta=d \zeta_v$ where $\zeta_v$ is a visible point from the origin. Note that from the estimate \fref{bd:interortho} and the formulas \fref{id:formulaOmegabelt}, at a fixed value of $\zeta \in B(0,\epsilon^{-1}/2)$:
$$
\# \left\{ \xi_1,...,\xi_{n-1},\xi_1',...,\xi_{n-1}',\tilde \xi_1, \quad (\underline \xi,\underline{\xi'},\tilde \xi_1,\tilde \xi_2,\tilde \xi_3)\in \mathcal R \cap B(0,\epsilon^{-1})\right\} \approx  \left(\frac{\epsilon^{-(d-1)} }{|\zeta_v|}\right)^{2n-1}.
$$
Note that given a visible point $\zeta_v$ with $|\zeta_v|\leq L$, there are $\approx \frac{L}{|\zeta_v|}$ points of the form $k\zeta $ with $k\in \mathbb Z$ in the ball of radius $L$. We thus estimate, using a dyadic partition into multiples of visible points, that for $L=2^i$, $i\geq 1$:
\bee
\sum_{|\zeta|\leq L} \left(\frac{L^{d-1}}{|\zeta_v|}\right)^{2n-1}&\approx& \sum_{j=1}^{i-1} \sum_{2^j\leq |\zeta|< 2^{j+1}, \ \zeta\mbox{ visible}} \frac{L}{2^j}\left(\frac{L^{d-1}}{2^j}  \right)^{2n-1}\approx \sum_{j=1}^{i-1} 2^{j(d-2n)}L^{(d-1)(2n-1)+1}\\
&\approx& \left\{ \begin{array}{l l} L^{2n(d-1)}L^{2-2n} \qquad \mbox{for }2n<d,\\ L^{2n(d-1)}L^{2-d}\langle \log L \rangle \qquad \mbox{for }2n=d,\\ L^{2n(d-1)}L^{2-d} \qquad \mbox{for }2n>d, \end{array} \right.
\eee
where we used that visible points have positive density in $\mathbb Z^d$. As a consequence:
\bee
&&\# \left\{ (\underline{\xi},\underline{\xi'},\tilde \xi_1,\tilde \xi_2, \tilde \xi_3)\in \mathcal R, \ A(\ep \tilde \xi_1)A(\ep \tilde \xi_2)A(\ep (\tilde \xi_3-\tilde \xi_2-\tilde \xi_1)) \prod_{k=1}^{n-1} A(\ep \xi_k) \prod_{k=0}^{n-2} A\xi_k') >0 \right\} \\
&\approx& \left\{ \begin{array}{l l} \epsilon^{-2n(d-1)+2n-2-d} \qquad \mbox{for }2n<d,\\ \epsilon^{-2n(d-1)-2} \langle \log \epsilon\rangle \qquad \mbox{for }2n=d,\\ \epsilon^{-2n(d-1)-2} \qquad \mbox{for }2n>d. \end{array} \right.
\eee
Note that on $\mathcal R$ the time integral is equal to $(\frac{t^n}{n!})^2$. Hence in the computation of $\mathcal F_{P_{belt}}$ for the fully resonant term in the decomposition \fref{id:decompositionmathcalF}:
$$
\mathcal F_{\mathcal R}(\overline{\ell^*},\ell^*,P_{belt} ) \approx   \left\{ \begin{array}{l l}  \left(\frac{t^2}{T_{kin}}\right)^n \epsilon^{2n-2}  \qquad \mbox{for }2n<d, \\  \left(\frac{t^2}{T_{kin}}\right)^n \epsilon^{d-2}   \langle \log \epsilon \rangle \qquad \mbox{for }2n=d,\\ \left(\frac{t^2}{T_{kin}}\right)^n \epsilon^{d-2}  \qquad \mbox{for }2n>d . \end{array} \right.
$$
The non fully resonant term in the decomposition \fref{id:decompositionmathcalF} is estimated as in the proof of Lemma \ref{lem:nonresonant}, with the use of the bound \fref{bd:resolvantnonresonantdeg12} to obtain a $\frac{1}{|\zeta_v|^{2n-1}}$ factor, that is then estimated as above. We do not provide the details here. This leads to the estimate:
$$
\mathcal F_{\mathcal R^c}(\overline{\ell^*},\ell^*,P_{belt} ) \lesssim   \left\{ \begin{array}{l l} \frac{1}{t} \left( \frac{t^2}{T_{kin}}\right)^n \epsilon^{2n-2} \langle \log \epsilon \rangle^{2n+1} \qquad \mbox{for }2n<d, \\ \frac{1}{t} \left(\frac{t^2}{T_{kin}}\right)^n  \epsilon^{d-2}\langle \log \epsilon \rangle^{2n+2}  \qquad \mbox{for }2n=d,\\ \frac{1}{t}\left(\frac{t^2}{T_{kin}}\right)^n \epsilon^{d-2} \langle \log \epsilon \rangle^{2n+1} \qquad \mbox{for }2n>d . \end{array} \right.
$$
The two above bounds imply the bound \fref{bd:exampleladder2}.

\end{proof}

\section{Failure of convergence for small kinetic time for any dispersion relation}

\label{sectionsmall}

This subsection is devoted to the proof of Proposition \ref{pr:tleq1}. The special graph we will consider is $G_n=G^*$ described in Subsubsection \ref{subsubsec:example}. There, only fully resonant configurations ($\Omega=0$ at each vertex) contributed to the norm of $u_G$. Here however, the contribution comes from nearly fully resonant configurations ($\Omega \approx 0$ at each vertex). Paired diagrams which maximise this number of configurations involve degree one vertices with special properties, linked to more refined topological properties of the graph. One therefore needs to refine further the algorithm used to compute $\mathcal F(\ell,\ell',P)$ in the previous Subsection.

Recall Definition \ref{def:degreeonecases} for the types of degree one vertices. Note in particular that from the identity \fref{id:formulaFPbelt}, the belt pairing is such that the corresponding paired graph has $2n-2$ degree one vertices of linear type with the \emph{same} dual variable $\zeta$, one degree zero vertex and one degree two vertex.

\subsection{Key contribution: the belt pairing}

Recall the identity \fref{LpLpexpression}. We claim that the worst contribution comes from the belt pairing already examined in Subsubsection \ref{subsubsec:example}. We refer to this Subsection for the notations, in particular to Step 1 of the proof of Proposition \ref{pr:tgeq1examples}.

\begin{lemma}

Let $d\geq 2$ and $A$ be any positive definite symmetric $d\times d$ and take the corresponding dispersion relation \fref{def:H}. Fix any $0<\nu \ll 1$ small, and let $\epsilon^{1-\nu }\leq t\leq \epsilon^{1-\frac 1d}$. Then for any $n\in \mathbb N$ with $2n\geq d+2$, there exist two constants $0<c<C$ such that for the only interaction histories $\ell^*$ and $\overline{\ell^*}$ associated with the interaction diagram $G^*$, and the belt pairing $P_{belt}$:
\be \label{bd:mathcalFPbelt}
\frac{c}{T_{kin}^{n}}\ep^{d-1} t \leq \mathcal F(\overline{\ell^*},\ell^*,P_{belt}) \leq \frac{C}{T_{kin}^{n}}\ep^{d-1} t.
\ee

\end{lemma}

\begin{proof}

Fix $0<\kappa \ll 1$. The identity \fref{id:formulaFPbelt} was derived for the Laplacian dispersion relation $H=\Id$, but easily adapts for general dispersion relations \fref{def:H}: the variable $\zeta$ is now defined as $\zeta=2H(\tilde \xi_3-\tilde \xi_2)\in 2H\mathbb Z^d$. We decompose \fref{id:formulaFPbelt} according to the size of $\zeta$:
\be \label{id:decompositionbelt}
\mathcal F_{P_{belt}} = \underbrace{\left(\frac{\lambda^2}{(2\pi)^d}\right)^{2n}\ep^{d(2n+1)}\sum_{|\zeta|\leq \ep^{\kappa}t^{-1}}[...]}_{I} \qquad+\qquad \underbrace{\left(\frac{\lambda^2}{(2\pi)^d}\right)^{2n}\ep^{d(2n+1)}\sum_{|\zeta|\geq \ep^{\kappa}t^{-1}}[...]}_{II}+O(\lambda^{4n}\ep^{d2n}t^{2n})
\ee

\textbf{Step 1} \emph{$|\zeta| \epsilon^\kappa t^{-1} )$ contribution.} We fix $1\leq |\zeta|\leq \epsilon^\kappa t^{-1} $ in \fref{id:formulaFPbelt}. We first use continuum approximation to transform the sums into integrals. For fixed $\tilde \xi_2$ and $s\leq t$, by Poisson summation formula, denoting by $a$ the inverse Fourier transform of $A$:
\bee
\sum_{\xi \in \mathbb Z^d} e^{-i  \zeta \cdot( \xi-\tilde \xi_2)s} A(\xi) &=&e^{i\zeta \cdot\tilde \xi_2 s}  \frac{(2\pi)^{\frac d2}}{\ep^d}a\left(\frac{\zeta }{\ep}s\right)+  e^{i\zeta \cdot\tilde \xi_2 s} \sum_{y\in \mathbb (2\pi \mathbb Z^d)\backslash\{0\}} \frac{(2\pi)^{\frac d2}}{\ep^d}a\left(\frac{\zeta}{\ep}s+\frac{y}{\ep}\right)\\
&=& \int_{\xi \in \mathbb R^d} e^{-i  \zeta \cdot( \xi-\tilde \xi_2)s} A(\ep \xi)+O(\ep^{\infty}),
\eee
where we used the fact that $|s\tilde \xi |\leq \epsilon^\kappa$ and that $a$ is Schwartz. We integrate in time with respect to the $s_0$ variable, yielding
\bee
&& \int_{\mathbb R_+^{n-1}}\int_0^{t-\sum_{k=1}^{n-1}s_k} e^{-i\zeta \cdot(\tilde \xi_3-\tilde \xi_1)s_0}ds_0 \prod_{k=1}^{n-1} e^{-i\zeta \cdot(\xi_k-\tilde \xi_2)s_k} {\bf 1}(\sum_{k=1}^{n-1}s_n\leq t)ds_1...ds_{n-1} \\
&=&\frac{1}{-i\zeta \cdot(\tilde \xi_3-\tilde \xi_1)}\Bigl[ e^{-it\zeta \cdot(\tilde \xi_3-\tilde \xi_1)}\int_{\mathbb R_+^{n-1}}  \prod_{k=1}^{n-1} e^{-i\zeta \cdot(\xi_k-\tilde \xi_2+\tilde \xi_1-\tilde \xi_3)s_k} {\bf 1}(\sum_{k=1}^{n-1}s_n\leq t)ds_1...ds_{n-1}\\
&& \qquad \qquad \qquad \qquad \qquad \qquad \qquad  -\int_{\mathbb R^{n-1}_+}  \prod_{k=1}^{n-1} e^{-i\zeta \cdot(\xi_k-\tilde \xi_2)s_k} {\bf 1}(\sum_{k=1}^{n-1}s_k\leq t)ds_1...ds_{n-1}\Bigr].
\eee
Applying the above identity and Fourier transformation:
\bee
&& \int_{\underline \xi \in \mathbb Z^{d(n-1)}} \int_{\mathbb R_+^{n+1}} e^{-i  \zeta \cdot(\tilde \xi_3-\tilde \xi_{1})s_0} \prod_{k=1}^{n-1}e^{-i  \zeta \cdot( \xi_{k}-\tilde \xi_2)s_k} A(\ep \xi_n) \delta(t-\sum_{i=0}^{n}s_i) \,d\underline{\xi}\\
 &=&\frac{(2\pi)^{\frac d2(n-1)}\ep^{-d(n-1)}}{-i\zeta \cdot(\tilde \xi_3-\tilde \xi_1)}\Bigl[ e^{-it\zeta \cdot(\tilde \xi_3-\tilde \xi_1)}\int_{\mathbb R^{n-1}_+}  \prod_{k=1}^{n-1} e^{-i\zeta \cdot(-\tilde \xi_2+\tilde \xi_1-\tilde \xi_3)s_k}a\left(\frac{\zeta s_k}{\ep}\right) {\bf 1}(\sum_{k=1}^{n-1}s_n\leq t)d\underline{s}\\
&& \qquad \qquad \qquad \qquad \qquad \qquad \qquad \qquad - \int_{\mathbb R^{n-1}_+}  \prod_{k=1}^{n-1} e^{-i\zeta \cdot(-\tilde \xi_2)s_k}a\left(\frac{\zeta s_k}{\ep}\right) {\bf 1}(\sum_{k=1}^{n-1}s_k\leq t)d\underline{s}\Bigr].
\eee
Note that if $s_1,...,s_{n-1}$ are such that $0\leq s_1,...,s_{n-1} \leq t$ and $s_1+...+s_{n-1}>t$, then at least one of the $s_k$'s is such that $s_k\gtrsim t$, implying $a(\zeta \ep^{-1}s_k)=O(\ep^{\infty})$ as $\ep t\geq \ep^{-\nu}\gg 1$ and $a$ is Schwartz. We combine the previous sum/integral approximation relying on Poisson summation, the above partial integration with $s_0$, and this remark to simplify the time integration domain:
\bee
&& \sum_{\underline \xi \in \mathbb Z^{d(n-1)}} \int_{\mathbb R_+^{n+1}} e^{-i  \zeta \cdot(\tilde \xi_3-\tilde \xi_{1})s_0} \prod_{k=1}^{n-1}e^{-i  \zeta \cdot( \xi_{k}-\tilde \xi_2)s_k} A(\ep \xi_k) \delta(t-\sum_{i=0}^{n}s_i)\\
 &=&\frac{(2\pi)^{\frac d2(n-1)}\ep^{-d(n-1)}}{-i\zeta \cdot(\tilde \xi_3-\tilde \xi_1)}\Bigl[ e^{-it\zeta \cdot(\tilde \xi_3-\tilde \xi_1)}\int_{\mathbb R^{n-1}}  \prod_{k=1}^{n-1} e^{-i\zeta \cdot(-\tilde \xi_2+\tilde \xi_1-\tilde \xi_3)s_k}a\left(\frac{\zeta s_k}{\ep}\right) {\bf 1}( s_1,...,s_{n-1}\geq 0)d\underline{s}\\
&& \qquad \qquad \qquad \qquad  \qquad \qquad- \int_{\mathbb R^{n-1}}  \prod_{n=1}^{n-1} e^{-i\zeta \cdot(-\tilde \xi_2)s_k}a\left(\frac{\zeta s_k}{\ep}\right) {\bf 1}(s_1,...,s_{n-1}\geq 0)d\underline{s}\Bigr]+O(\ep^{\infty})
\eee
We introduce the following function:
$$
F[\theta](\xi) =\left(\int_0^\infty e^{i\theta \cdot \xi s} a(\theta s)ds\right)^{N-1}
$$
(which given $|\theta|=1$ as a parameter is a smooth function of $\xi$). The above identity becomes, by Fubini, $s$ time renormalisation, introducing $\bar \xi_1=\tilde \xi_3-\tilde \xi_1$ and denoting $\theta=\frac{\zeta}{|\zeta|}$:
$$
...=\left(\frac{(2\pi)^{\frac d2}}{|\zeta|\ep^{d-1}}\right)^{n-1}\frac{1}{-i\zeta \cdot\bar \xi_1}\left[ e^{-it\zeta \cdot\bar \xi_1}F[\theta](\ep(\tilde \xi_2+\bar \xi_1))-F[\theta](\ep \tilde \xi_2)\right] +O(\ep^{\infty})
$$
Therefore we obtain the following expression for the first quantity in \fref{id:decompositionbelt}:
\bea
\nonumber I &= \frac{\lambda^{4n}}{(2\pi)^d}\ep^{2n+3d-2}\sum_{\zeta \in 2A\mathbb Z^d, \ |\zeta|\leq \ep^{\kappa}t^{-1}}\sum_{(\bar \xi_1,\tilde \xi_2)\in \mathbb{Z}^{2d}} \frac{1}{|\zeta|^{2(n-1)}}A( \frac{\ep}{2}A^{-1}\zeta+\ep \tilde \xi_2-\ep \bar \xi_1)A(\ep \tilde \xi_2) A(-\ep \tilde \xi_2+\ep \bar \xi_1)) \\
\label{id:intercalculbelt} & \qquad \qquad\qquad \qquad \frac{1}{|\zeta \cdot\bar \xi_1|^2}\left| e^{-it\zeta \cdot\bar \xi_1}F[\theta](\ep(\tilde \xi_2+\bar \xi_1))-F[\theta](\ep \tilde \xi_2)\right|^2 +O(\ep^{\infty})
\eea
We decompose:
$$
e^{-it\zeta \cdot\bar \xi_1}F[\theta](\ep(\tilde \xi_2+\bar \xi_1))-F[\theta](\ep \tilde \xi_2)=(e^{-it\zeta \cdot\bar \xi_1}-1)F[\theta](\ep(\tilde \xi_2+\bar \xi_1))+F[\theta](\ep(\tilde \xi_2+\bar \xi_1)-F[\theta](\ep \tilde \xi_2).
$$
We estimate that, uniformly in $\tilde \xi_2$ and $\zeta$:
\be \label{bd:intercalculbelt}
\sum_{\bar \xi_1\in \mathbb{Z}^{d}} A( \frac \ep2 A^{-1} \zeta+\ep \tilde \xi_2-\ep \bar \xi_1) A(\ep \tilde \xi_2) A(-\ep \tilde \xi_2+\ep \bar \xi_1)) \left|\frac{F[\theta](\ep(\tilde \xi_2+\bar \xi_1)]-F[\theta](\ep \tilde \xi_2)]}{\zeta \cdot\bar \xi_1}\right|^2 =O\left(\frac{\ep^{2-d}}{|\zeta|^2}\right),
\ee
which is because $\left|F[\theta](\ep(\tilde \xi_2+\bar \xi_1)]-F[\theta](\ep \tilde \xi_2)]\right|\lesssim \frac{\ep}{|\zeta|} |\bar \xi_1.\zeta|$ from the smoothness of $F$, and that the integrand has support in the zone $|\bar \xi_1|\lesssim \ep^{-1}$. We also estimate that, using $|e^{i\omega}-1|^2=4\sin^2(\omega/2)$:
\bee
&& \sum_{\bar \xi_1\in \mathbb{Z}^{d}} A\left( \frac \ep 2 A^{-1}\zeta+\ep \tilde \xi_2-\ep \bar \xi_1\right) A(\ep \tilde \xi_2) A(-\ep \tilde \xi_2+\ep \bar \xi_1) \left|\frac{e^{-it\zeta \cdot\bar \xi_1}F[\theta](\ep(\tilde \xi_2+\bar \xi_1)]-F[\theta](\ep(\tilde \xi_2+\bar \xi_1)]}{\zeta \cdot\bar \xi_1}\right|^2\\
&=&  \int_{\bar \xi_1\in \mathbb{Z}^{d}} A\left( \frac \ep 2 A^{-1} \zeta+\ep \tilde \xi_2-\ep \bar \xi_1 \right) A(\ep \tilde \xi_2) A(-\ep \tilde \xi_2+\ep \bar \xi_1)  \frac{4\sin^2\left( \frac{t\zeta \cdot\bar \xi_1}{2} \right)}{|\zeta \cdot\bar \xi_1|^2}\left| F[\theta](\ep(\tilde \xi_2+\bar \xi_1)] \right|^2+O(t^2\ep^{1-d})\\
\eee
where the second equality approximates the Riemann sum by an integral (recall that $t\ep \gg 1$). We renormalise the variable $\bar \xi_1$, use the convergence of $\lambda \frac{\sin(\frac{\omega}{\lambda})^2}{\omega^2}=\pi \delta(\omega=0)+O(\lambda)$ as $\lambda \rightarrow 0$ as a distribution to obtain:
$$
...= \frac{2\pi t\ep^{1-d}}{|\zeta|} \int_{\bar \xi_1\in \mathbb{R}^{d}} A\left( \frac \ep 2 A^{-1} \zeta+\ep \tilde \xi_2- \bar \xi_1\right) A(\ep \tilde \xi_2) A(-\ep \tilde \xi_2+ \bar \xi_1) \left| F[\theta](\ep\tilde \xi_2+\bar \xi_1) \right|^2\delta(\theta.\bar \xi_1=0) +O(\ep^{2-d})
$$
where we used $\ep^{1-\nu}\leq t \leq \ep^{1/2}$. We inject the above equality and the inequality \fref{bd:intercalculbelt} in \fref{id:intercalculbelt} using Cauchy-Schwarz and $t\gg \ep$, and obtain:
\begin{align*}
I &= \frac{\lambda^{4n}}{(2\pi)^d}\ep^{2n+2d}\sum_{|\zeta|\leq \ep^{\kappa}t^{-1}} \sum_{\tilde \xi_2\in \mathbb Z^d}\frac{1}{|\zeta|^{2(n-1)}} \\
& \left(\frac{2\pi t\ep^{-1}}{|\zeta|} \int_{\bar \xi_1 \in \mathbb{R}^{d}} A( \frac \ep 2 A^{-1}\zeta+\ep \tilde \xi_2- \bar \xi_1)A(\ep \tilde \xi_2) A(-\ep \tilde \xi_2+ \bar \xi_1)\left| F[\theta](\ep(\tilde \xi_2+\bar \xi_1)] \right|^2\delta(\theta.\bar \xi_1=0)+O\left(t^{\frac 12}\ep^{-\frac 12}\right) \right) \\
&= \frac{\lambda^{4n}}{8\pi^3}\ep^{2n+d-1}t \underset{|\zeta|\leq \ep^\kappa t^{-1}}{\sum}\underset{(\bar \xi_1,\tilde \xi_2)\in \mathbb{R}^{2d}}{\int} \frac{1}{| \zeta|^{2n-1}}A\left( \frac \ep2 A^{-1}\zeta+ \tilde \xi_2- \bar \xi_1\right)A( \tilde \xi_2) A(- \tilde \xi_2+ \bar \xi_1) \left| F^{n-1}[\theta](\tilde \xi_2) \right|^2\delta(\theta.\bar \xi_1=0) \\
& +O( \lambda^{4n}\ep^{2n+d-\frac 12}t^{\frac 12} ) \\
& \approx   T_{kin}^{-n}\ep t.
\end{align*}

\textbf{Step 2} \emph{The $\zeta \gtrsim t^{-1}$ contribution}. To treat the $II$ term in , we apply first the resolvant identity \fref{id:FresolvantmathcalRc}, and then the algorithm described in the proof of Lemma \ref{lem:nonresonant} to estimate iteratively the sums of $\frac{1}{\left|\alpha-\sum_{i=k}^{n}\Omega_i+\frac{i}{t}\right|}$ terms (or involving $\alpha'$). At the first degree zero vertex we bound $\frac{1}{\left|\alpha'-\sum_{i=1}^{n}\Omega_i+\frac{i}{t}\right|}\leq t$. Then at each linear degree one vertex we use the estimate \fref{bd:degre1lineaire} with $a=0$ and with the fact that $\max (\frac{1}{|\zeta|},t)\leq \epsilon^{-\kappa}t$. Finally, at the last degree two vertex with use the estimate \fref{bd:degre2} with $a=0$. This implies that:
$$
II\lesssim \ep^{-C\kappa }\lambda^{4n}\ep^{d(2n+1)} t(\ep^{-(d-1)-\kappa}t)^{2n-2}\ep^{-2(d-1)}\ep^{-d}=\lambda^{4n}\ep^{2n}t^{2n-1} \ep^{-C(n)\kappa} \ll T_{kin}^{-n}\ep^{d-1} t
$$
as $2n\geq d+2$ and $t\leq \epsilon^{1-\frac 1d}$ and $n\geq 3$. The above estimate, together with the estimate for I found in Step 1, yield the desired estimate.

\end{proof}

\subsection{Lower order terms}

In the identity \fref{LpLpexpression} we found a sharp bound for the $\mathcal F(\overline{\ell^*},\ell^*,P_{belt})$ term in the previous subsubsection. The graph for $\mathcal F(\overline{\ell^*},\ell^*,P_{belt})$ has a very specific structure: it involves the maximal number of degree one vertices, which are all of linear type, and which all have the same dual variable. A byproduct of the proof of the Lemma below is that this is the only such pairing. By using the fact that for other pairing more degree two or quadratic degree one vertices appear, or that linear degree one vertices may have different dual variables, we are able to improve their upper bound. The key new idea in the proof is that estimating the contribution of a free variable may involve an upper bound with a factor that depends on another free variable, and this factor is converted into a gain factor once one estimates the contribution of this other free variable. This sees the very structure of the paired graph, and improves on the previous techniques used in \cite{CG} which involved universal bounds at a vertex which were independent of the other free variables.

\begin{lemma} \

Pick any $1-\frac{1}{2d+1} < c_2<c_1<1$ and $n\geq d+1$. There exists a constant $c>0$ such that for all pairings $P\neq P_{belt}$, for all
\be \label{id:hptbeltpairing}
\ep^{c_1}\leq t \leq \ep^{c_2}
\ee
there holds:
\be \label{bd:FPbetterthanFPbelt}
|\mathcal F(\overline{\ell^*},\ell^*,P)(t)|\lesssim \frac{1}{T_{kin}^n}t\epsilon^{d-1+c} \ll \frac{1}{T_{kin}^n}t\epsilon^{d-1}.
\ee

\end{lemma}

\begin{proof}

\textbf{Step 1:} \emph{Reduction to a special case}. Recall Definition \ref{def:degreeonecases}. We claim that if the following condition: $n_1^q=0$, $n_1=n_1^l=2n-2$, $n_0=n_2=1$, is not satisfied then the bound \fref{bd:FPbetterthanFPbelt} holds true.

To prove it, we refine the algorithm previously used in the proof of Proposition \ref{pr:L2expansion}, in order to upper bound $|\mathcal F(\overline{\ell^*},\ell^*,P)(t)|)$. We refer to this proof for the description of the algorithm.

\underline{The basic estimate} At each vertex, we use the following estimates with parameter $a=0$:
\begin{itemize}
\item \emph{For degree two vertices} we use \fref{bd:degre22} which gives a factor $t\ep^{1-2d}$.
\item \emph{For degree one vertices of linear type} we use \fref{bd:degre1lineaire2} which gives a factor $\ep^{1-d}$.
\item \emph{For degree one vertices of quadratic type} \fref{bd:degre1quadratic2} gives a factor $t\ep^{1-d}$.
\item \emph{For degree zero vertices} we simply bound $|\alpha+\sum_{i=k}^n\Omega_i+\frac{i}{t}|^{-1}\leq t$.
\end{itemize}
This produces using \fref{id:ni}:
\be \label{bd:interbeltbasicestimate}
\mathcal F_{P}(t)\lesssim \lambda^{4n}\ep^{d(2n+1)} t^{n_2}\ep^{(1-2d)n_2} \ep^{(1-d)n_1^l} t^{n_1^q}\ep^{(1-d)n_1^q} t^{n_0}\ep^{-\kappa}=\frac{1}{T_{kin}^n}\ep^{-n_2}t^{n_0+n_1^q+n_2}\ep^{-\kappa}.
\ee

\underline{The improved estimate}. In the previous algorithm, at all vertices we took $a=0$ in all bounds \fref{bd:degre1lineaire2}, \fref{bd:degre1quadratic2} and \fref{bd:degre22}. However, by incorporating a product of $a$ weights of the form $\max \left(\frac{1}{\langle \cdot-\tilde \xi \rangle},t\right)$ as displayed in these estimates, they are all improved by a factor $t^a$ (unless for one degenerate case considered later separately). These weights appear the following way.

Each time a degree one vertex of linear type is estimated, the identity \fref{id:degreeonelinear} and the estimate \fref{bd:degre1lineaire2} pull out a factor $\max \left(\frac{1}{\langle \zeta \rangle},t\right)$. The dual variable $\zeta$ at this vertex depends solely on free variables appearing after in the algorithm. This factor $\max \left(\frac{1}{\langle \zeta \rangle},t\right)$ is kept as is when estimating vertices appearing after whose free variables do not enter in the decomposition of $\zeta$. When reaching the first vertex at which a free variable is attached that enters in the decomposition of $\zeta$, then this weight $\max \left(\frac{1}{\langle \zeta \rangle},t\right)$ can be incorporated in any of the estimates \fref{bd:degre1lineaire2}, \fref{bd:degre1quadratic2} and \fref{bd:degre22}, improving the bound by a factor $t$. 

 Since free variables may appear in the decomposition of multiple dual variables of degree one vertices appearing before them, a product of multiple such weights may appear, in which case $a\geq 2$. Corollary \ref{cor:weightedvertices} allows for up to $a=d-1$ weights to be incorporated. Hence if there are $d-1$ or less degree one vertices of linear type, we gain a factor $t^{n_1^l}$ as each weight could be transformed into a $t$ gain. If there are $d-1$ or more degree one vertices of linear type, at least a gain factor $t^{d-1}$ is allowed (but maybe no better estimate is possible, in the case all such weights are estimated at a single vertex which precisely happens for the belt pairing). Therefore, there holds the following improved estimate in comparison to the basic estimate \fref{bd:interbeltbasicestimate} (except in the degenerate case considered later separately):
 $$
\mathcal F_{P}(t)\lesssim \frac{1}{T_{kin}^n}\ep^{-n_2}t^{n_0+\min (d-1,n_1^l)+n_1^q+n_2}\ep^{-\kappa}.
$$
We now consider the above estimate in various cases. Assume first that $n_1^l\leq d-1$. Then since $n_0+n^1_q+n_1^l+n_2=2n$ and $n_2\leq n$ from \fref{id:ni}, from \fref{id:hptbeltpairing} and $n\geq d+1$:
$$
\mathcal F_{P}(t)\lesssim\frac{1}{T_{kin}^n}\ep^{-n_2}t^{n_0+n_1^l+n_1^q+n_2}\ep^{-\kappa}\lesssim  \frac{1}{T_{kin}^n} \ep^{-n}t^{2n}\ep^{-\kappa}\lesssim  \frac{1}{T_{kin}^n} \ep^{-(d+1)}t^{2(d+1)}\ep^{-\kappa}\lesssim  \frac{1}{T_{kin}^n} t \ep^{d-1+c}
$$
for some $c>0$ so the desired bound \fref{bd:FPbetterthanFPbelt} is proved. Assume next that $n_1^l\geq d-1$ and that $n_1^q\geq 1$ and $n_2\geq 1$. Then we get in this case too an improvement since $n_0=n_2$ from \fref{id:ni} and from \fref{id:hptbeltpairing}:
$$
\mathcal F_{P}(t)\lesssim  \frac{1}{T_{kin}^n}\ep^{-n_2}t^{n_0+(d-1)+n_1^q+n_2}\ep^{-\kappa} \lesssim  \frac{1}{T_{kin}^n} t^{d-1}t^{n_1^q}(\ep^{-1}t^2)^{n_2}\ep^{-\kappa}\lesssim   \frac{1}{T_{kin}^n} t^{d+2} \ep^{-1}\ep^{-\kappa} \lesssim  \frac{1}{T_{kin}^n} t \ep^{d-1+c}.
$$
Assume now that $n_1^l\geq d-1$ and that $n_1^q=0$ and $n_2\geq 2$. Then the estimate is improved similarly:
$$
\mathcal F_{P}(t)\lesssim  \frac{1}{T_{kin}^n}\ep^{-n_2}t^{n_0+(d-1)+n_2}\ep^{-\kappa} \lesssim  \frac{1}{T_{kin}^n} t^{d-1}(\ep^{-1}t^2)^{n_2}\ep^{-\kappa}\lesssim   \frac{1}{T_{kin}^n} t^{d+3} \ep^{-2}\ep^{-\kappa} \lesssim  \frac{1}{T_{kin}^n} t \ep^{d-1+c}.
$$
Assume now that $n_1^l\geq d-1$, and that $n_0=n_2= 0$. Then we use \fref{id:nondegenn0}. Instead of using the estimate \fref{bd:degre1lineaire2} at the vertex $v_1'$, we use that on the support of $|\Delta_{\overline{\ell^*},\ell^*,P}|$, the corresponding sum in \fref{bd:degre1lineaire2} contains only one element from \fref{id:nondegenn0} hence is $\lesssim t$. This improves \fref{bd:degre1lineaire2} by a factor $\ep^{d-1}$ so that:
$$
\mathcal F_{P}(t)\lesssim  \frac{1}{T_{kin}^n} t^{(d-1)}\ep^{d-1}\ep^{-\kappa} \lesssim  \frac{1}{T_{kin}^n} t \ep^{d-1+c}.
$$
Consequently, gathering the bounds for all the above cases, we have proved the desired estimate whenever $n_1^q+n_0\geq 2$ or $n_0=0$. Hence, from \fref{id:ni}, the only remaining case corresponds to $n_1^q=0$, $n_2=n_0=1$ and $n_1^l=2n-2$, which is precisely what we claimed in Step 1.

We finish by considering the degenerate case in which the $t^{\min (d-1,a)}$ gain did not hold. This corresponds in the case $a=d-1$ in estimate \fref{bd:degre1quadratic2}. In this case the gain factor is only $\ep^{d-1}t^{-1/2}$ instead of $t^{d-1}$. However, since this happens at a quadratic degree one vertex and that $n_0\geq 1$ from \fref{id:ni}, we still find the desired bound:
$$
\mathcal F_{P}(t)\lesssim \frac{1}{T_{kin}^n}\ep^{-n_2}t^{n_0+n_1^q+n_2}\ep^{d-1}t^{-\frac 12}\ep^{-\kappa} \lesssim \frac{1}{T_{kin}^n}(\ep^{-1}t^2)^{n_2}t^{n_1^q-\frac 12}\ep^{d-1}\ep^{-\kappa}\lesssim \frac{1}{T_{kin}^n}\ep^{d-2}t^{\frac 52}\ep^{-\kappa}  \lesssim \frac{1}{T_{kin}^n} t\ep^{d-1+c}
$$
which shows \fref{bd:FPbetterthanFPbelt}.\\

\noindent \textbf{Step 2:} \emph{Identification of a special graph in a special case}. From step $1$, there remains to consider paired graphs for which $n_1^l=2n-2$, $n_1^q=0$ and $n_0=n_2=1$. These corresponds to paired graphs for which the bottom interaction vertex of the right subtree is of degree zero, the top interaction vertex of the left subtree is of degree two, and all other remaining interaction vertices are of degree one of linear type. We shall thus use the notation from Subsection \fref{subsubsec:example} for the name of the free variables: $(\tilde \xi_1,\tilde \xi_2,\tilde \xi_3,\xi_1,...,\xi_{n-1},\xi_1',...,\xi_{n-1}')$.

We describe the case $d=2$ and $n=3$ which retains all the difficulty of the other cases up solely to notational complexity. From Step 1 there remains to study the case $n_1=n_1^l=2n-2=4$ and $n_0=n_2=1$. We consider the right subtree, denoting its interaction vertices $v_1'$, $v_2'$, $v_3'$ as first, second, and third, from bottom to top. $v_1'$ has three initial vertices $v_{0,5}'$, $v_{0,6}'$, and $v_{0,7}'$ below, with parities of $-$, $+$ and $+$ respectively, $v_2'$ has $v_{0,3}'$ and $v_{0,4}'$ below it with parities $-$ and $+$, and $v_3'$ has $v_{0,1}'$ and $v_{0,2}'$ below with parities $-$ and $+$. Our first claim is that unless $v_{0,6}'$ (or $v_{0,7}'$ as these two vertices play a symmetric role) is paired with $v_{0,3}'$, and $v_{0,4}'$ is paired with $v_{0,1}'$, then \fref{bd:FPbetterthanFPbelt} holds true.

To show this, we go through the construction of the minimal spanning tree, see \cite{CG}. As $v_2'$ is of degree one, exactly one initial vertex below $v_2'$ is paired with one below $v_1'$. As $v_2'$ is of linear type, this pairing has to be between $v_{0,3}'$ and either $v_{0,6}'$ or $v_{0,7}'$ (indeed, if $v_{0,4}'$ was paired with $v_{0,5}'$ this would make $v_2'$ be of quadratic type from the integration of Kirchhoff's laws in the graph). We thus assume $v_{0,3'}$ is paired with $v_{0,6}'$ without loss of generality and denote by $\xi'_1$ the corresponding free variable. As $v_3'$ is of degree one, exactly one initial vertex below $v_3'$ has to be paired with one below $v_1'$ or one below $v_2'$. As $v_3'$ is of linear type, this paired vertex under it has to be $v_{0,1}'$ and we call $\xi'_2$ the free variable. Two cases are possible, either a) $v_{0,1}'$ is paired with $v_{0,7}'$, or b) it is paired with $v_{0,4}'$. Let us consider the case a) which is illustrated as follows:
\begin{center}
\includegraphics[width=10cm]{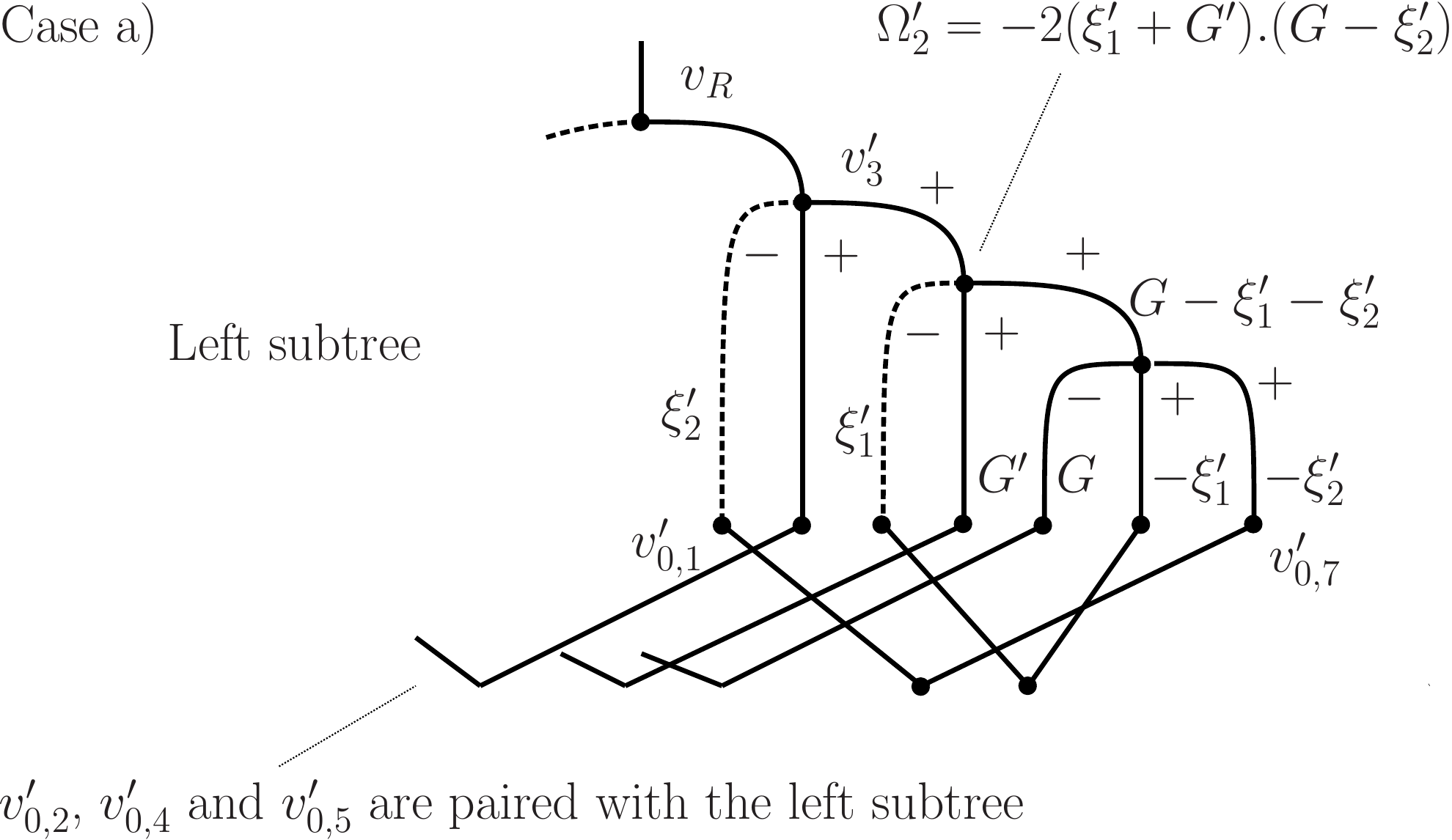}
\end{center}
where $G$ and $G'$ are independent of $\xi_1'$ and $\xi_2'$. We have in particular:
$$
\Omega'_2=-2(\xi_1'+G').(G- \xi_2').
$$
In view of Definition \ref{def:degreeonecases} and \fref{id:degreeonelinear}, the dual variable at $v_2'$ is $\zeta=-2(G'- \xi_2')$ which depends on $\xi_2'$. We perform the same algorithm as Case 2 and Case 3 of Step 1. At the vertex $v_2'$ we use \fref{bd:degre1lineaire} with $a=0$ which pulls out a $\max (\frac{1}{|G-\xi_2'|},t)$ factor. At the vertex $v_3'$ we use \fref{bd:degre1lineaire} with $a=1$, transforming this $\max (\frac{1}{|G-\xi_2'|},t)$ factor into a factor $t$ gain in comparison with the $a=0$ estimate. Next, the algorithm goes through the left subtree which contains at least one degree one vertex of linear type, whose dual variable $\tilde \zeta$ depends solely on free variables appearing after in the left subtree. Similarly, we apply the estimate \fref{bd:degre1lineaire} with $a=0$ at this vertex producing a $\max (\frac{1}{\tilde \zeta},t)$ factor which is then converted in a factor $t$ gain at the first vertex having a free variable entering in the decomposition of $\tilde \zeta$. The total gain factor is thus $t^2$. Hence:
$$
|\mathcal F_{P}(t)|\lesssim \frac{1}{T_{kin}^n}\ep^{-n_2}t^{n_0+n_1^q+n_2}\ep^{-\kappa}t^2=\frac{1}{T_{kin}^n}\ep^{-1}t^2 \ep^{-\kappa}t^{2}\lesssim \frac{t \ep }{T_{kin}^n} \ep^c
$$
as $n_0=1=n_2$ and $n_1^q=0$, and from \fref{id:hptbeltpairing}, so \fref{bd:FPbetterthanFPbelt} holds true. This proves the first claim.\\

\noindent We next obtain an analogue property for the left subtree by symmetry. Indeed, the algorithm to construct the spanning tree starts with the right subtree only by convention. By performing the very same argument in the case where the algorithm starts with examining the left subtree, we obtain the analogue of our first claim but for the left subtree: unless $v_{0,1}$ or $v_{0,2}$ is paired with $v_{0,5}$, and $v_{0,4}$ is paired with $v_{0,7}$, then \fref{bd:FPbetterthanFPbelt} holds true.\\

\noindent Therefore, there remains to examine pairings such that, without loss of generality, $v_{0,2}$ is paired with $v_{0,5}$, and $v_{0,4}$ is paired with $v_{0,7}$,  $v_{0,2}'$ is paired with $v_{0,5}'$, $v_{0,4}'$ is paired with $v_{0,7}'$, and $(v_{0,1},v_{0,3},v_{0,6})$ are paired with $(v_{0,2}',v_{0,5}',v_{0,7}')$. We recall that an initial vertex of parity $\sigma$ can only be paired with one of opposite parity $-\sigma$. Hence there are only two ways to pair $(v_{0,1},v_{0,3},v_{0,6})$ with $(v_{0,2}',v_{0,5}',v_{0,7}')$. First, in A) $v_{0,i}$ is paired with $v_{0,8-i}'$ for $i=1,3,6$, and then in B) $v_{0,1}$ is paired with $v_{0,2}'$, $v_{0,3}$ is paired with $v_{0,5}'$ and $v_{0,6}$ is paired with $v_{0,7}'$. Case B) is illustrated below. Since the free edges are the same as those of the belt pairing, see Step 1 in the proof of Proposition \ref{pr:tgeq1examples}, we keep the same notation for them:
\begin{center}
\includegraphics[width=9cm]{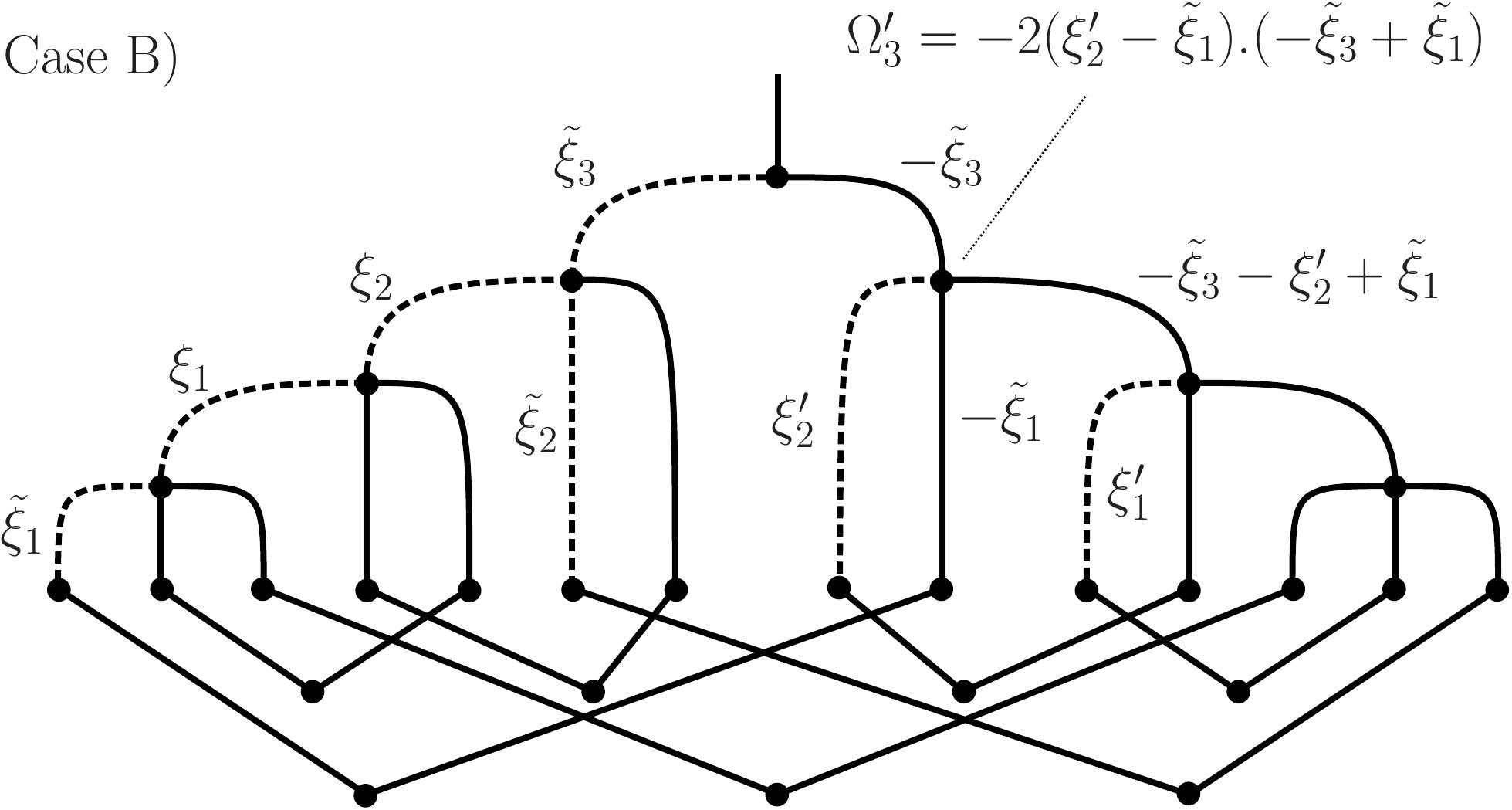}
\end{center}
and an integration of Kirchhoff's laws give that:
$$
\Omega_3'=-2(\xi_2'-\tilde \xi_1).(-\tilde \xi_3+\tilde \xi_1).
$$
Hence, $v_3'$ is a degree one vertex of linear type, whose dual variable is $-2(-\tilde \xi_3+\tilde \xi_1)$, which depends on $\tilde \xi_1$. Therefore, as in the previous case a), one can gain two times a factor $t$, producing a $t^2$ gain, and the bound is improved. Consequently, only case A) could possibly yield an estimate which does not satisfy \fref{bd:FPbetterthanFPbelt}. The proof is finished, as case A) is precisely the belt pairing.\\

\noindent \textbf{Step 3:} \emph{Identification of a special graph for $d\geq 3$, and $d\geq 2$ and $n\geq 4$}. The very same reasoning as in Step 2 works in these cases, so we just sketch it. 

First, if the conditions
\be \label{id:condition1}
\left\{ \begin{array}{l l} v_{0,2n-3}'\mbox{ is paired with } v_{0,2n} \mbox{ or } v_{0,2n+1}'\\
v_{0,2(n-k)-1}' \mbox{ is paired with } v_{0,2(n-k)+2}' \mbox{ for } k=2,...,n-1,\\
\mbox{and the remaining vertices are paired with he left subtree,} \end{array} \right.
\ee
are not satisfied, then there are either two or more degree two vertices in the whole paired graph, or there is at least one quadratic degree one vertex in the right subtree, or there is one linear degree one vertex in the right subtree whose dual variable depends on a free variable of the right subtree. In all cases, the bound \fref{bd:FPbetterthanFPbelt} holds true.

By a symmetric argument, the analogue of \fref{id:condition1} holds for the left subtree. Namely, if the conditions
\be \label{id:condition2}
\left\{ \begin{array}{l l}   v_{0,5}\mbox{ is paired with } v_{0,1} \mbox{ or } v_{0,2} \\
 v_{0,2k} \mbox{ is paired with } v_{0,2k+3} \mbox{ for } k=2,...,n-1,\\
 \mbox{and the remaining vertices are paired with he left subtree,} \end{array} \right.
\ee
are not satisfied, then the bound \fref{bd:FPbetterthanFPbelt} holds true.

Hence there remains to study pairings satisfying \fref{id:condition1} and \fref{id:condition2}. For such pairings, we only need to determine how $(v_{0,1},v_{0,3},v_{0,2n})$ are paired with $(v_{0,2}',v_{0,2n-1}',v_{0,2n+1}')$. In the case where $v_{0,1}$ is paired with $v_{0,2}'$, $v_{0,3}$ with $v_{0,2n-1}'$ and $v_{0,2n}$ with $v_{0,2n+1}'$ then $v_{n'}$ is a linear degree one vertex whose dual variable depends on the free variable attached to $v_1$, and the bound \fref{bd:FPbetterthanFPbelt} holds true. The only other possible pairing is the belt pairing, which ends the proof.

\end{proof}

\subsubsection{Proof of Proposition \ref{pr:tleq1}}

This proposition is easily proved using the two previous lemmas.

\begin{proof}[Proof of Proposition \ref{pr:tleq1}]

It is a direct consequence of the identity \fref{LpLpexpression}, and of the corresponding bounds \fref{bd:mathcalFPbelt} and \fref{bd:FPbetterthanFPbelt}.

\end{proof}

\section{Number theory results for generic dispersion relations and large times} \label{sec:numbergeneric}

This Section is devoted to the proof of the following theorem.

\begin{theorem} \label{theoremNT}
For almost all symmetric matrices close to the identity, and $L^{2-d+\kappa} < \delta <1$,
\begin{align}
\label{NT1} & \# \{ \xi, \; |\xi| \leq L, \; |H\zeta \cdot \xi| < \delta \} \lesssim_{H} L^{d-1} \sqrt{\delta} \qquad \mbox{if } 1\leq |\zeta |\leq \ep^{-1}\\
\label{NT2} & \# \{ \eta, \; |\eta| < L , \; ||\eta|_H^2 - a| < \delta \} \lesssim_{H} L^{d-1} \sqrt{\delta} \\
\label{NT3} & \# \left\{\eta,\xi \in \mathbb Z^d \mbox{ with } |\eta|,|\xi|\leq L, \quad | H\xi \cdot \eta- a|\leq \delta \right\} \lesssim_{H} L^{2d-2} \delta.
\end{align}
\end{theorem}

\subsection{Proof of the bound~\eqref{NT1} for linear degree one vertices} 

\begin{proposition}
\label{herisson}
For almost all (in the Lebesgue sense) symmetric matrices $H$ close to the identity, for all $L\geq 1$, $\kappa>0$, and $\delta > L^{-100 d}$, there holds, for any $|\xi_0| \leq L$, 
$$
\# \{ \xi, \; |\xi| \leq L, \; |H\xi_0 \cdot \xi| < \delta \} \lesssim_{H,\kappa} L^{d-1} \delta + L^{2 - \frac{2}{d}+\kappa} \delta^{1/d} + L^{d-2+\frac{2}{d}+\kappa} \delta^{1-\frac{1}{d} }.
$$
In particular, under the above conditions,
$$
\# \{ \xi, \; |\xi| \leq L, \; |H\xi_0 \cdot \xi| < \delta \} \lesssim_{H,\kappa} L^{d-1} \sqrt{\delta} \qquad \mbox{if $L^{2-2d+\kappa} < \delta < 1$}.
$$
\end{proposition}

The following lemma will be the key to the proof of the above proposition.

\begin{lemma}
For $v \in \mathbb{R}^d$, $\frac{1}{2} \leq |v| \leq 1$, let
$$
E_{R,\delta}(v) = \{ \xi \in \mathbb{Z}^d, \; | \xi | \sim R, \; |v \cdot \xi| < \delta \}.
$$
For $M > R^{d-1} \delta$, there exists a set $\mathcal{E}_{R,\delta,M}$ with volume
$$
\operatorname{Vol} \mathcal{E}_{R,\delta,M} \lesssim \sum_{r=1}^{d-1} \frac{\delta^rR^{r(d-1)+}}{M^d}
$$
such that: if $v \notin \mathcal{E}_{R,\delta,M}$, then
\begin{equation}
\label{ecureuil}
\# E_{R,\delta}(v) < M.
\end{equation}
\end{lemma}

\begin{proof} Since $M > R^{d-1} \delta$, the inequality~\eqref{ecureuil} is immediately satisfied if $\dim E_{R,\delta}(v) = d$. Indeed, it is a classical result in convex geometry that if a convex body $K \subset \mathbb{R}^d$ is symmetric with respect to the origin and contains a subset of $\mathbb{Z}^d$ of dimension $d$, then $\# (K \cap \mathbb{Z}^d) \lesssim_d \operatorname{Vol} K$.

We now turn to the case where $\dim E_{R,\delta}(v) = r$, where $r \in \{ 1,\dots,d-1 \}$. Our aim is to find which $v$ have to be excluded in order to ensure that~\eqref{ecureuil} holds when $E_{R,\delta}(v)$ has dimension $r$.

We start with a linear subspace $H$ of dimension $r$, which we think of as the span of $E_{R,\delta}(v)$. 
Let $q_1, \dots, q_r$ be an almost orthogonal basis (over $\mathbb{Z}^d$) of $H \cap \mathbb{Z}^d$. Decompose $v$ into
$$
v = n + \sum_{i=1}^{r} \beta_i q_i,
$$
where $n$ is orthogonal to $H$, and similarly, any $\xi \in H \cap \mathbb{Z}^d$ into
$$
\xi = \sum_{i=1}^r \gamma_i q_i.
$$
Then
$$
v \cdot \xi = \beta^t Q \gamma,
$$
where ${\beta}$ and $\gamma$ denote the vectors with coordinates $\beta_i$ and $\gamma_i$ respectively, and $Q$ is the symmetric matrix $( q_i \cdot q_j )_{i,j}$. We want to bound
$$
\# \{ \gamma, \; |\gamma_i| \leq R / |q_i|, \; |\beta^t Q \gamma| < \delta \},
$$
assuming that this set is $r$-dimensional. Therefore, by the convexity argument mentioned at the very beginning of the proof it is
$$
\lesssim \operatorname{Vol}  \{ \gamma, \; |\gamma_i| \leq R / |q_i|, \; |\beta^t Q \gamma| < \delta \}.
$$
Changing variables to $\zeta_i = \frac{|q_i|}{R} \gamma_i$, and letting $\beta' = Q \beta$, this is
$$
=  \frac{R^r}{\prod_{i=1}^r |q_i|} \operatorname{Vol} \left\{ \zeta, \; |\zeta_i| \leq 1, \; \left| \sum_{i=1}^r \zeta_i \frac{R \beta_i'}{|q_i|} \right| < \delta \right\} \sim \frac{R^r}{\prod_{i=1}^r |q_i|} \min \left( 1 , \delta \left( \sum_{i=1}^r R \frac{|\beta'_i}{|q_i|} \right)^{-1} \right).
$$
This is $\lesssim M$ if
$$
\mbox{either}  \qquad \prod |q_i| \geq \frac{R^r}{M} \qquad \mbox{or} \qquad \sum \frac{|\beta'_i|}{|q_i|} \geq \frac{\delta R^{r-1}}{M \prod |q_i|}.
$$
Therefore, if the first inequality above is not satisfied, it is natural to choose the exceptional set 
$$
\mathcal{F}_{q,R,M,\delta}^r = \left\{ v\,, |\beta'_i| \leq \frac{\delta R^{r-1} |q_i|}{M \prod |q_i|}\right\}.
$$
Since the basis $(q_i)$ is almost orthogonal, we must add a factor $\sim \prod |q_i|^{-2}$ to deduce the volume in $\beta$ from the volume in $\beta'$; and a factor $\sim \prod |q_i|$ to deduce the volume in $v$ from the volume in $\beta$. Therefore,
$$
\operatorname{Vol} \mathcal{F}_{q,R,M,\delta}^r \lesssim \frac{\delta^r R^{r(r-1)}}{M^r \prod |q_i|^{r}}.
$$
Setting
$$
\mathcal{E}_{R,M,\delta} = \cup_{\substack{r=1,\dots,d-1 \\ |q_i| < R \\ \prod|q_i| < R^r / M }} \mathcal{F}_{q,R,M,\delta}^r,
$$
its volume can be bounded by
$$
\operatorname{Vol}  \mathcal{E}_{R,M,\delta} \leq \sum_{ \substack{r=1,\dots,d-1 \\ |q_i| < R \\ \prod|q_i| < R^r / M }} \frac{\delta^r R^{r(r-1)}}{M^r \prod |q_i|^{r}}.
$$
In order to estimate this sum, we use the fact that
$$
\# \{ q \in \mathbb{Z}^d, \; \prod |q_i| \sim N \} \lesssim N^{d+}
$$
(indeed, assuming that $|q_i| \sim 2^j_i$ leaves $2^{dj_i}$ choices for $q_i$, and $\prod |q_i| \sim N$ if and only if $2^{j_1+...+j_r} \sim N$) which implies that
$$
\operatorname{Vol}  \mathcal{E}_{R,M,\delta}\lesssim  \sum_{r=1}^{d-1} \frac{\delta^rR^{r(d-1)+}}{M^d}.
$$
\end{proof}

We can now turn to the proof of the proposition

\begin{proof}[Proof of Proposition~\ref{herisson}]
We now want to ensure that, for any $\zeta$ with $|\zeta| \sim S$, there holds
\begin{equation}
\label{scarabee}
\# \{ \xi, \; |\xi| \sim R, \;|(H\zeta) \cdot \xi| < \delta \} \lesssim N
\end{equation}
(where the value of $N$ has not been set yet), which can also be written 
$$
\# \left\{ \xi, \; |\xi| \sim R, \;|(H S^{-1} \zeta ) \cdot \xi| < {\delta / S} \right\} \lesssim N.
$$
Choosing $N \geq R^{d-1} \delta /S$, and setting $v = H S^{-1} \zeta$, we learn from the previous result that the above holds as long as $v \notin \mathcal{E}_{R,N,\delta/S}$, with
$$
\operatorname{Vol}  \mathcal{E}_{R,N,\delta /S} \lesssim \sum_{r=1}^{d-1} \frac{\delta^rR^{r(d-1)+}}{N^dS^r}.
$$
Thus, one can choose an exceptional set in the (matrix valued) variable $H$ of comparable size. Summing over all possible values of $\zeta$, we just proved that we can pick a set $\mathcal{D}_{R,S}$ such that 
$$
\operatorname{Vol} \mathcal{D}_{R,S,N} \lesssim S^d \sum_{r=1}^{d-1} \frac{\delta^r R^{r(d-1)+}}{N^dS^r} \sim S (R^{d-1+} \delta)^{d-1} N^{-d} + S^{d-1} R^{(d-1)+} \delta N^{-d},
$$
and, for any $H \notin \mathcal{D}_{R,S,N}$, and for any $\zeta$ with $|\zeta| \sim S$,~\eqref{scarabee} holds.  Choosing
$$
N = N(L,\nu) = L^{d-1} \delta + \frac{1}{\nu} L^{2 - \frac{2}{d}+} \delta^{1/d} + \frac{1}{\nu} L^{d-2+ \frac{2}{d}+} \delta^{1-\frac{1}{d} } , 
$$
and letting
$$
\mathcal{D}_{L,\nu} = \cup_{\substack{R,S \leq L \\ \delta > L^{-100 d}}} \mathcal{D}_{R,S,N(L,\nu)}, \qquad \mathcal{D}_\nu = \cup_L \mathcal{D}_{L,\nu}
$$
(where $R,S,L,\delta$ range over dyadic values $2^{\mathbb{Z}}$), we obtain
$$
\operatorname{Vol} \mathcal{D}_{L,\nu} \lesssim L^- \nu^d, \qquad \operatorname{Vol} \mathcal{D}_\nu \lesssim \nu^d.
$$
There remains to set
$$
\mathcal{D} = \cap_\nu \mathcal{D}_\nu,
$$
which has measure zero, as the exceptional set, outside of which the desired property holds.
\end{proof}

\subsection{Proof of the bound~\eqref{NT2} for quadratic degree one vertices}
\begin{proposition}
For almost all (in the Lebesgue sense) symmetric matrices $H$ close to the identity, for all $a \in \mathbb{R}$, $L\geq 1$, $\kappa>0$, and $1>\delta > L^{-100 d}$,
\be \label{bd:degree1quadra}
\# \{ \eta, \; |\eta| < L , \; ||\eta|_H^2 - a| < \delta \} \lesssim_{A,\kappa} L^{d-2} \delta +  L^{\frac{d-1}{2}+\kappa} \sqrt \delta + L^{\frac{d}{2}+\kappa}.
\ee
In particular,
$$
\# \{ \eta, \; |\eta| < L , \; ||\eta|_H^2 - a| < \delta \} \lesssim_{A,\kappa} L^{d-1} \sqrt{\delta} \qquad \mbox{if $\delta > L^{2-d+\kappa}$}.
$$
\end{proposition}

\begin{proof} Let $\chi$ be a one-dimensionals smooth cutoff function with compact support; slighlty abusing notations, we also denote $\chi$ its tensorization
$$
\chi(x_1,\dots,x_d) = \chi(x_1) \dots \chi(x_d).
$$
Applying the circle method, we want to estimate
$$
\delta \int_{-\infty}^{\infty} K(L,\tau) e^{-2\pi i a \tau}\widehat \chi \left( \delta \tau  \right)\,d\tau
$$
where
$$
K(L,\tau) = \sum \chi \left( \frac{\eta}{L} \right) e^{2\pi i \tau |\eta|^2_H}.
$$
We now split
\begin{align*}
 \delta \int_{-\infty}^{\infty} K(L,\tau) e^{-2\pi i a  \tau} \widehat \chi \left( \delta \tau  \right)\,d\tau \leq \underbrace{\delta  \left|\sum_{\eta}\int_{|\tau|\leq \frac{1}{L}} \dots \,d\tau \right|}_{I}+\underbrace{\delta \left| \sum_{\eta}\int_{|\tau|\geq \frac{1}{L}} \dots \,d\tau \right|}_{II}.
\end{align*}
We can bound, by Poisson summation and the stationary phase lemma,
$$
|K(L,\tau)| \lesssim \left( \frac{L}{\tau L + \frac{1}{L}} \right)^{d/2}.
$$
As a consequence,
$$
I \lesssim \delta \int_{-1/L}^{1/L} \left( \frac{L}{\tau L + \frac{1}{L}} \right)^{d/2} \,d\tau \lesssim L^{d-2} \delta.
$$
Turning to $II$, we bound it by the Cauchy-Schwarz inequality
\begin{align*}
II & \lesssim \sqrt{\delta} \left( \int_{ |\tau| > \frac{1}{L}} \left| \sum_{\eta} \chi \left( \frac{\eta}{L} \right) e^{2\pi i (|\eta|^2_H - a) \tau} \right|^2 |\widehat{\chi}(\delta \tau)| \,d\tau \right)^{1/2}
\end{align*}
We now use that $ \left| \sum_{\eta} \chi \left( \frac{\eta}{L} \right) e^{2\pi i (|\eta|^2_H - a) \tau} \right|^2 = \sum_{\eta,\eta'} \chi \left( \frac{\eta}{L} \right)\chi \left( \frac{\eta'}{L} \right) e^{2\pi i (|\eta|_H^2 - |\eta'|_H^2)}$. Changing the summation variables to $\alpha = \eta-\eta'$ and $\beta = \eta + \eta'$, and restricting implicitly the summation over $\alpha$ and $\beta$ to all $\alpha,\beta$ which have the same parity, we obtain that
\begin{align*}
II \leq \sqrt{\delta} \left( \int_{  |\tau| > \frac{1}{L}}  \sum_{\alpha, \beta} \chi\left(\frac {\alpha + \beta}{2L} \right) \chi\left(\frac {\alpha - \beta}{2L} \right)e^{2\pi i (H \alpha \cdot \beta) \tau}| \widehat{\chi}(\delta \tau)| \, d\tau \right)^{1/2}.
\end{align*}
Next, we average over $H$ and use once again the Cauchy-Schwarz inequality:
\begin{align*}
\int II\,dH & \lesssim \sqrt \delta \left( \int  \int_{  |\tau| > \frac{1}{L}}  \sum_{\alpha, \beta}\chi\left(\frac {\alpha + \beta}{2L} \right) \chi\left(\frac {\alpha - \beta}{2L} \right) e^{2\pi i (H \alpha \cdot \beta) \tau}| \widehat{\chi}(\delta \tau)|\, d\tau \,dH \right)^{1/2} \\
& \leq \sqrt \delta \left( \sum_\alpha \int \int_{  |\tau| > \frac{1}{L}}  \left| \sum_{\beta} \chi\left(\frac {\alpha + \beta}{2L} \right) \chi\left(\frac {\alpha - \beta}{2L} \right)e^{2\pi i (H \alpha \cdot \beta) \tau} \right| | \widehat{\chi}(\delta \tau)| \,d\tau \,dH \right)^{1/2}.
\end{align*}
By Abel summation, this is
\begin{align*}
\dots & \lesssim \sqrt \delta \left( \sum_{|\alpha| \lesssim L} \int \int_{ |\tau| > \frac{1}{L}} \prod_{j=1}^d \min \left( L , \frac{1}{\| \tau H \cdot \alpha_j \|} \right)| \widehat{\chi}(\delta \tau)| \,d\tau \, dH \right)^{1/2},
\end{align*}
and inequality~\eqref{averagedbound} gives
\begin{align*}
\dots & \leq \sqrt \delta \left( (\ln L)^d \sum_\alpha \left[ \int_{ \frac{1}{L} < |\tau| < \frac{1}{|\alpha|}} \frac{1}{\tau^d |\alpha|^d} \,d\tau + \int_{ |\tau| > \frac{1}{|\alpha|}} | \widehat{\chi}(\delta \tau)| \,d\tau \right] \right)^{1/2} \\
& \lesssim  L^{\frac{d-1}{2}} (\ln L)^{\frac{d+1}{2}} \sqrt \delta + L^{d/2} (\ln L)^{d/2}.
\end{align*}
The end of the proof is now the same as in Proposition~\eqref{colibri}, and will be omitted.
\end{proof}

\subsection{Proof of the bound~\eqref{NT3} for degree two vertices} 

\begin{proposition} \label{colibri}
For almost all (in the Lebesgue sense) symmetric matrices $H$ close to the identity, for all $a \in \mathbb{R}$, $L\geq 1$, $\kappa>0$, and $1>\delta > L^{-100 d}$,
\be \label{bd:degree2}
\# \left\{\eta,\xi \in \mathbb Z^d \mbox{ with } |\eta|,|\xi|\leq L, \quad | H\xi \cdot \eta- a|\leq \delta \right\} \lesssim_{H,\kappa} L^{2d-2} \delta + L^{d+\kappa}.
\ee
In particular, under the above conditions,
$$
\# \left\{\eta,\xi \in \mathbb Z^d \mbox{ with } |\eta|,|\xi|\leq L, \quad | H\xi \cdot \eta- a|\leq \delta \right\} \lesssim_{A,\kappa} L^{2d-2} \delta \qquad \mbox{if $\delta > L^{2-d}$}.
$$
\end{proposition}

\begin{proof}
\emph{Step 0: Preliminary notations} 
We denote, for $\delta>0$, $L \leq 1$, $a \in \mathbb R$, $H\in \mathbb R^{d\times d}$ a positive definite symmetric matrix 
$$
N(H,a,L,\delta):= \# \left\{\eta,\xi \in \mathbb Z^d \mbox{ with } |\eta|,|\xi|\leq L, \quad | H\xi \cdot \eta- a|\leq \delta \right\}.
$$
We write a symmetric matrix close to the identity as:
$$
H=\begin{pmatrix} 1+h_{1,1} & h_{2,1}&h_{3,1} & ... & h_{d,1} \\ h_{2,1} & 1+h_{2,2} &... &...&  h_{d,2} \\ ... &... &...&...&... \\ h_{d,1} & h_{d,2} &...& ...& 1+h_{d,d} \end{pmatrix},
$$
and denote by $\mathcal H$ the following fixed neighbourhood of the identity: $|a_{i,j}|\leq r $ for $1\leq j\leq i \leq d$ for some fixed $0<r\ll 1$. We equip $\mathcal A$ with the measure $dH=\prod_{1\leq j \leq i \leq d} dh_{i,j}$.\\

\bigskip
\noindent \emph{Step 1: major and minor arcs}. Fix $L\geq 1$. We claim that:
\be \label{bd:degree2inter}
N(H,a,L,\delta)dH \lesssim L^{2d-2}\delta + P(H,L,\delta),
\ee
with
$$
\int_{\mathcal H} P(H,L,\delta) \,dH \lesssim  L^{d} \langle \ln L \rangle^{d},
$$
and now prove this claim. Let $\chi$ be a Schwartz class function with $\chi(x)=1$ for $|x|\leq 1$, and with $\hat \chi\geq 0$. We apply the circle method and isolate the oscillating and non-oscillating part:
\begin{align*}
N&\leq \sum_{|\eta|,|\xi|\leq L} \chi \left(\frac{H\xi \cdot \eta-a}{\delta} \right)= \delta  \sum_{|\eta|,|\xi|\leq L}\int_{-\infty}^{\infty} e^{2\pi i (H\xi \cdot \eta- a) \tau}\widehat \chi \left( \delta \tau  \right)\,d\tau \\
& \leq \underbrace{\delta  \int_{|\tau|\leq \frac{1}{L}}  \left| \sum_{|\xi|, |\eta|\leq L}[...]\right|\,d\tau }_{I}+\underbrace{\delta \sum_{|\xi|\leq L}\int_{|\tau|\geq \frac{1}{L|\xi|}}  \left| \sum_{|\eta|\leq L}[...]\right|\,d\tau}_{II}.
\end{align*}
The integrand of $I$ can be bounded by
$$
\left| \sum_{|\xi|, |\eta|\leq L}e^{2\pi i \tau (H\xi \cdot \eta)} \right|  \lesssim \left( \frac{L}{\tau L + \frac{1}{L}} \right)^d,
$$
so that
$$
I \lesssim L^{2d-2} \delta.
$$

\bigskip

\noindent \emph{Step 2: bounding the minor arcs in average} Applying the identity
$$
\left|\sum_{-n}^n e^{i2\pi r n}\right|= \left|\frac{e^{i2\pi r (n+1)}-e^{-i2\pi r n}}{e^{i2\pi r}-1}\right|\lesssim \frac{1}{\| r\|},
$$
where $\| r\|$ stands for the distance to the nearest integer, we get
$$
II \lesssim  \delta \sum_{|\xi|\leq L}\int_{|\tau|\geq \frac{1}{L}} \hat \chi (\delta \tau ) \prod_{j=1}^d \min \left(L,\frac{1}{\| \tau H_j \cdot \xi\|}\right)d\tau ,
$$
where $H_j$ is the $j$-th line of the matrix $H$. We now average the integrand in II with respect to the matrix $H$. We assume without loss of generality that $\xi$ satisfies $\xi_1\geq |\xi_i|$ for $1\leq i \leq d$, so that $\xi_1\gtrsim |\xi|$. Noting that for $1<j\leq d$, $H_j \cdot \xi$ depends on $h_{j,1}$ but is independent of $h_{i,1}$ for $i\neq j$, we integrate first over the first column of $H$ from top to bottom, followed by an integration over the remaining variables (irrelevant, see below):
\bee
&&\int_{\mathcal H}  \prod_{j=1}^d \min \left(L,\frac{1}{\| \tau H_j \cdot \xi\|}\right)dH\\
&& \quad \quad \leq \int_{[-r,r]^{\frac{(d-1)d}{2}}} \prod_{1<j\leq i\leq d}dh_{i,j} \int_{-r}^r \min \left(L,\frac{1}{\| \tau H_d \cdot \xi\|}\right)dh_{d,1} \times ... \\
&&  \quad \quad \quad \quad \quad \quad\quad \quad \quad \quad \quad \quad \quad \quad \quad \quad \quad \quad ...\times \int_{-r}^r  \min \left(L,\frac{1}{\| \tau H_1 \cdot \xi\|}\right)dh_{1,1}.
\eee
We evaluate the last integral first. First, note that $\tau H_1 \cdot \xi=\tau h_{1,1}\xi_1+\tau \xi_1+\sum_{j=2}^d\tau h_{1,j}\xi_j=\lambda h_{1,1}+C_1$ where $\lambda =\tau \xi_1$ and $C_1=\tau \xi_1+\sum_{j=2}^d\tau h_{1,j}\xi_j$. Second, note that $\| \cdot \|$ is an even periodic function with period one. We then evaluate the integral in two regimes. If $\lambda r \geq 1$, then this function has over $[-r,r]$ more than one period but less than $C\lambda r$ so that
$$
\int_{-r}^r  \min \left(L,\frac{1}{\| \lambda h_{1,1} +C_1\|}\right)dh_{1,1} \lesssim  \lambda \int_0^{\frac{1}{2\lambda}}  \min \left(L,\frac{1}{|\lambda h_{1,1} |}\right) \lesssim  \lambda \int_{0}^{\frac{1}{L \lambda}}Ldh_{1,1} + \lambda \int_{\frac{1}{L\lambda}}^{\frac 1 \lambda} \frac{dh_{1,1}}{\lambda h_{1,1}} \lesssim \langle \ln L \rangle.
$$
If $\lambda r \leq 1$ however, the function has no more than two periods over $[-r,r]$ and we estimate:
\begin{align*}
& \int_{-r}^r  \min \left(L,\frac{1}{\| \lambda h_{1,1} +C_1 \|}\right)dh_{1,1} \lesssim  \int_0^{r}  \min \left(L,\frac{1}{|\lambda h_{1,1} |}\right) \,dh_{1,1}\\
 & \qquad \qquad \qquad \lesssim \int_{0}^{\frac 1 {L \lambda}}L \,dh_{1,1}+ \int_{\frac{1}{L \lambda}}^{r} \frac{dh_{1,1}}{\lambda h_{1,1}} \lesssim \frac{\langle \ln L \rangle +\langle \ln \lambda \rangle}{\lambda} \leq \frac{\langle \ln L \rangle }{|\tau||\xi|}.
\end{align*}
The other remaining integrals, over the variables $h_{1,2}, \dots,h_{1,d}$, can be estimated in the very same way. After, the integration over the remaining variables produces an irrelevant factor $1$ contribution. Hence, as a result we can bound:
\begin{equation}
\label{averagedbound}
\int_{\mathcal H}  \prod_{j=1}^d \min \left(L , \frac{1}{\| \tau H_j \cdot \xi\|}\right)dH\lesssim \left\{ \begin{array}{l l l} \displaystyle \frac{\langle \ln L \rangle^d}{|\tau|^d |\xi|^d} && \mbox{if } \frac{1}{L|\xi|}\leq \tau \leq \frac{1}{|\xi|},\\ \langle \ln L \rangle^d  && \mbox{if } \frac{1}{|\xi|}\leq \tau,  \end{array} \right.
\end{equation}
which finally leads us to the upper bound:
\begin{align*}
\nonumber \int_{\mathcal H}II dH  &\lesssim \delta  \sum_{|\xi|\leq L} \int_{\frac{1}{L}\leq \tau \leq \frac{1}{|\xi|}} \hat \chi (\delta \tau ) \frac{\langle \ln L \rangle^d}{|\tau|^d |\xi|^d}d\tau +  \delta \sum_{|\xi|\leq L} \int_{\frac{1}{|\xi|}\leq \tau} \hat \chi (\delta \tau ) \langle \ln L \rangle^d \, d\tau \\
\label{bd:degre2inter3}&\lesssim  \delta L^{d-1} \langle \ln L \rangle^{d+1} + L^d \langle \ln L\rangle^d \\
& \lesssim  L^d \langle \ln L\rangle^d.
\end{align*}
After setting $P(H,L,\delta) = II$, this leads to the estimate \fref{bd:degree2inter}.

\bigskip \noindent
\emph{Step 4: From averaged to pointwise}. For any $i,j\in \mathbb N$, with $1\leq j \leq di$, applying the Bienaym\'e-Chebychev inequality to the bound \fref{bd:degree2inter} with dyadic decomposition $L=2^i$ and $\delta=2^{-j}$ gives that
$$
P(H,2^i,2^{-j})\lesssim 2^{id} \langle i \rangle^{d+3}.
$$
outside a measurable set $\mathcal H_{\nu,i,j}$ of measure $C\nu i^{-3}$. The set $\mathcal H_\nu =\cap_{i,j} \mathcal H_{\nu,i,j}$ then has measure
$$
|\mathcal H_\nu|\lesssim \sum_{i\in \mathbb N} \sum_{j=1}^{di} \nu i^{-3}\lesssim \nu.
$$
Hence the set $\mathcal H_{\infty}=\cap_{n\in \mathbb N}\mathcal H_{1/n}$ has measure 0, and on its complement the desired bound \fref{bd:degree2} holds true by the very definition of $\mathcal H_{\nu,i,j}$.
\end{proof}

\section{Number theory results for small times, and large times for Laplacian}

The first results of this section aim at finding weighted resolvent bounds for small times $t\leq 1$ and for any dispersion relation. We start by a technical result for weighted sums on particular sets. Consider for $|u|=1$, $p \in \mathbb{R}$, $l>0$ and $L\geq \max(1,l)$ the strip of size $(l,L)$:
$$
S_{u,p,l,L}=\left\{\xi \in \mathbb Z^d, \qquad |\xi|\leq L, \ |\xi \cdot u - p|\leq l\right\},
$$
as well as the annulus of size $(r,R)$ for $r>0$ and $R\geq (\max{1,r})$ and $H$ symmetric close to $\Id$:
$$
A_{r,R}=\left\{\xi \in \mathbb Z^d, \qquad R-r\leq |\xi|_H\leq R \right\}.
$$
We shall make the slight abuse of notation in the sequel that a strip of size $(l,L)$ will refer to a strip of size $(\tilde l,\tilde L)$ with $\tilde l\approx l$ and $\tilde L\approx L$ with universal constants in these inequalities.

\begin{lemma}[Weighted sums on particular sets]

Take $\xi_1,...,\xi_a\in \mathbb R^d$ for $a=0,...,d-1$. Then:
\be \label{bd:weightedstrip}
\sum_{S_{u,p,l,L}} \frac{1}{\langle \xi-\xi_1\rangle}...\frac{1}{\langle \xi-\xi_a\rangle}  \lesssim  \left\{\begin{array}{l l} L^{d-1-a} \max(1,\ell) \quad \mbox{for } a<d-1,\\  \langle \log L \rangle \max(1,\ell) \quad \mbox{for } a=d-1,\end{array} \right.
\ee
\be \label{bd:weightedannulus}
\sum_{A_{r,R}}  \frac{1}{\langle \xi-\xi_1\rangle}...\frac{1}{\langle \xi-\xi_a\rangle} \lesssim  \left\{\begin{array}{l l} R^{d-1-a} \max(1,r) \quad \mbox{for } a<d-1,\\  \langle \log R \rangle\max (1,r) \quad \mbox{for } a=d-1\end{array} \right.
\ee

\end{lemma}

\begin{remark}

The bounds of this lemma are responsible for the time interval limitation $[\ep,\ep^{1-\frac{1}{2d+1}}]$ in Proposition \ref{pr:tleq1}. Improving the time interval further than $\ep^{1-\frac{1}{2d+1}}$ would require to improve the bounds \fref{bd:weightedstrip} and \fref{bd:weightedannulus} in the case where the strip and annuli have width $\ll 1$ and with weights involving the visible part of the lattice points $\xi-\xi_k$, which would require a much more delicate number theoretical analysis.

\end{remark}

\begin{proof}

Optimising the distances, the worst case happen when $\xi_1,...,\xi_a\in S_{u,p,l,L}$. In that case, $\max(\langle \xi-\xi_1\rangle,...,\langle\xi-\xi_a\rangle)\lesssim L$. We perform a dyadic decomposition for $j_1,...,j_k$ with $1\leq 2^{j_i}\lesssim L$ for $i=1,...,a$ and define:
$$
E_{j_1,...,j_k}:=\left\{\xi \in S_{u,p,l,L}, \ 2^{j_1}\leq \langle \xi-\xi_1\rangle <2^{j_1+1}, \ ..., \ 2^{j_a}\leq \langle \xi-\xi_a\rangle <2^{j_a+1}\right\}.
$$
If $\min (2^{j_1},...,2^{j_a})\lesssim \max(1,\ell)$ then $E_{j_1,...,j_k}$ is contained in a ball of size $\lesssim 2^{\min(j_1,...,j_a)}$ so that:
\bee
\sum_{j_1,...,j_a, \ \min (2^{j_1},...,2^{j_a})\lesssim \max(1,\ell)} \frac{1}{\langle \xi-\xi_1\rangle}...\frac{1}{\langle \xi-\xi_a\rangle}&\lesssim &\sum_{j_1,...,j_a, \ \min (2^{j_1},...,2^{j_a})\lesssim \max(1,\ell) }2^{-j_1-...-j_a+d\min(j_1,...,j_a)}\\
&\lesssim & (\max{1, \ell})^{d-a}
\eee
If $\min (2^{j_1},...,2^{j_a})\gtrsim \max(1,\ell)$ then $E_{j_1,...,j_k}$ is contained in a strip of size $(\ell,2^{\min(j_1,...,j_a)})$ so that:
\bee
&&\sum_{j_1,...,j_a, \ \min (2^{j_1},...,2^{j_a})\gtrsim \max(1,\ell)} \frac{1}{\langle \xi-\xi_1\rangle}...\frac{1}{\langle \xi-\xi_a\rangle}\\
&\lesssim &\sum_{j_1,...,j_a, \ \min (2^{j_1},...,2^{j_a})\gtrsim \max(1,\ell) }2^{-j_1-...-j_a+(d-1)\min(j_1,...,j_a)} \max(1,\ell)\\
&\lesssim & \left\{ \begin{array}{l l} (\max{1, \ell}) L^{d-1-a} \max(1,\ell) \qquad \mbox{if }a<d-1,\\ \langle \log L \rangle (\max{1, \ell})\qquad \mbox{if }a=d-1.\end{array} \right.
\eee
The two above estimates imply \fref{bd:weightedstrip}. The proof of the other estimate \fref{bd:weightedannulus} is analogous and we omit it. Notice that since the estimates only needs to be proved for $r\geq 1$ there is no difference between the estimates for $H=\Id$ and other symmetric matrices close to $\Id$.

\end{proof}

We then employ the weighted estimates of the previous Lemma to find weighted resolvent bounds.

\begin{lemma}[Weighted estimates for sums at vertices] \label{lem:weightedvertices}

Let $\ep^{1-}\leq t\leq 1$, $|\zeta|,|\eta|\lesssim \epsilon^{-1}$, $\alpha \in \mathbb R$, and $H$ a symmetric matrix close to the identity. Then for any $a=0,...,d-1$ and $\xi_1,...,\xi_a\in \mathbb Z^d$:
\be \label{bd:degre1lineaire}
\sum_{|\xi|\lesssim \ep^{-1}}  \frac{1}{\langle \xi-\xi_1\rangle}... \frac{1}{\langle \xi-\xi_a\rangle} \frac{|1-\delta(\zeta)-\delta(\xi+\eta)|}{|2H\zeta \cdot \xi+\alpha+\frac{i}{t}|} \lesssim \langle \log \ep \rangle^2 \ep^{1-d+a}\max \left(\frac{1}{\langle \zeta \rangle},t \right)
\ee
\be \label{bd:degre1quadratic}
\sum_{|\xi|\lesssim \ep^{-1}}   \frac{1}{\langle \xi-\xi_1\rangle}... \frac{1}{\langle \xi-\xi_a\rangle}  \frac{1}{\left|2\left|\xi+\frac{\zeta}{2}\right|^2_H+\alpha+\frac{i}{t}\right|} \lesssim \left\{ \begin{array}{l l}   \langle \log \ep \rangle^2 t \ep^{1-d+a}\quad \mbox{for }0\leq a\leq d-2,\\  \langle \log \ep \rangle^2 \sqrt{t} \quad \mbox{for }a=d-1. \end{array} \right.
\ee
and for $\sigma,\sigma'\in \{-1,1\}$ with $(\sigma,\sigma')\neq (1,1)$ and $(\iota_{i},\iota_i')\in \{-1,0,1\}\times \{-1,0,1\}\backslash\{0,0\}$ for $i=1,...,a$:
\bea 
\nonumber &&\sum_{|\xi|,|\xi'|\lesssim \ep^{-1}}   \frac{1}{\langle \iota_1 \xi+\iota_1'\xi'-\xi_1\rangle}... \frac{1}{\langle \iota_a\xi+\iota_a'\xi'-\xi_a\rangle}  \frac{1}{||\xi|^2_H+\sigma'|\xi'|^2_H+\sigma |\xi+\xi'-\tilde \xi|^2_H+\alpha+\frac{i}{t}|} \\
\label{bd:degre2} &\lesssim & \left\{ \begin{array}{l l} \ep^{2-2d-\kappa}\quad \mbox{for }a=0,\\  \langle \log \epsilon \rangle^2 t \ep^{1-2d+a} \quad \mbox{for }1\leq a\leq d-1. \end{array} \right.
\eea

\end{lemma}

\begin{corollary}[Weighted estimates for sums at vertices] \label{cor:weightedvertices}

Keeping the notations of the previous lemma:
\be \label{bd:degre1lineaire2}
\sum_{|\xi|\lesssim \ep^{-1}}  \max\left(\frac{1}{\langle \xi-\xi_1\rangle},t\right)... \max\left(\frac{1}{\langle \xi-\xi_a\rangle},t\right) \frac{|1-\delta(\zeta)-\delta(\xi+\eta)|}{|2H \zeta \cdot \xi+\alpha+\frac{i}{t}|} \lesssim t^a \ep^{1-d-\kappa}\max \left(\frac{1}{|\zeta|},t \right)
\ee
\be \label{bd:degre1quadratic2}
\sum_{|\xi|\lesssim \ep^{-1}} \max\left(\frac{1}{\langle \xi-\xi_1\rangle},t\right)... \max\left(\frac{1}{\langle \xi-\xi_a\rangle},t\right)  \frac{1}{\left|2\left|\xi+\frac{\zeta}{2}\right|^2_H+\alpha+\frac{i}{t}\right|}  \lesssim \left\{ \begin{array}{l l}  t^{1+a} \ep^{1-d-\kappa}\quad \mbox{for } 0\leq a\leq d-2,\\  \sqrt{t}\ep^{-\kappa}+t^{d}\ep^{1-d-\kappa} \quad \mbox{for }a=d-1. \end{array} \right.
\ee
\bea 
\nonumber &&\sum_{|\xi|,|\xi'|\lesssim \ep^{-1}}   \max\left(\frac{1}{\langle \iota_1 \xi+\iota_1'\xi'-\xi_1\rangle},t\right)... \max \left(\frac{1}{\langle \iota_a\xi+\iota_a'\xi'-\xi_a\rangle},t\right)  \frac{1}{||\xi|^2_H+\sigma'|\xi'|^2_H+\sigma |\xi+\xi'-\tilde \xi|^2_H+\alpha+\frac{i}{t}|} \\
\label{bd:degre22} &\lesssim & t^{1+a} \ep^{1-2d-\kappa}
\eea

\end{corollary}

\begin{proof}

\textbf{Step 1} \emph{Proof of \fref{bd:degre1lineaire}}. First, if $\zeta=0$, then the numerator $|1-\delta(\zeta)-\delta(\xi+\eta)|=\delta(\xi+\eta)$ forces $\xi=-\eta$ and the sum contains only one nonzero term. As $|2H\zeta \cdot \xi+\alpha+\frac{i}{t}|\geq t^{-1}$ the quantity to be bounded is then $\leq t$ which proves \fref{bd:degre1lineaire}.

We now assume $\zeta\neq 0$ and write $\tilde \zeta=2H\zeta$. Let $S=\{|\xi|\leq \epsilon^{-1}, \ |\tilde \zeta \cdot \xi+\alpha|\leq |\tilde \zeta|\}$. Then $S$ is a strip of size $(1,\epsilon^{-1})$ hence using \fref{bd:weightedstrip} and $|\tilde \zeta \cdot \xi+\alpha+\frac{i}{t}|\geq t^{-1}$:
$$
\sum_{S} \frac{1}{\langle \xi-\xi_1\rangle}... \frac{1}{\langle \xi-\xi_a\rangle} \frac{1}{|\tilde  \xi \cdot \xi+\alpha+\frac{i}{t}|} \lesssim \left\{\begin{array}{l l} \epsilon^{1-d+a} t & \mbox{for } a<d-1,\\  \langle \log \epsilon \rangle t & \mbox{for } a=d-1.\end{array} \right.
$$
Let then $T=\{|\xi|\leq \ep^{-1}, \ |\tilde  \xi \cdot \xi+\alpha|\geq |\tilde \zeta|\}$. We decompose in dyadic strips for $k\geq 0$: $T_k:=\{|\xi|\leq \ep^{-1}, \ 2^k|\tilde \zeta|\leq |\tilde  \xi \cdot \xi+\alpha|\leq 2^{k+1}|\tilde \zeta|\}$, which are of size $(2^k, \ep^{-1})$ with $2^k\lesssim \ep^{-1}$, so that using \fref{bd:weightedstrip}:
\bee
&& \sum_{T}\frac{1}{\langle \xi-\xi_1\rangle}... \frac{1}{\langle \xi-\xi_a\rangle} \frac{1}{|\tilde  \xi \cdot \xi+\alpha+\frac{i}{t}|} \lesssim \sum_{k}\sum_{T_k} \frac{1}{\langle \xi-\xi_1\rangle}... \frac{1}{\langle \xi-\xi_a\rangle} \frac{1}{2^k|\tilde \zeta|} \\
&  \lesssim & \left\{\begin{array}{l l} \frac{ \langle \log \ep \rangle \ep^{1-d+a}}{| \zeta|} & \mbox{for } a<d-1,\\ \frac{ \langle \log \ep \rangle^2}{| \zeta|} & \mbox{for } a=d-1.\end{array} \right.
\eee
The above estimates on $S$ and $T$ yield the desired result \fref{bd:degre1lineaire}.\\

\noindent
\textbf{Step 3} \emph{Proof of \fref{bd:degre1quadratic}} Up to replacing the sum over $\mathbb Z^d$ by one over $\mathbb Z^d/2$ and to changing variables it suffices to prove the desired bound for:
$$
\sum_{|\xi|\lesssim \ep^{-1}}   \frac{1}{\langle \xi-\xi_1\rangle}... \frac{1}{\langle \xi-\xi_a\rangle}  \frac{1}{\left|\left|\xi \right|^2_H+\alpha+\frac{i}{t}\right|}
$$
 We start with the case $\alpha \geq -Ct^{-1}$, where $C$ is an arbitrary constant. Then the set $A=\{||\xi|^2_H+\alpha |\leq t^{-1}\}$ is contained in a ball of radius $\lesssim t^{-1/2}$. We decompose the remaining part $B=\{||\xi|^2_H+\alpha |\geq t^{-1}\}$ into sets $B_k=\{2^kt^{-1}\leq ||\xi|^2_H+\alpha |\leq 2^{k+1}t^{-1}$ for $1\leq 2^k\lesssim \ep^{-2}t$. $B_k$ is contained in a ball of radius $\lesssim 2^{k/2}t^{-1/2}$. Hence, using \fref{bd:weightedannulus} and $||\xi|^2_H+\alpha+\frac{i}{t}|\geq t^{-1}$ on $A$:
\bee
&& \sum_{|\xi|\leq \ep^{-1}}\frac{1}{\langle \xi-\xi_1\rangle}... \frac{1}{\langle \xi-\xi_a\rangle} \frac{1}{||\xi|^2_H+\alpha+\frac{i}{t}|} \leq \sum_A t \frac{1}{\langle \xi-\xi_1\rangle}... \frac{1}{\langle \xi-\xi_a\rangle} +\sum_k \sum_{B_k} 2^{-k}t \frac{1}{\langle \xi-\xi_1\rangle}... \frac{1}{\langle \xi-\xi_a\rangle} \\\\
&\lesssim& t (t^{-\frac 12})^{d-a} +\sum_k 2^{-k}t \left(2^{\frac k2}t^{-\frac 12}\right)^{d-a}  \lesssim  \left\{\begin{array}{l l} \ep^{2-d+a} &  \mbox{for } a\leq d-3,\\ \langle \log \epsilon \rangle & \mbox{for } a=d-2,\\ \sqrt{t} & \mbox{for } a=d-1 \end{array} \right.
\eee
where we used that $\ep^{1-}\leq t $. This suffices for \fref{bd:degre1quadratic}. 

If $\alpha \leq -2\ep^{-2}$ then using the rough bound $||\xi|^2_H+\alpha+\frac{i}{t}|\geq \ep^{-2}$ as $|\xi|\leq \ep^{-1}$ one obtains:
$$
 \sum_{|\xi|\leq \ep^{-1}}\frac{1}{\langle \xi-\xi_1\rangle}... \frac{1}{\langle \xi-\xi_a\rangle} \frac{1}{||\xi|^2_H+\alpha+\frac{i}{t}|}\lesssim \ep^2 \sum_{|\xi|\leq \ep^{-1}}\frac{1}{\langle \xi-\xi_1\rangle}... \frac{1}{\langle \xi-\xi_a\rangle}\lesssim \ep^{2-d+a}
$$
and \fref{bd:degre1quadratic} is proved in this case too. 

Hence, matters reduce to the case $-2\ep^{-2}\leq \alpha \leq -C/t$ for some universal $C\gg 1$. Let us consider the set $C=\{ ||\xi|^2_H+\alpha |\geq |\alpha|\}$, and decompose $C=\cup_k C_k=\{ 2^k|\alpha|\leq ||\xi|^2_H+\alpha |\leq 2^{k+1} |\alpha|\}$ for $1\leq 2^k\lesssim \ep^{-2}|\alpha|^{-1}$. Then $C_k$ is contained in a ball of size $\lesssim \sqrt{2^k|\alpha|}$ so that using \fref{bd:weightedannulus} and $\sqrt{t}\gtrsim \sqrt{|\alpha|}^{-1}$:
$$
\sum_{C} \frac{1}{\langle \xi-\xi_1\rangle}... \frac{1}{\langle \xi-\xi_a\rangle} \frac{1}{||\xi|^2_H+\alpha+\frac{i}{t}|} \lesssim \sum_k \sum_{C_k} \frac{2^{-k} |\alpha|^{-1}}{\langle \xi-\xi_1\rangle...\langle \xi-\xi_a\rangle} \lesssim  \left\{\begin{array}{l l}\ep^{2-d+a} & \mbox{for } a\leq d-3,\\ \langle \log \ep \rangle & \mbox{for } a=d-2, \\ \sqrt t.\end{array} \right.
$$
where we used $|\alpha|^{-\frac 12}\lesssim \sqrt t$ for the last case. We now consider the set $D=\{ t^{-1} \leq ||\xi|^2_H+\alpha |\leq |\alpha|\}$ and decompose $D_k=\{ 2^kt^{-1}\leq ||\xi|^2_H+\alpha |\leq 2^{k+1}t^{-1}$ for $1\leq 2^{k}\lesssim t|\alpha|$. Then $D_k$ is contained in an annulus of size $\left(\frac{2^kt^{-1}}{\sqrt{|\alpha|}},\sqrt{|\alpha|}\right)$, so that using \fref{bd:weightedannulus}, $|t|^{-1}\lesssim |\alpha|\lesssim \ep^{-2}$ and $\ep^{1-}\leq t$:
\bee
&& \sum_{D} \frac{1}{\langle \xi-\xi_1\rangle}... \frac{1}{\langle \xi-\xi_a\rangle}  \frac{1}{||\xi|^2_H+\alpha+\frac{i}{t}|} \lesssim \sum_k \sum_{D_k} \frac{2^{-k}t}{\langle \xi-\xi_1\rangle...\langle \xi-\xi_a\rangle}\\
&\lesssim & \langle \log \ep \rangle \sum_{\frac{2^kt^{-1}}{\sqrt{|\alpha|}}\geq 1} \sqrt{|\alpha|}^{d-2-a}+ \langle \log \ep \rangle \sum_{\frac{2^kt^{-1}}{\sqrt{|\alpha|}}\leq 1} 2^{-k}t \sqrt{|\alpha|}^{d-1-a} \\
&\lesssim &  \left\{\begin{array}{l l} \langle \log \ep \rangle^2 t\ep^{1-d+a}, \quad \mbox{for } a\leq d-2, \\ \langle \log \ep \rangle^2 \sqrt t \quad \mbox{for } a=d-1.\end{array} \right.
\eee
Finally, we consider the set $E=\{  ||\xi|^2_H+\alpha |\leq t^{-1} \}$, which is contained in an annulus of size $(\frac{t^{-1}}{\sqrt{|\alpha|}},|\sqrt \alpha|)$. Hence using  \fref{bd:weightedannulus}, $t^{-1}\lesssim |\alpha|\lesssim \ep^{-2} $ and $\ep^{1-}\leq t$:
\bee
\sum_{E}  \frac{1}{\langle \xi-\xi_1\rangle}... \frac{1}{\langle \xi-\xi_a\rangle}  \frac{1}{||\xi|^2_H+\alpha+\frac{i}{t}|} &\lesssim& \langle \log \epsilon \rangle \left\{\begin{array}{l l} \sqrt{|\alpha|}^{d-2-a}, \quad \mbox{for } \frac{t^{-1}}{\sqrt{|\alpha|}}\geq 1, \\ t \sqrt{|\alpha|}^{d-1-a}, \quad \mbox{for }\frac{t^{-1}}{\sqrt{|\alpha|}}\leq 1,\end{array} \right. \\
&\lesssim & \left\{\begin{array}{l l} \langle \log \ep \rangle t\ep^{1-d+a}, \quad \mbox{for } a\leq d-2, \\ \langle \log \ep \rangle \sqrt t \quad \mbox{for } a=d-1.\end{array} \right.
\eee
The collection of the above estimates proves \fref{bd:degre1quadratic}.

\textbf{Step 4} \emph{Proof of \fref{bd:degre2}}. For $a=0$ this is Lemma B.1 in \cite{CG}, so we consider the case $1\leq a\leq d-1$. By performing a change of variables, without loss of generality we can restrict ourselves to the case $\sigma=-1$, $\sigma'=+1$ and $\alpha=1$ in which case $|\xi|^2-|\xi'|^2+|\xi+\xi'-\tilde \xi|^2=2(\xi+\xi').(\xi-\tilde \xi)+|\tilde \xi|^2$. We write $\eta=\xi+\xi'$ and $\eta'=\xi-\xi'$. Then we split between weights involving or not $\eta'$ the following way. There exist $j\in \{0,1,...,a\}$, constants $(c_i)_{1\leq i \leq j}\in \{\pm1\}^j$, constants $(c_i,c_i')_{j+1\leq i \leq k}$ with $|c_i'|\geq \frac 12$ for $i=j+1,...,a$, such that:
$$
\iota_i\xi+\iota_i'\xi'-\xi_1=c_i \eta-\xi_i \mbox{ for }1\leq i \leq j, \quad \mbox{and}\quad \iota_i\xi+\iota_i'\xi'-\xi_1=c_i \eta+c_i'\eta'-\xi_i \mbox{ for }j+1\leq i \leq k.
$$
The desired bound is then obtained using the bounds proved in Step 2, and that $\ep^{1-}\leq t$:
\bee
&&\sum_{|\eta|\leq \ep^{-1}} \sum_{|\eta'|\leq \ep^{-1}}  \frac{1}{\langle c_1  \eta-\xi_1\rangle}... \frac{1}{\langle c_j  \eta-\xi_j\rangle} \frac{1}{\langle c_{j+1}\eta+c_{j+1}'\eta'-\xi_{j+1}\rangle}... \frac{1}{\langle c_{a}\eta+c_{a}'\eta'-\xi_{a}\rangle}  \frac{1}{|2\eta.\eta'+|\tilde \xi|^2+\alpha+\frac{i}{t}|}\\
&\lesssim&\sum_{|\eta|\leq \ep^{-1}}  \frac{1}{\langle c_1  \eta-\xi_1\rangle}... \frac{1}{\langle c_j  \eta-\xi_j\rangle}  \langle \log \ep \rangle^{2} \ep^{1-d+a-j}\max \left(\frac{1}{|\eta|},t \right) \lesssim  \langle \log \epsilon \rangle^2 t\ep^{1-2d+a}
\eee

\end{proof}

The second part of this section deals with the Laplacian dispersion relation $H=\Id$. It aims at understanding the number of configurations of wave numbers realising some orthogonality conditions. Given any $\xi=(\xi_1,...,\xi_d) \in \mathbb Z^d\backslash \{0\}$, we write:
$$
\xi_v=\frac{\xi}{\text{gcd}(\xi_1,...,\xi_d)}
$$
where $\text{gcd}(\xi_1,...,\xi_d)$ denotes the greatest common divisor of $\xi_1$,...,$\xi_d$. Wave numbers $\xi$ such that $\xi=\xi_v$ are said to be visible from the origin.

\begin{lemma}[Degree one and two, resonant case] \label{lem:comptagecompletementresonnant}

For any $\alpha \in \mathbb R$, $\kappa>0$, for any $\tilde \xi,\eta \in \mathbb Z^d$ with $|\tilde \xi|,|\eta|\leq \ep^{-1}$ and $C\geq 1$ one has:
\be \label{bd:degree1quadraticresonant}
\# \left\{ \xi \in \mathbb Z^d, \quad |\xi|\leq \epsilon^{-1} \mbox{ and } \left| \alpha -|\xi|^2-|\tilde \xi+\xi|^2\right|<\frac 12 \right\} \lesssim \left\{ \begin{array}{l l} \epsilon^{2-d-\kappa} \qquad \text{if } d=2,3,\\ \epsilon^{2-d}\qquad \text{if } d\geq 4, \end{array} \right.
\ee
and if $\tilde \xi \neq 0$:
\be \label{bd:degree1linearresonant}
\# \left\{ \xi \in \mathbb Z^d, \quad |\xi|\leq \epsilon^{-1}, \quad \left| \alpha +|\xi|^2-|\tilde \xi+\xi|^2\right|<\frac 12 \right\}\lesssim \frac{\epsilon^{1-d}}{|\tilde \xi_v|},
\ee
One has in addition for any $\sigma \in \{\pm 1\}$:
\be \label{bd:degree2resonant}
\{ \xi,\xi'\in \mathbb Z^2, \quad |\xi|,|\xi'|\leq \epsilon^{-1} \mbox{ and } \left| \alpha+|\xi|^2+\sigma |\xi'|^2-|\tilde \xi-\xi-\xi'|^2 \right|<\frac 12 \}   \lesssim \left\{ \begin{array}{l l} \epsilon^{-2}\langle \log \epsilon \rangle \qquad \text{if } d=2,\\ \epsilon^{2-2d}\qquad \text{if } d\geq 3, \end{array} \right.
\ee
and that there exists $c=c(d)>0$ such that:
\be \label{bd:degree2resonantlower}
\{ \xi,\xi'\in \mathbb Z^d, \quad |\xi|,|\xi'|\leq \epsilon^{-1} \mbox{ and } \xi \cdot \xi'=0 \}  \geq  \left\{ \begin{array}{l l}c \epsilon^{-2}\langle \log \epsilon \rangle \qquad \text{if } d=2,\\ c\epsilon^{2-2d}\qquad \text{if } d\geq 3. \end{array} \right.
\ee

\end{lemma}

\begin{proof}

\textbf{Step 1} \emph{Proof of the first estimate}. Notice first that for $\alpha \gg \epsilon^{-2}$, as $|\xi|,|\tilde \xi|\lesssim \epsilon^{-1}$, the set is empty and the bound is true. Assuming now that $|\alpha|\lesssim \epsilon^{-2}$, writing $\xi'=\xi+\frac{\tilde \xi}{2}$ and $\alpha'=\alpha+\frac{|\tilde \xi|^2}{4}$ there holds:
$$
 \left| \alpha -|\xi|^2-|\tilde \xi+\xi|^2\right|=\left|\alpha'+2|\xi'|^2\right|<\frac 12 \quad \Leftrightarrow |\xi'|^2\in \left( \frac{\alpha'}{2}-\frac 14,\frac{\alpha'}{2}+\frac 14 \right).
$$
As $\left|\frac{\alpha'}{2}\right|\lesssim \epsilon^{-2}$, the above set is a sphere of radius $\lesssim \ep^{-1}$. We recall that given $n\in \mathbb N$, the number $r_d(n)=\#\{\xi\in \mathbb Z^d, \ |\xi|^2=n\}$ of ways to represent $n$ as the sum of $d$ squares is $r_d(n)\lesssim n^{\frac d2-1+\kappa}$ for any $\kappa>0$ if $d=2,3$, and is $r_d(n)\lesssim n^{\frac d2-1}$ if $d\geq 4$ (see for example the textbook \cite{Gr}). This is the desired bound.

\textbf{Step 2} \emph{Proof of the second estimate}. We recall known results on lattices. For a given $\tilde \xi\in \mathbb Z^d\backslash \{0\}$, the condition $\tilde \xi.\xi=0$ is equivalent to $\tilde \xi_v.\xi=0$. The set $E$ of integer points satisfying $\tilde \xi_v.\xi=0$ is a $d-1$ dimensional lattice. Let $(u_1,...,u_{d-1})\in \mathbb Z^{d-1}$ be any basis of $E$, with $|u_i|\lesssim |\tilde \xi_v|$ for $i=1,...,d-1$. The set $P=\{x_1u_1+...+x_{d-1}u_{d-1}, \ x\in \mathbb R^{d-1}, \ 0\leq x_i \leq 1 \mbox{ for }i=1,...,d-1\}$ is its fundamental parallelepiped. The $d-1$ measure $|P|$ of $P$ is independent of the choice of the basis. The density of $E$ is $|P|^{-1}$, so that $\# \{ \xi \in E, \ |\xi|\leq \epsilon^{-1}\}\approx \frac{\ep^{1-d}}{|P|}$.

We now prove the standard result that $|P|=|\tilde \xi_v|$ for completeness. Indeed, there exists $u_d\in \mathbb Z^d$ with $u_d.\tilde \xi_v=1$ since $\tilde \xi_v$ is visible. Since $\tilde \xi_v$ spans $E^\perp$, the volume of the fundamental parallelepiped of $(u_1,...,u_{d-1},u_d)$ is $|P|\cdot \frac{|u_d.\tilde \xi_v|}{|\tilde \xi_v|}=\frac{|P|}{|\tilde \xi_v|}$. On the other hand, $(u_1,...,u_d)$ is a basis of $\mathbb Z^d$, so the volume of its fundamental parallelepiped is equal to $1$, that of the standard basis. Hence $|P|=|\tilde \xi_v|$ and
\be \label{bd:interortho}
\# \left\{ \xi \in \mathbb Z^d, \quad |\xi|\leq \epsilon^{-1} \mbox{ and } \xi \cdot \tilde \xi=0 \right\}\approx \frac{\epsilon^{1-d}}{|\tilde \xi_v|}.
\ee
We come back to the second bound of the Lemma. It is obvious as long as there is no point $\xi^0$ with $\left| \alpha +|\xi|^2-|\tilde \xi+\xi|^2\right|<\frac 12$. If there is such a point $\xi^0$, we decompose $\xi=\xi^0+\xi^1$ and get:
$$
\left| \alpha +|\xi|^2-|\tilde \xi+\xi|^2\right|=\left| \alpha +|\xi^0|^2-|\tilde \xi+\xi^0|^2+2\tilde \xi \cdot \xi^1\right|<\frac 12 \quad \Leftrightarrow \quad \tilde \xi \cdot \xi^1=0.
$$
Using the above bound we obtain the second estimate of the Lemma.

\textbf{Step 3} \emph{Proof of the third and fourth estimates}. We prove it via simple arithmetic relations and parametrisations, but more sophisticated number analytic techniques may provide a sharp constant for the inequalities. We only treat the case $\sigma=+1$ as the case $\sigma=-1$ is similar. First, we change variables $\eta=\xi+\tilde \xi$ and $\eta'=\tilde \xi+\xi'$ and write:
$$
\alpha+|\xi|^2+ |\xi'|^2-|\tilde \xi-\xi-\xi'|^2=\alpha+2\left(\xi+\tilde \xi) \right.\left(\tilde \xi+\xi'\right)=\alpha+2\eta.\eta'.
$$
Using the second inequality of the Lemma one finds:
$$
\# \left\{ \eta' \in \mathbb Z^d, \quad |\eta'|\leq C\epsilon^{-1} \mbox{ and } \left| \alpha +2 \eta.\eta'\right|<\frac 12 \right\}\lesssim \frac{\epsilon^{1-d}}{|\eta_v|}.
$$
We now parametrise all points $|\eta|\leq \epsilon^{-1}$ using a dyadic partition for their associated visible point $\eta_v$ and obtain:
\bee
\sum_{\eta \in \mathbb Z^d, \ 1\leq |\eta|\leq \epsilon^{-1}}  \frac{\epsilon^{1-d}}{|\eta_v|} \lesssim \sum_{j=1}^{\lfloor \log_2 \epsilon^{-1}\rfloor} \sum_{2^j\leq |\eta_v|< 2^{j+1}, \ \eta_v \mbox{ visible}} \frac{\epsilon^{-1}}{2^j}\cdot \frac{\epsilon^{1-d}}{2^j} .
\eee
Since visible points have a positive density in $\mathbb Z^d$,
\bee
\sum_{\eta \in \mathbb Z^d, \ 1\leq |\eta|\leq \epsilon^{-1}}  \frac{\epsilon^{1-d}}{|\eta_v|} \lesssim \sum_{j=1}^{\lfloor \log_2 \epsilon^{-1}\rfloor} \sum_{2^j\leq |\eta_v|< 2^{j+1}, \ \eta_v \mbox{ visible}} 2^{j(d-2)}\ep^{-d} \lesssim  \left\{ \begin{array}{l l} \epsilon^{-2} \langle \log \epsilon\rangle \qquad \mbox{for }d=2, \\  \epsilon^{2-2d} \qquad \mbox{for }d=2\end{array} \right.
\eee
The same reasoning, but this time using the lower bound provided by \fref{bd:interortho}, shows the last bound of the Lemma.

\end{proof}

The next Lemmas build on the previous one to obtain bounds on sums of terms of the form $\frac{1}{|\alpha-\Omega+\frac{i}{t}|}$ at a degree one vertex and then at a degree two vertex, for $t\gg 1$ and for non-resonant configurations.

\begin{lemma}[Degree one, nonresonant case] 

Consider any dimension. Pick any $|\tilde \xi|,|\eta|\leq \epsilon^{-1}$. Then for all $\alpha \in \mathbb R$ and $\kappa>0$:
\be \label{bd:resolvantnonresonantdeg11}
\sum_{|\xi|\leq \epsilon^{-1}, \ | \alpha -|\xi|^2-|\tilde \xi+\xi|^2|\geq \frac 12} \frac{1}{ \left| \alpha -|\xi|^2-|\tilde \xi+\xi|^2\right|}\lesssim \epsilon^{2-d-\kappa},
\ee
and if $\tilde \xi \neq 0$, writing $\tilde \xi=d\tilde \xi_v$ where $\tilde \xi_v$ is a visible point:
\be \label{bd:resolvantnonresonantdeg12}
\sum_{|\xi|\leq \epsilon^{-1}, \ | \alpha -|\xi|^2+|\tilde \xi+\xi|^2|\geq \frac 12} \frac{1}{ \left| \alpha -|\xi|^2+|\tilde \xi+\xi|^2\right|}\lesssim \frac{\epsilon^{1-d}}{\langle \tilde \xi_v\rangle}\langle \log \epsilon \rangle,
\ee
and for any $\sigma \in \{-1,1\}$:
\be \label{bd:resolvantnonresonantdeg2}
\sum_{|\xi|,|\xi'|\leq \epsilon^{-1}, \ |\alpha+|\xi|^2+\sigma |\xi'|^2-|\tilde \xi-\xi-\xi'|^2|\geq \frac 12} \frac{1}{\left|\alpha+|\xi|^2+\sigma |\xi'|^2-|\tilde \xi-\xi-\xi'|^2\right|}\lesssim \left\{ \begin{array}{l l}  \epsilon^{-2} \langle \log \epsilon \rangle^2 \qquad \mbox{for }d=2, \\ \epsilon^{2-2d} \langle \log \epsilon \rangle \qquad \mbox{for }d\geq 3 . \end{array} \right.
\ee

\end{lemma}

\begin{proof}

For the first bound, as $|\xi|,|\tilde \xi|\lesssim \ep^{-1}$ there exists $K>0$ such that we can estimate:
$$
\sum_{|\xi|\leq \epsilon^{-1}, \ | \alpha -|\xi|^2-|\tilde \xi+\xi|^2|\geq \frac 12} \frac{1}{ \left| \alpha -|\xi|^2-|\tilde \xi+\xi|^2\right|}\lesssim \sum_{0\leq n \leq K\ep^{-2}} \frac{\# \{ |\xi|\lesssim \ep^{-1}, \ |\xi|^2+|\tilde \xi+\xi|^2=n \} }{ \left\langle \alpha -n\right\rangle} \lesssim \ep^{2-d-\kappa},
$$
where we used the first inequality of Lemma \ref{lem:comptagecompletementresonnant}. For the second bound, if $\tilde \xi=0$ then the numerator $|1-\delta (\tilde \xi)-\delta (\xi-\eta)|=\delta (\xi-\eta)$ forces the term to have only the $\xi=\eta$ term being nonzero, and the estimate \fref{bd:resolvantnonresonantdeg12} holds. If $\tilde \xi\neq 0$ then the desired estimate is obtained the same way from the second inequality of Lemma \ref{lem:comptagecompletementresonnant}:
$$
\sum_{|\xi|\leq \epsilon^{-1}, \ | \alpha -|\xi|^2+|\tilde \xi+\xi|^2|\geq \frac 12} \frac{1}{ \left| \alpha -|\xi|^2+|\tilde \xi+\xi|^2\right|}\lesssim \sum_{0\leq n \leq K\ep^{-2}} \frac{\# \{ |\xi|\lesssim \ep^{-1}, \ |\xi|^2-|\tilde \xi+\xi|^2 =n \} }{ \left\langle \alpha -n\right\rangle} \lesssim  \frac{\epsilon^{-1}}{|\tilde \xi_v|}\langle \log \epsilon \rangle.
$$
Finally, the third bound is obtained the very same way from the third inequality of Lemma \ref{lem:comptagecompletementresonnant}.

\end{proof}

\begin{lemma} \label{lem:tech3}
For $t\geq 1$ and $j\geq 1$:
$$
\int_{-1}^1 \frac{d\alpha}{|\alpha+\frac{i}{t}|^j}\lesssim \left\{ \begin{array}{l l} \langle \log t \rangle \quad \mbox{for }j=1, \\ t^{j-1} \quad \mbox{for }j\geq 2, \end{array} \right.
$$
\end{lemma}

\begin{proof}
We omit the proof of this basic integral estimate.
\end{proof}

\bibliographystyle{amsplain}
\bibliography{bibeqonrv}

\end{document}